\documentclass[10pt, oneside, reqno]{article}

	\usepackage{algpseudocode}
	\usepackage{amsmath}
	\usepackage{amssymb}
	\usepackage{amsthm}
	\usepackage[USenglish]{babel}
	\usepackage{bm}
	\usepackage{eqparbox} 
	\usepackage{enumitem}
	\usepackage[T1]{fontenc}
	\usepackage[top=0.8in,bottom=0.8in,left=1in,right=1in]{geometry}
	\usepackage{graphicx}
		\graphicspath{{Pics/}}
	\usepackage{makecell}
	\usepackage{mathtools}
	\usepackage{nicefrac}
	\usepackage{pgffor}
	\usepackage{subcaption}
	\usepackage[dvipsnames]{xcolor}
	\usepackage{hyperref}
	\usepackage{algorithm}
	\usepackage[capitalize,nameinlink]{cleveref}

	\counterwithin{figure}{subsection}
	\numberwithin{equation}{section}

	\SetEnumitemKey{widestL}{
		labelwidth=\eqboxwidth{listlabel@\EnumitemId},
		leftmargin=!,
		format=\mylistlabel,
	}
	\newcommand\mylistlabel[2][l]{\eqmakebox[listlabel@\EnumitemId][#1]{#2}}

	\setlist[enumerate,1]{
		label=\text{\upshape{(\roman*)}},
	}
	\setlist[itemize,1]{
		label=\textbullet,
	}
	\setlist[itemize,2]{
		wide,
		label=\raisebox{0.3ex}{\tiny$\bullet$},
		leftmargin = 0pt,
		labelindent = 0pt,
	}
	\setlist{
		widestL,
		topsep=1ex,
		itemsep=0pt,
	}

	\hypersetup{%
		breaklinks,
		hypertexnames = true,
		plainpages    = false,
		colorlinks = true,
	}
	\makeatletter
		\newcommand{\theHALG@line}{\thealgorithm.\arabic{ALG@line}}
	\makeatother
	\NewDocumentCommand{\MFCQ}{t+}{%
		\IfBooleanTF{#1}{%
			\hyperref[definition_MFCQ+]{MPCC-MFCQ$^+$}%
		}{%
			\hyperref[definition_MFCQ]{MPCC-MFCQ}%
		}%
	}

	\NewDocumentCommand{\BCQ}{t+}{%
		\IfBooleanTF{#1}{%
			\hyperref[definition_BCQ+]{BCQ$^+$}%
		}{%
			\hyperref[definition_BCQ]{BCQ}%
		}%
	}

	\newcommand{\BCCQ}{\hyperref[definition_BCCQ]{BCCQ}}

\theoremstyle{plain}
	\newtheorem{theorem}{Theorem}[section]
		\Crefname{theorem}{Theorem}{Theorems}
	\newtheorem{assumption}[theorem]{Assumption}
		\Crefname{assumption}{Assumption}{Assumptions}
	\newtheorem{corollary}[theorem]{Corollary}
		\Crefname{corollary}{Corollary}{Corollaries}
	\newtheorem{definition}[theorem]{Definition}
		\Crefname{definition}{Definition}{Definitions}
	
		\Crefname{fact}{Fact}{Facts}
	\newtheorem{lemma}[theorem]{Lemma}
		\Crefname{lemma}{Lemma}{Lemmas}
	\newtheorem{proposition}[theorem]{Proposition}
		\Crefname{proposition}{Proposition}{Propositions}

\theoremstyle{definition}
	\newtheorem{example}[theorem]{Example}
		\Crefname{example}{Example}{Examples}
	\newtheorem{remark}[theorem]{Remark}
		\Crefname{remark}{Remark}{Remarks}

	\newcommand{\createthmlists}[1]{%
		\newlist{#1enumerate}{enumerate}{1}%
		\setlist[#1enumerate,1]{%
			label = \upshape{(\roman*)},%
			ref = \text{\thetheorem(\roman*)},%
		}%
		\crefalias{#1enumeratei}{#1}%
		\newlist{#1enumerateq}{enumerate}{1}%
		\setlist[#1enumerateq,1]{%
			label = \upshape{(\alph*)},%
			ref = \text{\thetheorem(\alph*)},%
		}%
		\crefalias{#1enumerateqi}{#1}%
		\AtBeginEnvironment{#1}{%
			\expandafter\let\expandafter\enumerate\csname #1enumerate\endcsname
			\expandafter\let\expandafter\endenumerate\csname end#1enumerate\endcsname
			\expandafter\let\expandafter\enumerateq\csname #1enumerateq\endcsname
			\expandafter\let\expandafter\endenumerateq\csname end#1enumerateq\endcsname
		}{}{}%
	}%

	\createthmlists{theorem}
	\createthmlists{assumption}
	\createthmlists{corollary}
	\createthmlists{definition}
	\createthmlists{example}
	\createthmlists{fact}
	\createthmlists{lemma}
	\createthmlists{proposition}
	\createthmlists{remark}

	\numberwithin{figure}{section}
	\numberwithin{table}{section}

	\DeclareSymbolFont{varmathbb}{U}{txmia}{m}{it}
	\DeclareMathSymbol{\mathbbcharN}{\mathord}{varmathbb}{142}
	\DeclareMathSymbol{\mathbbcharR}{\mathord}{varmathbb}{146}

	\newcommand{\N}{\mathbbcharN}
	\newcommand{\R}{\mathbbcharR}
	
	\newcommand{\y}{\bm{y}}
	\newcommand{\z}{\bm{z}}

	\DeclareMathOperator*{\argmin}{arg\,min}
	\DeclareMathOperator{\B}{B}
	\DeclareMathOperator{\cl}{cl}
	\DeclareMathOperator{\conv}{conv}
	\DeclareMathOperator{\dist}{dist}
	\DeclareMathOperator{\dom}{dom}
	\DeclareMathOperator{\epi}{epi}
	\DeclareMathOperator{\env}{env}
	\DeclareMathOperator{\graph}{gph}
	\DeclareMathOperator{\interior}{int}
	\DeclareMathOperator{\lev}{lev}
	\DeclareMathOperator*{\Lim}{Lim}
	\DeclareMathOperator*{\Liminf}{Lim\,inf}
	\DeclareMathOperator*{\Limsup}{Lim\,sup}
	\DeclareMathOperator*{\maximize}{maximize}
	\DeclareMathOperator*{\minimize}{minimize}
	\DeclareMathOperator{\prox}{prox}
	\DeclareMathOperator{\stt}{subject\ to}

	\NewDocumentCommand{\set}{s m o}{
		\IfBooleanTF{#1}{
			\left\{#2
				\IfValueT{#3}{
					\,\middle|\,#3
				}
			\right\}
		}{
			\{#2
				\IfValueT{#3}{
					\,|\,#3
				}
			\}
		}
	}

	\title{%
		A Lasry--Lions Envelope Approach for Mathematical Programs with Complementarity Constraints\thanks{%
			This work was supported by:
			Fond for Scientific Research Vlaanderen (FWO) projects G033822N, G081222N, G0A0920N;
			KU Leuven internal funding C14/24/103;
			Japan Society for the Promotion of Science KAKENHI grants JP21K17710 and JP24K20737.%
		}%
	}

	\author{%
		Jia Wang\thanks{%
			Department of Electrical Engineering (ESAT-STADIUS),
			KU Leuven, Kasteelpark Arenberg 10, 3001 Leuven, Belgium.
			{\sf\{jia.wang,panos.patrinos\}@esat.kuleuven.be}%
		}\and
		Andreas Themelis\thanks{%
			Faculty of Information Science and Electrical Engineering (ISEE),
			Kyushu University, 744 Motooka, Nishi-ku 819-0395, Fukuoka, Japan.
			{\sf andreas.themelis@ees.kyushu-u.ac.jp}%
		}\and
		Ivan Markovsky\thanks{%
			International Centre for Numerical Methods in Engineering (CIMNE),
			Gran Capità, 08034 Barcelona, Spain\\
			Catalan Institution for Research and Advanced Studies (ICREA),
			Pg. Lluis Companys 23, 08010 Barcelona, Spain.
			{\sf imarkovsky@cimne.upc.edu}%
		}\and
		Panagiotis Patrinos\footnotemark[2]
	}

	\date{}

\begin{document}

	\maketitle

	\begin{abstract}
		We propose a homotopy method for solving mathematical programs with complementarity constraints (CCs).
		The indicator function of the CCs is relaxed by the Lasry--Lions double envelope, an extension of the Moreau envelope that enjoys an additional smoothness property, making it amenable to fast optimization algorithms.
		The proposed algorithm mimics the behavior of homotopy methods for systems of nonlinear equations or penalty methods for constrained optimization: it solves a sequence of smooth subproblems that progressively approximate the original problem, using the solution of each subproblem as the starting point for the next one.
		In the limiting setting, we establish the convergence to Mordukhovich and Clarke stationary points.
		We also provide a worst-case complexity analysis for computing an approximate stationary point.
		Preliminary numerical results on a suite of benchmark problems demonstrate the effectiveness of the proposed approach.
	\end{abstract}



	\section{Introduction}
		We consider the following constrained mathematical program with complementarity constraints (MPCC):
		\begin{equation}\label{MPCC_simple}
		\begin{split}
		\minimize_{x\in C}
		\ \
		& f(x)
		\\
		\stt\ \
		& G(x)\geq0,\
		H(x)\geq0,\
		\langle G(x),H(x)\rangle =0,
		\end{split}
		\end{equation}
		where \(f:\R^{n}\to\R\) is a smooth function, and \(G,H:\R^{n}\to\R^{p}\) are smooth mappings.
		The set \(C\subseteq\R^{n}\) is nonempty, closed and convex.
		We focus on smooth objective and constraint functions, as all the essential difficulties are associated with the complementarity structure.
		Additionally, we assume that the solution set of problem~\eqref{MPCC_simple} is nonempty.
		
		MPCC is a class of important problems as they arise in many engineering and economics applications.
		Examples include
		multi-rigid-body contact problems in robotic control~\cite{Pang_1996},
		obstacle avoidance in autonomous driving~\cite{Hermans_2021},
		and policy modeling for price and subsidy regulation~\cite{Murphy_2016};
		see~\cite{Ferris_1997}  for a broader overview of applications.
		In addition, many practical applications are formulated as bilevel optimization problems, including Stackelberg games that represent hierarchical market structures~\cite{Chen_1995,Ye_1997}.
		In these problems, the feasible set is defined as the solution set of a parametric lower-level optimization problem.
		By replacing the lower-level problem with its Karush-Kuhn-Tucker (KKT) optimality conditions, the bilevel problem can be transformed into an MPCC{}.

		Although MPCCs have broad applications, they are challenging both theoretically and numerically: it is well known that standard constraint qualifications (CQs) such as the Mangasarian-Fromovitz constraint qualification (MFCQ) and the linear independence constraint qualification (LICQ) are violated at all feasible points of the complementarity constraints~\cite[Proposition 2.15]{Flegel_2005b}.
		As a result, the set of Lagrange multipliers at a stationary point may be unbounded~\cite{Izmailov_2012,Leyffer_2006}.
		Nevertheless, notable theoretical progress has been made, including the development of stationarity concepts such as strongly (S)~\cite{Scheel_2000}, Mordukhovich (M)~\cite{outrata_1999}, Clarke (C)~\cite{Scheel_2000}, alternative (A)~\cite{Flegel_2003,Flegel_2005c}, and weakly (W)~\cite{Scheel_2000} stationary points.
		Meanwhile, a variety of MPCC-tailored CQs have been developed, such as MPCC-LICQ~\cite[Definition 3(a)]{Kanzow_2015}, MPCC-MFCQ~\cite[Definition 3(b)]{Kanzow_2015}, MPCC-constant rank CQ (CRCQ)~\cite[Definition 3(c)]{Kanzow_2015}, MPCC-constant positive linear dependence CQ (CPLD)~\cite[Definition 3(d)]{Kanzow_2015}, MPCC-Abadie CQ~\cite[Definition 3.1]{Flegel_2005c}, MPCC-Guignard CQ~\cite[Definition 3.1]{Flegel_2005}, among others.
		A comprehensive overview of further MPCC-CQs can be found in~\cite{Ye_2005}, while an in-depth review of stationarity concepts and optimality conditions is provided in the PhD thesis~\cite{Flegel_2005b}.
		
		On the algorithmic side, many numerical methods have been developed to solve MPCCs.
		One of the most widely used approaches for handling complementarity constraints is the regularization method, which relaxes the constraint \(G_i(x)H_i(x)=0\) to \(G_i(x)H_i(x)\leq\tau\) for all \(i\in\set{1,\ldots,p}\) and solves a sequence of regularized subproblems as \(\tau\searrow0\); see~\cite{Scholtes_2001}.
		Other relaxation strategies for \(G_i(x)H_i(x)=0\) have also been proposed; see~\cite{Hoheisel_2013,Kadrani_2009,Kanzow_2013,Kanzow_2015}.
		An alternative way is the penalty method, which incorporates the orthogonality constraints to the objective via a penalty term, i.e., \(f(x)+\rho \langle G(x),H(x) \rangle \), and solves a sequence of subproblems with increasing \(\rho>0\); see~\cite{Anitescu_2007,Hu_2004,Ralph_2004}.
		In addition, other types of MPCC-tailored algorithms have been introduced, including interior point~\cite{Leyffer_2006,Raghunathan_2005}, active-set~\cite{Izmailov_2008}, and augmented Lagrangian methods~\cite{Jia_2023}, all of which have demonstrated good numerical performance.
		Furthermore, nonlinear programming (NLP) solvers have also been successfully applied to MPCCs; see~\cite{Fletcher_2006,Izmailov_2012}.
		Most of methods in the literature analyze convergence in the limiting setting,
		showing that accumulation points satisfy one of the S/M/C-stationarity conditions under appropriate assumptions.
		However, the methods with complexity guarantees, which aim to find an approximate stationary point within a finite number of steps, have not yet been fully studied in the literature.
		Moreover, many existing methods often involve solving general nonlinear subproblems or rely on complex parameter tuning strategies, making them less attractive for practical implementation.
		Motivated by these observations, the goal of this paper is to develop an algorithm for MPCCs that is both easy to implement and equipped with complexity guarantees.
		
		Recently, a quadratic penalty method for nonconvex optimization was proposed in~\cite{Grapiglia_2023}, which features both simple subproblem structure and straightforward parameter tuning.
		Therefore, we leverage this framework to design our algorithm for MPCCs and establish its complexity guarantees.
		It is important to note, however, that the quadratic penalty for complementarity constraints corresponds to the Moreau envelope of their indicator function, which is nonsmooth.
		To address this nonsmoothness, we propose using a surrogate: the Lasry--Lions double envelope~\cite{Attouch_1993}.
		The subproblems derived from the Lasry--Lions double envelope are smooth, and thus amenable to efficient smooth optimization techniques, such as fast gradient-based solvers~\cite{Ghadimi_2019} or L-BFGS-B~\cite{Zhu_1997}.
		To the best of our knowledge, the use of the Lasry--Lions double envelope is new.
		The only closely related work is the recent paper~\cite{Simoes_2021}, which applied this tool in nonsmooth nonconvex optimization for signal processing.

		\subsection{Contribution}
			The main contribution of this work is the development of a novel homotopy-based algorithm for MPCCs, leveraging the Lasry--Lions double envelope to construct a smooth approximation of the complementarity constraints, which pose particular challenges due to their inherent nonsmoothness, nonconvexity, and lack of standard CQs.
			Second, we introduce novel MPCC-tailored CQs that are weaker than the widely used MPCC-MFCQ, under which we establish the convergence to C/M-stationary points.
			Third, we propose a practical finite-iteration version of the algorithm that is easy to implement, and we provide a rigorous worst-case complexity analysis for this variant.
			Finally, we validate the effectiveness and robustness of the proposed method through extensive numerical experiments on MPCC benchmark problems, and show that it is competitive with state-of-the-art NLP solvers and the classical regularization method~\cite{Scholtes_2001} in both solution quality and computational efficiency.
			In particular, the proposed method is capable of efficiently handling large-scale and complex complementarity constraints, making it a promising tool for practical applications in engineering and economics.

		\subsection{Notation}
			The set of natural numbers is \(\N\coloneqq\set{1,2,\ldots}\).
			The \emph{extended real line} is denoted by \(\overline{\R}\coloneqq\R\cup\set{\pm \infty}\).
			The set of nonnegative and strictly positive reals are \(\R_+\coloneqq[0,\infty)\) and \(\R_{++}\coloneqq(0,\infty)\).
			We denote the positive part of a real number by \({[\cdot]}_+\coloneqq\max\set{0,\cdot}\).
			For \(\varphi:\R^{n}\rightarrow\overline{\R}\), the \emph{domain} is defined as \(\dom\varphi\coloneqq\set{x\in\R^n}[\varphi(x)<\infty]\),
			the \emph{conjugate} is given by \(\varphi^*(y)\coloneqq\sup_x\set{\langle y,x\rangle-\varphi(x)}\) and the \emph{biconjugate} is defined as \(\varphi^{**}(x)\coloneqq\sup_y\set{\langle y,x\rangle-\varphi^*(y)}\).
			Given \(\varepsilon>0\), we denote \(\varepsilon\)-\(\arg\min\varphi\coloneqq\set{x\in\R^n}[\varphi(x)\leq\inf\varphi+\varepsilon]\).
			The collection of subsequences of \(\N\) is \(\mathcal{N}_{\infty}^{\sharp}\).
			The \emph{convex hull} of a set \(E\subseteq \R^n\) is denoted by \(\conv E\), and its \emph{interior} by \(\interior E\).
			The notation \(D\) always stands for the \emph{complementarity set} in this work, i.e.,
			\begin{equation}\label{CC_set}
				D\coloneqq\set{z\in\R^{2}_{+}}[z_{1}z_{2}=0]\subset\R^{2}.
			\end{equation}
			We denote the \(i\)th complementarity constraint pair by \(F_i\coloneqq(G_i,H_i):\R^{n}\to\R^{2}\), so that the complementarity constraints can be expressed by \(F_i(x)\in D\) for all \(i\in\set{1,\ldots,p}\).
			The \emph{Jacobian} of \(F_i\) at \(x\) is denoted by \(JF_i(x)\), i.e.,
			\[
			JF_i(x)=
			\begin{pmatrix}
				\nabla G_i{(x)}^\top
				\\
				\nabla H_i{(x)}^\top
			\end{pmatrix}
			\in\R^{2\times n}.
			\]
			The \emph{Euclidean distance} between a point \(z\) and the set \(D\) is denoted \(\dist(z,D)\coloneqq\min_{z'\in D}\|z'-z\|\) and the (possibly set-valued) \emph{Euclidean projection} onto the set \(D\) is \(\Pi_{D}(z)\coloneqq\arg\min_{z'\in D}\|z'-z\|\).
			The closed ball with center at \(x\in\R^n\) and radius \(\xi\) is \(\B(x,\xi)\coloneqq\set{x'\in\R^n}[\|x'-x\|\leq\xi]\).

	\section{Preliminaries and Supporting Results}
		In this section, we begin by reviewing key concepts from variational and nonsmooth analysis, which form the mathematical foundation for our discussion.
		We then introduce several stationarity concepts commonly used in the MPCC literature.
		Finally, we review some MPCC-tailored CQs, and use them as the foundation to propose a novel and weaker variant.
		The notation and terminology are primarily based on the works of~\cite{Burke_2013,Hoheisel_2013,Jia_2023}.
		
		\subsection{Fundamentals of Variational Analysis}
			The \emph{epigraph} of \(\varphi:\R^{n}\rightarrow\overline{\R}\) is the set \(\epi \varphi\coloneqq\set{(x,\alpha)\in\R^{n+1}}[\varphi(x)\leq\alpha]\),
			and \(\varphi\) is said to be \emph{lower semicontinuous (lsc)} if \(\epi \varphi\) is closed, and proper if \(\epi \varphi\) is nonempty and \(\varphi(x)>-\infty\) for all \(x\in\R^n\).
			For a sequence of sets \(\set{C^\nu}\) with \(C^\nu\subseteq \R^n\) for all \(\nu\in\N\), we define the \emph{outer limit} as
			\[
			\Limsup_{\nu\to\infty} C^{\nu}\coloneqq\set{x}
			[\exists K\in \mathcal{N}_{\infty}^{\sharp},\
			\{x^\nu\}\to_{K}x:
			x^{\nu}\in C^{\nu}\ \ \forall\nu\in K]
			\]
			and the \emph{inner limit} as
			\[
			\Liminf_{\nu\to\infty} C^{\nu}\coloneqq\set{x}
			[\exists k\in \N,\
			\{x^\nu\}\to x:
			x^{\nu}\in C^{\nu}\ \ \forall\nu\geq k].
			\]
			We say that \(\set{C^\nu}\) \emph{set-converges} if its outer limit and inner limit are equal, i.e.,
			\[
			\Lim_{\nu\to\infty} C^{\nu}\coloneqq\Limsup_{\nu\to\infty} C^{\nu}=\Liminf_{\nu\to\infty} C^{\nu}.
			\]
			We say that a sequence \(\set{\varphi^{\nu}}\) of functions \(\varphi^{\nu}:\R^{n}\rightarrow\overline{\R}\) \emph{epi-converges} to \(\varphi:\R^{n}\rightarrow\overline{\R}\) if
			\[
			\Lim_{\nu\to\infty} \epi \varphi^{\nu}=\epi \varphi,
			\]
			which is denoted by \(\varphi^{\nu}\overset{\mathrm{e}}{\rightarrow}\varphi\). For set-valued mapping \(S:\R^{n}\rightrightarrows\R^{m}\), the \emph{graph} of \(S\) is the set \(\graph S\coloneqq\set{(x,u)\in\R^{n+m}}[u\in S(x)]\).
			We say that a sequence \(\set{S^{\nu}}\) of set-valued mappings \(S^{\nu}:\R^{n}\rightrightarrows\R^{m}\) \emph{graphically converges} to \(S:\R^{n}\rightrightarrows\R^{m}\) if
			\[
			\Lim_{\nu\to\infty} \graph S^{\nu}=\graph S,
			\]
			which is denoted by \(S^{\nu}\overset{\mathrm{g}}{\rightarrow}S\).
			Considering the complementarity set \(D\), its \emph{indicator function} and  Euclidean projection are respectively given by
			\[
				\delta_{D}(z)
			=
				\begin{cases}
					0 &  \text{if } z\in D \\
					+\infty &  \text{otherwise}
				\end{cases}
			\quad\text{and}\quad
				\Pi_{D}(z)
			=
				\begin{cases}
					\Pi_{\R_{+}^{2}}(z) &  \text{if }z\notin\R_{++}^{2} \\
					(z_1,0) &  \text{if } z_1>z_2>0 \\
					(0,z_2) &  \text{if }z_2>z_1>0 \\
					\set{(z_1,0),(0,z_2)} &  \text{if } z_1=z_2>0.
				\end{cases}
			\]
			We next review some cone definitions for a set \(E\subseteq \R^n\), which are important for characterizing different stationarity conditions.
			The \emph{tangent cone} at a point \(x\in E\) is
			\[
			T_{E}(x)\coloneqq\set{d\in\R^{n}}[\exists E\ni x^{\nu}\to x\ \text{and}\ t^{\nu}\searrow 0\ \text{such that}\ \tfrac{x^{\nu}-x}{t^{\nu}}\to d].
			\]
			The \emph{regular normal cone} of \(E\) at \(x\) is defined as
			\[
			\hat{N}_{E}(x)\coloneqq\set{
			d\in\R^{n}}[\langle d,d'\rangle \leq0\ \ \forall d'\in T_{E}(x)].
			\]
			The \emph{limiting normal cone} of \(E\) at \(x\) is defined as
			\[
			N_{E}(x)\coloneqq\set{d\in\R^{n}}[\exists E \ni x^{\nu}\to x\ \text{and}\ d^{\nu}\in\hat{N}_{E}(x^{\nu})\ \text{such that}\ d^{\nu}\to d].
			\]
			As for a convex set \(C\), we have \(\hat{N}_{C}(x)=N_{C}(x)\) for all \(x\in C\).
			Based on the above definitions, the regular normal cone and the limiting normal cone of the complementarity set \(D\) at \(z\in D\) are explicitly given by
			\[
			\hat{N}_{D}(z)
			=
			\begin{cases}
				\set{0}\times\R & \text{if } z_{1}>0 \\
				\R\times\set{0} & \text{if } z_{2}>0 \\
				\R_{-}^{2} & \text{if } z=(0,0)
			\end{cases}
			\quad
			\text{and}
			\quad
			N_{D}(z)
			=
			\begin{cases}
				\set{0}\times\R & \text{if } z_{1}>0 \\
				\R\times\set{0} & \text{if } z_{2}>0 \\
				\R_{-}^{2}\cup D & \text{if } z=(0,0).
			\end{cases}
			\]
			As a set-valued mapping, note that the limiting normal cone is \emph{outer semicontinuous}, in the sense that $\Limsup_{z^{\nu}\to z}N_{D}(z^{\nu})\subseteq N_{D}(z)$, and has better calculus properties compared to the regular normal cone.
			Furthermore, we introduce additional cone concepts as
			\begin{align*}
				N_D^{\rm C}(z)
			\coloneqq{} &
				\begin{cases}
					\hat{N}_{D}(z) & \text{if } z\neq(0,0) \\
					N_{D}((0,0))\cup\R_{+}^{2} & \text{if } z=(0,0)
				\end{cases}
			\shortintertext{and}
				N_D^{\rm W}(z)
			\coloneqq{} &
				\begin{cases}
					\hat{N}_{D}(z) & \text{if } z\neq(0,0) \\
					\R^{2} & \text{if } z=(0,0).
				\end{cases}
			\end{align*}
			We will refer to the cone \(N_D^{\rm C}\) as the \emph{Clarke limiting cone}, since it is related to the Clarke-stationarity in MPCCs.
			Note that \(N_D^{\rm C}\) is different from the \emph{Clarke normal cone} in~\cite[Exercise 6.38]{Rockafellar_1998}.
			The cone \(N_D^{\rm W}\) is referred to as the \emph{weak limiting cone}, which will be used to formulate a generalized version of MPCC-MFCQ (see \cref{definition_MFCQ+}).

		\subsection{Stationarity Concepts}
			To each of these cones, \(\hat{N}_{D}\), \(N_{D}\) and \(N_D^{\rm C}\), there corresponds one of the following stationarity concepts for MPCC problems.

			\begin{definition}[Stationary point]\label{definition_M_C}
			Relative to~\eqref{MPCC_simple}, a feasible point \(x\in\R^{n}\) is said to be \emph{strongly stationary (S-stationary)} if there exist \(y_{1},\ldots,y_{p}\in\R^{2}\) such that
			\[
			\begin{cases}
			\nabla f(x)
			+
			\sum_{i=1}^{p}
			JF_{i}{(x)}^{\top}y_{i}
			+
			N_{C}(x)
			\ni
			0\\[5pt]
			y_{i}\in \hat{N}_D(F_{i}(x))
			\end{cases}
			\]
			for all \(i\in\set{1,\ldots,p}\),
			or equivalently, if it satisfies
			\[
			-\nabla f(x)\in \sum_{i=1}^{p}JF_{i}{(x)}^{\top}\hat{N}_D(F_{i}(x))+N_{C}(x).
			\]
			If the above condition holds with \(\hat{N}_D(F_{i}(x))\) being replaced by \(N_D(F_{i}(x))\) for all \(i\in\set{1,\ldots,p}\), then we say that \(x\) is \emph{Mordukhovich stationary (M-stationary)}.
			Furthermore, if \(\hat{N}_D(F_{i}(x))\) is instead replaced by \(N_D^{\rm C}(F_{i}(x))\) for all \(i\in\set{1,\ldots,p}\), then \(x\) is said to be \emph{Clarke stationary (C-stationary)}.
			\end{definition}

			\begin{figure}[t]\centering
				\includegraphics[width=0.9\linewidth]{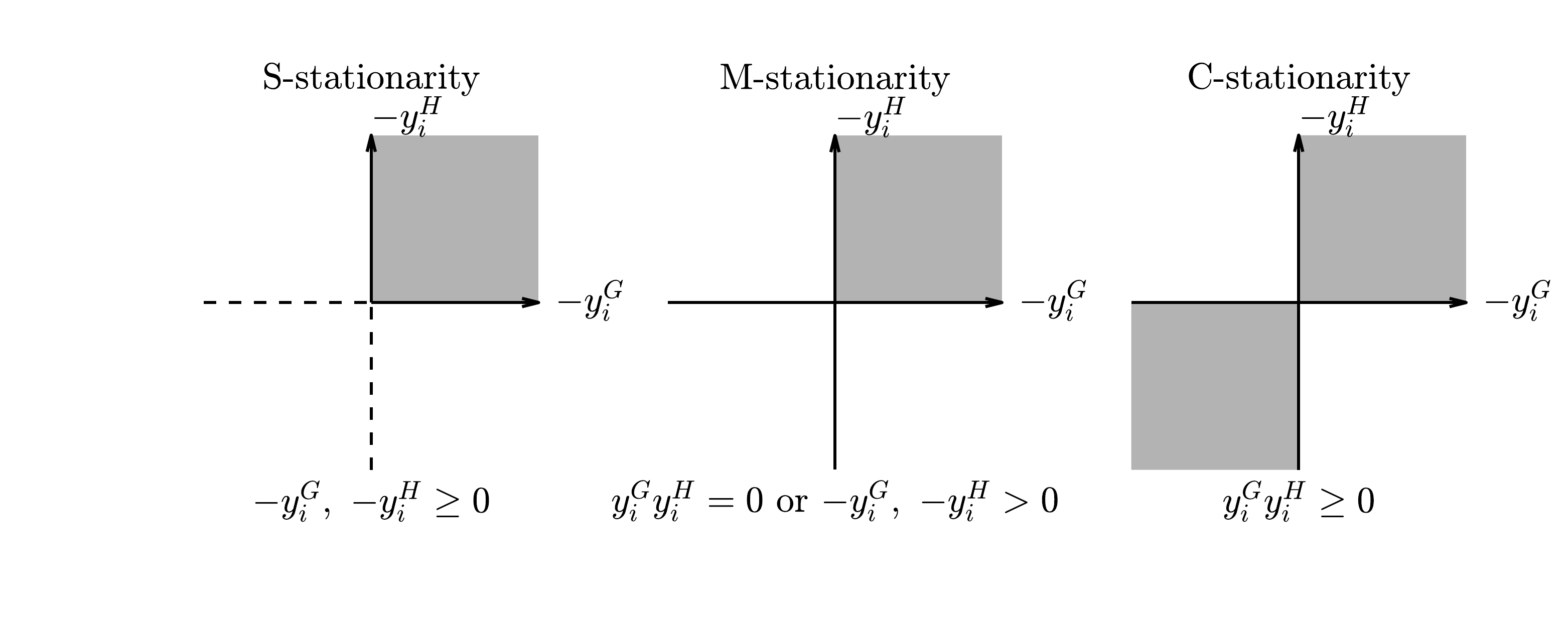}%
				\caption{Multipliers of the MPCC-Lagrangian associated with biactive constraints \(G_i(x)=H_i(x)=0\) for different stationary points.}%
				\label{fig_sta}%
			\end{figure}

			In the literature, \(-y_i\) typically serves as the multiplier associated with the \emph{MPCC-Lagrangian} \(\mathcal{L}:\R^n\times\R^{2p}\to\R\), given by
			\[
			\mathcal{L}(x,\y)=f(x)+\sum_{i=1}^{p}y_i^G G_i(x)+\sum_{i=1}^{p}y_i^H H_i(x),
			\]
			in which \(\y=(y_{1},\ldots,y_{p})\in\R^{2p}\) and \(y_i=(y_i^G,y_i^H)\in\R^2\); see~\cite{Izmailov_2012}.
			With this notation, the stationarity concepts introduced in \cref{definition_M_C} are consistent with the standard definitions in~\cite[Definition 2.3]{Hoheisel_2013}.
			Each of these stationarity concepts treats differently the signs of the multipliers associated with \emph{biactive} complementary constraint, that is, those satisfying \(G_i(x)=H_i(x)=0\); see \cref{fig_sta} for an illustration.
			In contrast, when the solution satisfies either \(G_i(x)>0,H_i(x)=0\) or \(G_i(x)=0,H_i(x)>0\) for all \(i\in\set{1,\ldots,p}\), the various stationarity conditions are equivalent.
			Note that the S-stationarity is equivalent to the standard KKT conditions of an MPCC, since there are no feasible descent directions~\cite{Hoheisel_2013}.
			On the other hand, for C/M-stationary points, feasible descent directions may exist.
			From this point onward, we will refer to \(y_i\) as the multiplier associated with the generalized equations in \cref{definition_M_C}.
			We now introduce the concept of asymptotic M-stationarity.

			\begin{definition}[AM-stationary point {\cite[Definition 2.2]{Jia_2023}}]\label{def:AKKT}%
			Relative to \eqref{MPCC_simple}, a feasible point \(x\in\R^{n}\) is said to be \emph{asymptotic Mordukhovich stationary (AM-stationary)} if there exist sequences \(\{(x^\nu,\y^\nu,\z^\nu)\}\) such that
			\[
				\begin{cases}
				\nabla f(x^\nu)+\sum_{i=1}^{p}JF_{i}{(x^\nu)}^\top y_{i}^\nu+N_{C}(x^\nu)\to0\\[5pt]
				(x^\nu,z_{i}^\nu)\to (x,F_{i}(x))\\[5pt]
				y_{i}^\nu\in N_{D}(z_{i}^\nu)
				\end{cases}
			\]
			for all \(i\in\set{1,\ldots,p}\),
			where \(\y^\nu=(y_{1}^\nu,\ldots,y_{p}^\nu)\in\R^{2p}\) and \(\z^\nu=(z_{1}^\nu,\ldots,z_{p}^\nu)\in\R^{2p}\).
			\end{definition}

			Clearly, any point that is M-stationary is also AM-stationary, as it corresponds to the constant sequence \((x^\nu,y_i^\nu,z_i^\nu)=(x,y_i,F_i(x))\) for all \(i\in\set{1,\ldots,p}\), and the converse holds as long as \(\set{\y^\nu}\) is bounded~\cite{Jia_2023}.
				Note that because of the continuity of \(\nabla f\) and \(JF_i\), there is no loss of generality in simplifying the above definition by taking \(z_i^\nu=F_i(x^\nu)\).
			The extra variable \(z_i^\nu\) is however convenient as it allows infeasible point \(x^\nu\), in the sense that \(F_i(x^\nu)\in D\) need not hold.

		\subsection{Constraint Qualifications}
			To guarantee that a local minimizer of the MPCC~\eqref{MPCC_simple} is stationary in the sense of \cref{definition_M_C}, special MPCC-tailored CQs are used;
			in the literature, MPCC-MFCQ is among the most widely used ones.
			For example, the methods proposed in~\cite{Lin_2005,Scholtes_2001} are shown to converge to C-stationary points under this CQ;\@ see also~\cite[Theorems 3.1 and 3.3]{Hoheisel_2013}.
			Let us consider general constraints
			\begin{equation}\label{general_feasible_set}
				\begin{cases}
					G_i(x)\geq0,\ H_i(x)\geq0,\ G_i(x)H_i(x)=0,\ i\in\set{1,\ldots,p}
				\\
					c_{i}(x)=0,\ i\in\mathcal{E}
				\\
					c_{i}(x)\leq0,\ i\in\mathcal{I}
				\end{cases}
			\end{equation}
			and define some crucial index sets at a feasible point \(x\) of~\eqref{general_feasible_set}:
			\begin{equation}\label{index_set_I00}
				\begin{split}
					I_{\mathcal{I}}(x)
				\coloneqq{} &
					\set{i\in\mathcal{I} }[c_{i}(x)=0],
				\\
					I_{0+}(x)
				\coloneqq{} &
					\set{i\in\set{1,\ldots,p}}[G_i(x)=0,\ H_i(x)>0],
				\\
					I_{+0}(x)
				\coloneqq{} &
					\set{i\in\set{1,\ldots,p}}[G_i(x)>0,\ H_i(x)=0],
				\\
					I_{00}(x)
				\coloneqq{} &
					\set{i\in\set{1,\ldots,p}}[G_i(x)=0,\ H_i(x)=0].
				\end{split}
			\end{equation}
			For simplicity we may omit the dependence on \(x\) in the index sets above, if there is no confusion.
			We are now in a position to recall the conventional MPCC-MFCQ needed.
			\begin{definition}[MPCC-MFCQ {\cite[Definition 3(b)]{Kanzow_2015}}]\label{definition_MFCQ}
			Let \(x\in C\) be feasible for~\eqref{general_feasible_set}.
			Then, the \emph{MPCC-Mangasarian-Fromovitz constraint qualification (MPCC-MFCQ)} holds at \(x\) if whenever
			\[
				\begin{cases}
					\alpha_i\in\R_+,\
					\pi_i\in\R,\
					(y_i^G,y_i^H)\in\R^2
				\\[5pt]
					\sum_{i\in I_{\mathcal{I}}}\alpha_i\nabla c_i(x)
					+\sum_{i\in \mathcal{E}}\pi_i\nabla c_i(x)
				\\[5pt]
					+\sum_{i\in I_{0+}\cup I_{00}}y_i^G\nabla G_i(x)
					+\sum_{i\in I_{+0}\cup I_{00}}y_i^H\nabla H_i(x)
					+N_C(x)
					\ni0
				\end{cases}
			\]
			it holds that
			\[
				\begin{cases}
					\alpha_i=0,\ i\in I_{\mathcal{I}}
				\\
					\pi_i=0,\ i\in \mathcal{E}
				\\
					y_i^G=0,\ i\in I_{0+}\cup I_{00}
				\\
					y_i^H=0,\ i\in I_{+0}\cup I_{00}.
				\end{cases}
			\]
			\end{definition}

			Note that the conventional \MFCQ\@ is not formulated within the framework of general composite constraints of the form \(\delta_E(h(x))\), where \(E\) can represent sets such as \(\R_+\), \(\{0\}\) or \(D\), and \(h\) may correspond to constraint functions like \(c_i\) or \(F_i\).
			To leverage the powerful tools of variational analysis, the constraints in~\eqref{general_feasible_set} can be reformulated within the composite framework.
			In this framework, a widely used concept of CQ is the \emph{basic constraint qualification (BCQ)}, as presented in~\cite[Theorem 5.5]{Burke_2013} and~\cite[Theorem 6.14]{Rockafellar_1998}, which is also called \emph{no nonzero abnormal multiplier CQ} or \emph{generalized MFCQ} in the literature~\cite{Jia_2023}.
			Accordingly, for the feasible set defined in~\eqref{general_feasible_set}, we introduce the BCQ as follows:

			\begin{definition}[BCQ {\cite[Theorem 6.14]{Rockafellar_1998}}]\label{definition_BCQ}
			Let \(x\in C\) be feasible for~\eqref{general_feasible_set}.
			Then, the \emph{basic constraint qualification (BCQ)} holds at \(x\) if whenever
			\[
				\begin{cases}
					\alpha_i\in N_{\R_-}(c_i(x)),\
					\pi_i\in\R,\
					y_{i} \in N_D(F_i(x))
				\\[5pt]
					\sum_{i\in \mathcal{I}}\alpha_i\nabla c_i(x)
					+\sum_{i\in \mathcal{E}}\pi_i\nabla c_i(x)
					+\sum_{i=1}^{p}
					JF_{i}{(x)}^\top y_i
					+N_C (x)\ni0
				\end{cases}
			\]
			it holds that
			\[
				\begin{cases}
					\alpha_i=0,\ i\in \mathcal{I}
				\\
					\pi_i=0,\ i\in \mathcal{E}
				\\
					y_i=(0,0),\ i\in\set{1,\ldots,p}.
				\end{cases}
			\]
			\end{definition}

			Both \MFCQ\@ and \BCQ\@ handle different types of constraints separately, lacking a unified treatment.
			Motivated by this observation, one of the main goals of this paper is to introduce a new algorithm that handles all constraints in~\eqref{general_feasible_set} in a unified manner.
			To this end, we begin with the following definition and example:

			\begin{definition}[Well-behaved CC]\label{definition_well_CC}
			\(F_i\coloneqq(G_i,H_i):\R^n\to\R^2\) is said to be a \emph{well-behaved complementarity constraint (well-behaved CC)}
			if there exists \(\xi\in\R_{++}\) such that \(F_i(x)\notin\B(0,\xi)\cap\R^2_{++}\) for every \(x\in\R^n\).
			\end{definition}

			The appeal of well-behaved CCs in avoiding the diagonal in the positive orthant close to the origin will be apparent in \cref{lemma_y_unbounded_bounded}.

			\begin{example}[Equality and inequality constraints are well-behaved]\label{example_equality_inequality_CC}
			Consider equality and inequality constraints, and their CC reformulation:
			\[
				\begin{cases}
					c_{i}(x)=0,\ i\in\mathcal{E}
				\\
					c_{i}(x)\leq0,\ i\in\mathcal{I}
				\end{cases}
				\quad
				\Leftrightarrow
				\quad
				F_i(x)=
				\begin{cases}
				( c_{i}(x),a),\ \ &i\in\mathcal{E}
				\\
				( -c_{i}(x),0),\ \ &i\in\mathcal{I}
				\end{cases}
				~~
				\text{with}
				~~
				F_i(x)\in D,
			\]
			where \(a>0\) is a given constant.
			Then, \(F_i\) is a well-behaved CC for all \(i\in\mathcal{E}\cup\mathcal{I}\).
			Specifically, when \(i\in\mathcal{E}\), any \(\xi\in(0,a)\) is such that \(F_i(x)\notin\B(0,\xi)\) for all \(x\in\R^n\). Therefore, \(F_i\) is a well-behaved CC according to \cref{definition_well_CC}.
			When \(i\in\mathcal{I}\), it is evident that \(F_i(x)\) never lies in \(\R^2_{++}\),
			which implies that \(F_i\) is well-behaved.
			It is worth noting that there are several alternative ways to reformulate the equality constraint.
			For example, one could use \(F_i=(c_i(x),-c_i(x))\in D\), or introduce two separate components: \(F_{i^+}=(c_i(x),0)\in D\) and \(F_{i^-}=(-c_i(x),0)\in D\).
			All of these formulations result in well-behaved CCs, as none of the corresponding \(F_i(x)\) lies in \(\R^2_{++}\).
			In the experiments presented in \cref{sec_experiments}, we use the \(F_{i^+},F_{i^-}\) reformulation, as it demonstrates better numerical performance compared to the alternatives.
			\end{example}

			The concept of well-behaved CCs is particularly useful when equality and inequality constraints are reformulated as CCs.
			Through this reformulation, optimization problems with constraints of the form~\eqref{general_feasible_set} can be uniformly recast into the structure of~\eqref{MPCC_simple}.
			Building on the notion of well-behaved CCs, we now introduce a generalized version of \MFCQ\@:

			\begin{definition}[MPCC-MFCQ\textsuperscript{\textbf{+}}]\label{definition_MFCQ+}
			Relative to~\eqref{MPCC_simple},
			let \(\mathcal{W}\subseteq\set{1,\ldots,p}\) be the index set of well-behaved CCs and \(y_i=(y_i^G,y_i^H)\in\R^2\) for \(i\in\set{1,\ldots,p}\).
			Then, \emph{MPCC-MFCQ\textsuperscript{\textbf{+}}} holds at a feasible point \(x\) if whenever
			\[
			\begin{cases}
				y_{i} \in N_D(F_i(x)),\ i\in\mathcal{W}
			\\[5pt]
				y_{i} \in N_D^{\rm W}(F_i(x)),\ i\notin\mathcal{W}
			\\[5pt]
				\sum_{i=1}^{p}
				JF_{i}{(x)}^\top y_i
				+N_C (x)\ni0
			\end{cases}
			\]
			one has that
			\[
			\begin{cases}
				\text{either}\ G_i\ \text{is constant or}\ y_{i}^G=0,
			\\
				\text{either}\ H_i\ \text{is constant or}\ y_{i}^H=0,
				\ i\in\set{1,\ldots,p}.
			\end{cases}
			\]
			\end{definition}

			The statement of constant \(G_i\) and \(H_i\) in \MFCQ+ may seem unusual at first glance, but it arises naturally from the reformulation of equality and inequality constraints as CCs.
			Specifically, when such constraints are reformulated as CCs following \cref{example_equality_inequality_CC}, one may have \(F_i=(G_i,H_i)\) for \(i\in\mathcal{W}\), where \(G_i(x)=c_i(x)\) and \(H_i(x)=a\geq0\).
			Since \(H_i\) is a constant mapping, it holds that
			\[
			JF_i{(x)}^\top y_i
			=
			\begin{pmatrix}
				\nabla c_i(x),0
			\end{pmatrix}
			\begin{pmatrix}
				y_i^G
				\\
				y_i^H
			\end{pmatrix}.
			\]
			Therefore, the multiplier \(y_i^H\) corresponding to such \(H_i\) is excluded from the requirement of being zero, as it is always annihilated by the zero block in the Jacobian.

			It is important to note that when \(\mathcal{W}=\mathcal{E}\cup\mathcal{I}\), \MFCQ+\@ is equivalent to \MFCQ\@.
			In contrast, when \(\mathcal{W}\supset \mathcal{E}\cup\mathcal{I}\), \MFCQ+\@ is weaker than \MFCQ\@, since the region of the multiplier \(y_i\) for \(i\in\mathcal{W}\setminus (\mathcal{E}\cup\mathcal{I})\) is restricted from \(N_D^{\rm W}(F_i(x))\) to \(N_D(F_i(x))\).
			Similar to \MFCQ+, by leveraging the concept of well-behaved CCs, we can modify \BCQ\@ accordingly.
			In this case, however, rather than obtaining a weaker CQ, we obtain an equivalent condition, but with a much simplified and condensed statement.

			\begin{definition}[BCQ\textsuperscript{\textbf{+}}]\label{definition_BCQ+}
			Relative to~\eqref{MPCC_simple},
			let \(y_i=(y_i^G,y_i^H)\in\R^2\) for \(i\in\set{1,\ldots,p}\).
			Then, \emph{BCQ\textsuperscript{\textbf{+}}} holds at a feasible point \(x\) if whenever
			\[
			\begin{cases}
				y_{i} \in N_D(F_i(x)),\ i\in\set{1,\ldots,p}
			\\[5pt]
				\sum_{i=1}^{p}
				JF_{i}{(x)}^\top y_i
				+N_C (x)\ni0
			\end{cases}
			\]
			one has that
			\[
			\begin{cases}
				\text{either}\ G_i\ \text{is constant or}\ y_{i}^G=0,
			\\
				\text{either}\ H_i\ \text{is constant or}\ y_{i}^H=0,
				\ i\in\set{1,\ldots,p}.
			\end{cases}
			\]
			\end{definition}

			The equivalence of \BCQ+ and \BCQ\@ stems from the fact that they both involve only limiting normal cones.
			Building on \BCQ+\@, we now introduce our main MPCC-tailored CQ as follows, which will serve as a key tool in the convergence analysis of our proposed algorithm.

			\begin{definition}[BCCQ]\label{definition_BCCQ}
			Relative to~\eqref{MPCC_simple},
			let \(\mathcal{W}\subseteq\set{1,\ldots,p}\) be the index set of well-behaved CCs and \(y_i=(y_i^G,y_i^H)\in\R^2\) for \(i\in\set{1,\ldots,p}\).
			Then, the \emph{basic Clarke constraint qualification (BCCQ)} holds at a feasible point \(x\) if whenever
			\[
			\begin{cases}
				y_{i} \in N_D(F_i(x)),\ i\in\mathcal{W}
			\\[5pt]
				y_{i} \in N_D^{\rm C}(F_i(x)),\ i\notin\mathcal{W}
			\\[5pt]
				\sum_{i=1}^{p}
				JF_{i}{(x)}^\top y_i
				+N_C (x)\ni0
			\end{cases}
			\]
			one has that
			\[
			\begin{cases}
				\text{either}\ G_i\ \text{is constant or}\ y_{i}^G=0,
				\\
				\text{either}\ H_i\ \text{is constant or}\ y_{i}^H=0,
				\ i\in\set{1,\ldots,p}.
			\end{cases}
			\]
			\end{definition}

			Applying the methods from \cref{example_equality_inequality_CC} to reformulate~\eqref{general_feasible_set} entirely into CCs reveals the following chain of implications among these MPCC-tailored CQs:
			\begin{equation}\label{relation_CQs}
			\text{\MFCQ}
			\Longrightarrow
			\text{\MFCQ+}
			\Longrightarrow
			\text{\BCCQ}
			\Longrightarrow
			\text{\BCQ+}
			\Longleftrightarrow
			\text{\BCQ}.
			\end{equation}
			Beyond the previous discussion, this chain shows that \BCCQ\@ is implied by \MFCQ+\@.
			Specifically, \BCCQ\@ restricts the region of the multiplier \(y_i\) from \(N_D^{\rm W}(F_i(x))\) to \(N_D^{\rm C}(F_i(x))\) for \(i\notin \mathcal{W}\), while keeping the region of the multiplier \(y_i\) for \(i\in \mathcal{W}\) unchanged.
			Consequently, if \MFCQ+\@ holds at a feasible point \(x\), then \BCCQ\@ is necessarily satisfied as well.

	\section{Lasry--Lions Double Envelope}
		We now introduce the fundamental concepts and key properties of the Lasry--Lions double envelope.
		This construction is derived from the Moreau envelope and serves as a tool for smoothing the given function.
		
		\begin{definition}[Moreau envelope]\label{definition_Moreau}
		Given \(\lambda>0\), the \emph{Moreau envelope} of a proper and lsc function \(\varphi:\R^{n}\to\overline{\R}\) is \(\env_\lambda\varphi:\R^{n}\to\overline{\R}\) given by
		\[
			\env_\lambda\varphi(x)
		\coloneqq
			\inf_{x'\in\R^{n}}
			\set{
			\varphi(x')
			+\tfrac{1}{2\lambda}\|x'- x\|^2}.
		\]
		The set of minimizers is denoted by \(\prox_{\lambda\varphi}(x)\).
		Furthermore, \(\varphi\) is \emph{prox-bounded} if there exists \(\lambda>0\) such that \(\env_\lambda\varphi(x)>-\infty\) for some \(x\in\R^{n}\).
		The supremum of the set of all such \(\lambda\) is the threshold \(\lambda_\varphi\).
		\end{definition}
		
		It is important to note that when \(\varphi\) is lower bounded, such as in the case of an indicator function, we have \(\lambda_\varphi=+\infty\).
		Building on \cref{definition_Moreau}, we define the Lasry-Lions envelope as follows:
		
		\begin{definition}[Lasry--Lions double envelope]\label{definition_LL}
		Given \(\lambda>\mu>0\), the \emph{Lasry--Lions double envelope} of a proper and lsc function \(\varphi:\R^{n}\to\overline{\R}\) is \(\env_{\lambda,\mu}\varphi:\R^{n}\to\overline{\R}\) given by
		\begin{equation}\label{sup_LL_definition}
			\env_{\lambda,\mu}\varphi(x)
		\coloneqq
			-\env_{\mu}(-\env_\lambda\varphi)(x)
		=
			\sup_{x'\in\R^{n}}
			\set{
			\env_\lambda\varphi(x')
			-\tfrac{1}{2\mu}\|x'- x\|^2}.
		\end{equation}
		We denote the set of maximizers of~\eqref{sup_LL_definition} by \(P_{\lambda,\mu}\varphi(x)\coloneqq\prox_{-\mu\env_\lambda\varphi}(x)\).
		Furthermore, the \emph{proximal hull} of \(\varphi\) is given by setting \(\mu=\lambda\) in~\eqref{sup_LL_definition}, and is denoted as \(\env_{\lambda,\lambda}\varphi(x)\).
		\end{definition}
		
		We summarize below some fundamental properties of the Lasry--Lions double envelope that are particularly valuable for the design and theoretical analysis of algorithms.
		
		\begin{proposition}[Basic properties of the Lasry--Lions double envelope]
		Let \(\varphi:\R^{n}\to\overline{\R}\) be proper, lsc, and prox-bounded with threshold \(\lambda_\varphi\).
			The following hold:%
			\begin{enumerate}
			\item
			\(\env_{\lambda,\mu}\varphi(x)\) with \(0<\mu<\lambda<\lambda_\varphi\) is a smooth function whose gradient is \(\max\set{\frac{1}{\mu},\frac{1}{\lambda-\mu}}\)-Lipschitz continuous.\label{fact_Lipschitz}
			\item
					\(\env_\lambda\varphi(x)
			\leq
			\env_{\lambda,\mu}\varphi(x)
			\leq
			\env_{\lambda-\mu}\varphi(x)\leq\varphi(x)\) for all \(0<\mu<\lambda<\lambda_\varphi\).\label{fact_sandwich}
				\item
					\(\inf\env_{\lambda,\mu}\varphi=\inf\env_{\lambda}\varphi=\inf\varphi\) and \(\arg\min\env_{\lambda,\mu}\varphi=\arg\min\env_{\lambda}\varphi=\arg\min\varphi\) for all \(0<\mu<\lambda<\lambda_\varphi\).\label{fact_global_optimum}
			\item
					Suppose that there exist \(0<\bar\mu<\bar\lambda<\lambda_\varphi\) such that \(\env_{\bar\lambda,\bar\mu}\varphi\) is level bounded. Then, any sequence \(x\in\varepsilon\)-\(\argmin\env_{\lambda,\mu}\varphi\) with \(\mu<\lambda\), \(\lambda\searrow0\), \(\mu\searrow0\) and \(\varepsilon\searrow0\) is bounded and has all its cluster points in \(\argmin\varphi\).\label{fact_preservation}
			\item
			Suppose further that \(\varphi\) is bounded from below,
			let \(\lambda>0\) be given and \(j\coloneqq\frac{1}{2}\|\cdot\|^2\).
			Then, \(\env_{\lambda,\mu}\varphi\) epi-converges to \(\env_{\lambda,\lambda}\varphi\) with \(\mu>0\) and \(\mu\nearrow\lambda\),
			the proximal hull is given by \(\env_{\lambda,\lambda}\varphi=(\varphi+\lambda^{-1}j){}^{**}-\lambda^{-1}j\) with domain \(\dom\env_{\lambda,\lambda}\varphi=\conv(\dom\varphi)\).\label{proposition_conv_envelope}
			\end{enumerate}
		\end{proposition}
		
		\begin{proof}
		\cref{fact_Lipschitz} follows from~\cite[Theorem 4.1]{Attouch_1993},~\cite[Proposition 12.62]{Rockafellar_1998} and~\cite[Proposition 2.3]{Simoes_2021};
		\cref{fact_sandwich,fact_global_optimum} are derived from~\cite[Example 1.46]{Rockafellar_1998};
		\cref{fact_preservation} is based on~\cite[Theorem 3.1]{Simoes_2021}.
		
		Regarding \cref{proposition_conv_envelope},
		since \(\varphi\) is bounded from below, \cref{fact_global_optimum} and~\cite[Example 1.44]{Rockafellar_1998} yield \(\inf\env_{\lambda,\lambda}\varphi=\inf\env_{\lambda}\varphi=\inf\varphi>-\infty\).
		By~\cite[Exercise 1.24 and Example 1.44]{Rockafellar_1998}, it follows that \(\env_{\lambda,\lambda}\varphi\) is proper, lsc and prox-bounded with threshold \(+\infty\).
		Furthermore, the Lasry--Lions double envelope satisfies \(\env_{\lambda,\mu}\varphi=\env_{\lambda-\mu}(\env_{\lambda,\lambda}\varphi)\)~\cite[Example 1.46]{Rockafellar_1998}.
		Since \(\env_{\lambda,\lambda}\varphi\) is proper and lsc, \cite[Theorem 1.25]{Rockafellar_1998} implies pointwise convergence \(\env_{\lambda-\mu}(\env_{\lambda,\lambda}\varphi)(x)\nearrow\env_{\lambda,\lambda}\varphi(x)\) for all \(x\) as \(\lambda-\mu\searrow0\).
		Finally, with the aid of~\cite[Proposition 7.4(d)]{Rockafellar_1998}, we conclude that \(\env_{\lambda,\mu}\varphi\overset{\mathrm{e}}{\rightarrow}\env_{\lambda,\lambda}\varphi\) as \(\mu\nearrow\lambda\).
		The domain \(\dom\env_{\lambda,\lambda}\varphi=\conv(\dom\varphi)\) and formulation \(\env_{\lambda,\lambda}\varphi=(\varphi+\lambda^{-1}j){}^{**}-\lambda^{-1}j\) follow from~\cite[Example 11.26(c)]{Rockafellar_1998}.
		\end{proof}
		
		\cref{proposition_conv_envelope} establishes the epi-convergence of the Lasry--Lions double envelope to the proximal hull of the original function.
		Furthermore, the expression \((\varphi+\lambda^{-1}j){}^{**}-\lambda^{-1}j\) shows that the proximal hull \(\env_{\lambda,\lambda}\varphi\) approaches the closed convex hull \(\varphi^{**}\) as \(\lambda\to+\infty\).
		Therefore, this result provides an important insight for algorithm design: one can start with large values of \(\lambda\) and \(\mu\), such that \(\lambda-\mu\) is small, yielding a nearly convex and smooth surrogate function,
		and then gradually decrease \(\lambda\) and \(\mu\) over iterations to improve approximation accuracy.
		On the other hand, the following lemma establishes the epi-convergence of the Lasry--Lions double envelope to the original function, which is crucial for guaranteeing the convergence of the minima and minimizers of the surrogate problems to those of the original problem.
		
		\begin{lemma}[Epi-convergence to the original function]\label{lemma_epi_converge_double_envelope}
		Let $\varphi:\R^{n}\to\overline{\R}$ be proper, lsc, and prox-bounded with threshold \(\lambda_\varphi\). If $\lambda$ and $\mu$ satisfy $\lambda\searrow0$ and $\mu\searrow0$ with $0<\mu<\lambda<\lambda_\varphi$, then  $\env_{\lambda,\mu}\varphi$ epi-converges to $\varphi$.
		\end{lemma}
		
		\begin{proof}
		According to \cite[Theorem 1.25]{Rockafellar_1998}, we have pointwise convergence $\env_\lambda\varphi(x)\nearrow\varphi(x)$ for all $x$ as $\lambda\searrow0$.
		Then, based on~\cite[Proposition 7.4(d)]{Rockafellar_1998}, we have $\env_\lambda\varphi\overset{\mathrm{e}}{\rightarrow}\varphi$ as $\lambda\searrow0$.
		Similarly for $\env_{\lambda-\mu}\varphi\overset{\mathrm{e}}{\rightarrow}\varphi$ as $\lambda-\mu\searrow0$.
		Finally, according to the sandwich property \cref{fact_sandwich},
		we conclude $\env_{\lambda,\mu}\varphi\overset{\mathrm{e}}{\rightarrow}\varphi$ from~\cite[Proposition 7.4(g)]{Rockafellar_1998}.
		\end{proof}

		\begin{figure}[t]
		\begin{minipage}{0.49\linewidth}
			\vspace*{0pt}%
			\centering
			\includegraphics[width=.75\linewidth]{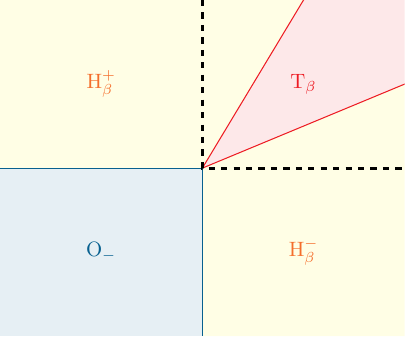}%
		\end{minipage}
		\hfill
		\begin{minipage}{0.49\linewidth}
			\caption{%
			Polygonal areas of the piecewise definition of $\env_{\lambda,\mu}\delta_{D}$.
			The thick dashed line corresponds to the complementarity constraint set $D$, while the red lines dividing the yellow $\mathrm{H}_{\beta}^{\pm}$ regions from the red $\mathrm{T}_{\beta}$ region are the graphs of $z_2=\frac{1}{1-\beta}z_1$ and $z_2=(1-\beta)z_1$ in the positive orthant, where $\beta=\nicefrac{\mu}{\lambda}\in(0,1)$.
			Apparently, as $\beta\searrow0$ the $\mathrm{T}_{\beta}$ region collapses into a half-line, while as $\beta\nearrow1$ it expands to the entire positive orthant.%
			}\label{fig:areas}
		\end{minipage}
		\end{figure}
		
		We now examine the envelopes of the indicator function associated with CCs, i.e., $\varphi=\delta_{D}$.
		The Moreau envelope of $\delta_{D}$ admits the following compact representation:
		\begin{equation}\label{moreau_compl}
		\env_\lambda\delta_D(z)
		=
			\tfrac{1}{2\lambda}\dist {(z,D)}^2
		=
			\tfrac{1}{2\lambda}
		\begin{pmatrix}
			\|z\|^{2}
			-
			\max\set{{[z_1]}_{+},{[z_2]}_{+}}^2
		\end{pmatrix}.
		\end{equation}

		We next present the explicit formulation of the Lasry--Lions double envelope of \(\delta_D\).

		\begin{theorem}[Lasry--Lions double envelope of \(\delta_D\)]\label{theorem_LL_env}
			Relative to the complementarity set \(D\) in~\eqref{CC_set}, let \(z\in\R^2\), \(\lambda>\mu>0\) and \(\beta\coloneqq\nicefrac{\mu}{\lambda}\in(0,1)\), the Lasry--Lions double envelope of the indicator function \(\delta_D\) is given by
			\begin{equation}\label{LL_env}
			\env_{\lambda,\mu}\delta_D(z)
			=
			\tfrac{1}{\lambda}r_{\beta}(z),
			\end{equation}
			with
			\begin{equation}\label{r_beta}
			r_{\beta}(z)
			=
			\left\{
			\begin{array}{llr}
			\frac{1}{2(1-\beta)}\|z\|^{2}
			& \ \ \text{if } z\in\R_{-}^{2}
			& \ \ [\mathrm{O}_{-}]
			\\ [5pt]
			\frac{1}{2\beta(2-\beta)}{(z_1+z_2)}^{2}
			-
			\frac{1}{2\beta}\|z\|^{2}
			& \ \ \text{if } (1-\beta)z_1\leq z_2\leq \frac{1}{1-\beta}z_1
			& \ \ [\mathrm{T}_{\beta}]
			\\ [5pt]
			\frac{1}{2(1-\beta)}\min\set{z_1,z_2}^2
			& \ \ \text{otherwise},
			& \ \ [\mathrm{H}_{\beta}^{\pm }]
			\end{array}
			\right.
			\end{equation}
			where the notations \(\mathrm{O}_{-},\mathrm{H}_{\beta}^{+},\mathrm{H}_{\beta}^{-},\mathrm{T}_{\beta}\) refer to the regions as labeled in \cref{fig:areas}.
		\end{theorem}
		\begin{proof}
			See \cref{appendix_theorem_LL_env}.
		\end{proof}
		
		In light of \cref{theorem_LL_env}, the gradient of \(\env_{\lambda,\mu}\delta_D(z)\) is characterized by
		\[
		\nabla\env_{\lambda,\mu}\delta_D(z)
		=
		\tfrac{1}{\lambda}R_{\beta}(z),
		\]
		where
		\begin{equation}\label{R_beta}
		R_{\beta}(z)
		=
		\left\{
		\begin{array}{llr}
		\frac{1}{1-\beta}z
		& \ \ \text{if } z\in\R_{-}^{2}
		& \ \ [\mathrm{O}_{-}]
		\\ [5pt]
		\frac{1}{\beta(2-\beta)}
			\binom{z_1+z_2}{z_1+z_2}
					-
					\tfrac{1}{\beta}z
		& \ \ \text{if } (1-\beta)z_1\leq z_2\leq \frac{1}{1-\beta}z_1
		& \ \ [\mathrm{T}_{\beta}]
		\\ [5pt]
		\frac{1}{1-\beta}\binom{z_1}{0}
		& \ \ \text{if }
		z_2\geq\frac{1}{1-\beta}{[z_1]}_{+}
		& \ \ [\mathrm{H}_{\beta}^{+}]
		\\ [5pt]
		\frac{1}{1-\beta}\binom{0}{z_2}
		& \ \ \text{if }
		z_1\geq\frac{1}{1-\beta}{[z_2]}_{+}.
		& \ \ [\mathrm{H}_{\beta}^{-}]
		\end{array}
		\right.
		\end{equation}
		\cref{fact_global_optimum} implies that the global minimum of \(r_\beta\) is \(r_\beta^\star=0\) and the minimizer set is \(D\).
		Furthermore, from the representation~\eqref{moreau_compl} and \cref{fact_sandwich}, it holds that
		\begin{equation}\label{r_beta_nonnegative}
		r_\beta(z)\geq\lambda\env_\lambda\delta_D(z)
		=
		\tfrac{1}{2}\dist {(z,D)}^2\geq0.
		\end{equation}
		The distinction between the Moreau and Lasry--Lions double envelopes is illustrated in \cref{fig:env}.
		The next proposition highlights two important properties of \(r_\beta\): the global Lipschitz continuity of the gradient and the \emph{Polyak-\L{}ojasiewicz (PL)} property.
		
		\begin{figure}[t]
		\begin{minipage}{0.4\linewidth}
			\vspace*{0pt}%
			\centering
			\includegraphics[width=\linewidth]{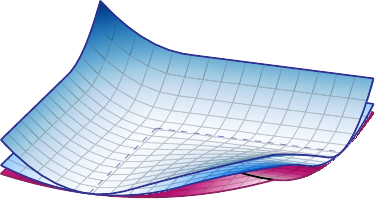}%
		\end{minipage}
		\hfill
		\begin{minipage}{0.59\linewidth}
			\caption{%
				Comparison between the Moreau and Lasry--Lions double envelopes.
				The red surface is the Moreau envelope \(\env_\lambda\delta_D\) when \(\lambda=1\).
				It is nonsmooth on \(\set{z\in\R^2}[z_1=z_2>0]\).
				The middle surface is the Lasry--Lions envelope \(\env_{\lambda,\beta\lambda}\delta_D\) when \(\lambda=1\) and \(\beta=0.5\), which is smooth everywhere and \(\env_{\lambda,\beta\lambda}\delta_D(z)\geq\env_\lambda\delta_D(z)\).
				Since \(\env_{\lambda,\beta\lambda}\delta_D\) provides an extra design freedom \(\beta\) compared to \(\env_\lambda\delta_D\), by increasing \(\beta\) one obtains a more accurate approximation of \(\delta_D(z)\), as in the top surface corresponding to \(\lambda=1\) and \(\beta=0.8\).
				Moreover, one has \(\env_{\lambda,\beta\lambda}\delta_D(z)=\env_\lambda\delta_D(z)=\delta_D(z)=0\) for \(z\in D\) (dashed line).%
			}%
			\label{fig:env}%
		\end{minipage}
		\end{figure}
		
		\begin{proposition}
			Fix \(\beta\in(0,1)\), and let functions \(r_\beta\) and \(R_\beta\) be as in \eqref{r_beta} and \eqref{R_beta}.
			The following hold:%
			\begin{enumerate}
			\item
				$R_{\beta}$ is globally Lipschitz continuous with modulus
				\(
				L_\beta
				=
				\max\set{
					\tfrac{1}{\beta},
					\tfrac{1}{1-\beta}
				}
				\).\label{proposition_globally_Lipschitz}
			\item
				\(r_\beta\) satisfies the PL inequality
			\[
			\tfrac{1}{2}
			\begin{Vmatrix}
				R_\beta(z)
			\end{Vmatrix}^2
			\geq
			l_\beta
			(r_\beta(z)-r_\beta^\star)
			\quad \forall z\in\R^2
			\]
			with \(l_\beta
			=
			\min\set{
				\tfrac{1}{1-\beta},
				\tfrac{1-\beta}{\beta(2-\beta)}
			}\) and \(r_\beta^\star=0\).\label{proposition_PL}
			\end{enumerate}
		\end{proposition}
		
		\begin{proof}
		By \cref{fact_Lipschitz}, the Lipschitz modulus of \(\nabla\env_{\lambda,\mu}\delta_D\) is \(L_{\nabla\env_{\lambda,\mu}\delta_D}=\max\set{\frac{1}{\mu},\frac{1}{\lambda-\mu}}\).
		Moreover, we have
		\(
		R_{\beta}(z)
		=
		\lambda\nabla\env_{\lambda,\mu}\delta_D(z)
		\).
		Therefore, the Lipschitz modulus of \(R_{\beta}\) is given by
		\[
			L_\beta
			=
			\lambda L_{\nabla\env_{\lambda,\mu}\delta_D}
			=
			\max\set{
				\tfrac{1}{\beta},
				\tfrac{1}{1-\beta}
			}.
		\]
		On the other hand, it is straightforward to verify that the PL inequality holds with \(l_\beta=\tfrac{1}{1-\beta}\) when \(z\in\mathrm{O}_{-}\cup\mathrm{H}_{\beta}^{\pm}\), recalling that these regions are defined in \cref{fig:areas}.
		Let us consider the remaining case where \(z\in\mathrm{T}_{\beta}\).
		Then, we obtain
		\begin{align*}
		&
			\tfrac{1}{2}
			\begin{Vmatrix}
				R_\beta(z)
			\end{Vmatrix}^2
			-
			l_\beta r_\beta(z)
		\\
		={} &
			\tfrac{1}{2}
			\begin{pmatrix}
				\frac{{(1-\beta)}^2+1}{\beta^2{(2-\beta)}^2}
				+
				\frac{1-\beta}{\beta(2-\beta)}l_\beta
			\end{pmatrix}\|z\|^2
			-
			\begin{pmatrix}
				\frac{1-\beta}{\beta^2{(2-\beta)}^2}
				+
				\frac{1}{\beta(2-\beta)}l_\beta
			\end{pmatrix}z_1z_2.
		\end{align*}
		By letting \(l_\beta=\tfrac{1-\beta}{\beta(2-\beta)}\), we further have
		\[
		\tfrac{1}{2}
		\begin{Vmatrix}
			R_\beta(z)
		\end{Vmatrix}^2
		-
		l_\beta r_\beta(z)
		=
		\tfrac{{(1-\beta)}^2}{\beta^2{(2-\beta)}^2}
		\begin{pmatrix}
			z_1-z_2
		\end{pmatrix}^2
		+
		\tfrac{1}{2\beta^2{(2-\beta)}^2}
		\|z\|^2
		\geq0,
		\]
		yielding the claimed inequality.
		\end{proof}
		
		The Lasry--Lions double envelopes in \cref{definition_LL} were originally introduced by~\cite{Lasry_1986} as approximations that can exhibit better behavior than the Moreau envelopes; for further details, see also~\cite{Attouch_1993,Simoes_2021}.
		Indeed, the properties established in this section make \(r_\beta\) particularly well suited for smoothing the highly combinatorial structure of CCs; see \cref{fig_example_kth3} below for an illustrative example.
		Therefore, it plays a central role in the design and analysis of the method proposed in \cref{sec_infinite,sec_finite}.

	\section{A Homotopy Approach for MPCC}\label{sec_infinite}
		Rather than addressing the original problem~\eqref{MPCC_simple} directly, we adopt an alternative strategy by solving a sequence of smoothed subproblems.
		They are formally defined as follows, given two parameters \(\lambda>\mu>0\):
		\begin{equation}\label{LL_subproblem}
		\minimize_{x\in C}
		\ \
		f(x)
		+
		\sum_{i=1}^{p}
		\env_{\lambda,\mu}\delta_{D}(F_{i}(x)).
		\end{equation}
		The stationary points \(x\) of~\eqref{LL_subproblem} are characterized by
		\begin{equation}\label{LL_KKT}
		\nabla f(x)
		+
		\tfrac{1}{\lambda}
		\sum_{i=1}^{p}
		JF_{i}{(x)}^{\top}
		R_{\beta}(F_{i}(x))
		+
		N_{C}(x)
		\ni
		0,
		\end{equation}
		where we remind that \(\beta:=\nicefrac{\mu}{\lambda}\in(0,1)\) and \(R_{\beta}:\R^{2}\to\R^{2}\) is the Lipschitz mapping defined in~\eqref{R_beta}.
		An approximate stationary point at iteration \(\nu\), computed with parameters \(\lambda^\nu\) and \(\beta^\nu\), is denoted by \(x^\nu\).
		This point is then used to initialize the next smoothed subproblem with updated parameters \(\lambda^{\nu+1}\) and \(\beta^{\nu+1}\).
		
		\begin{figure}[h]\centering
			\begin{subfigure}{0.32\textwidth}
			\includegraphics[width=\linewidth]{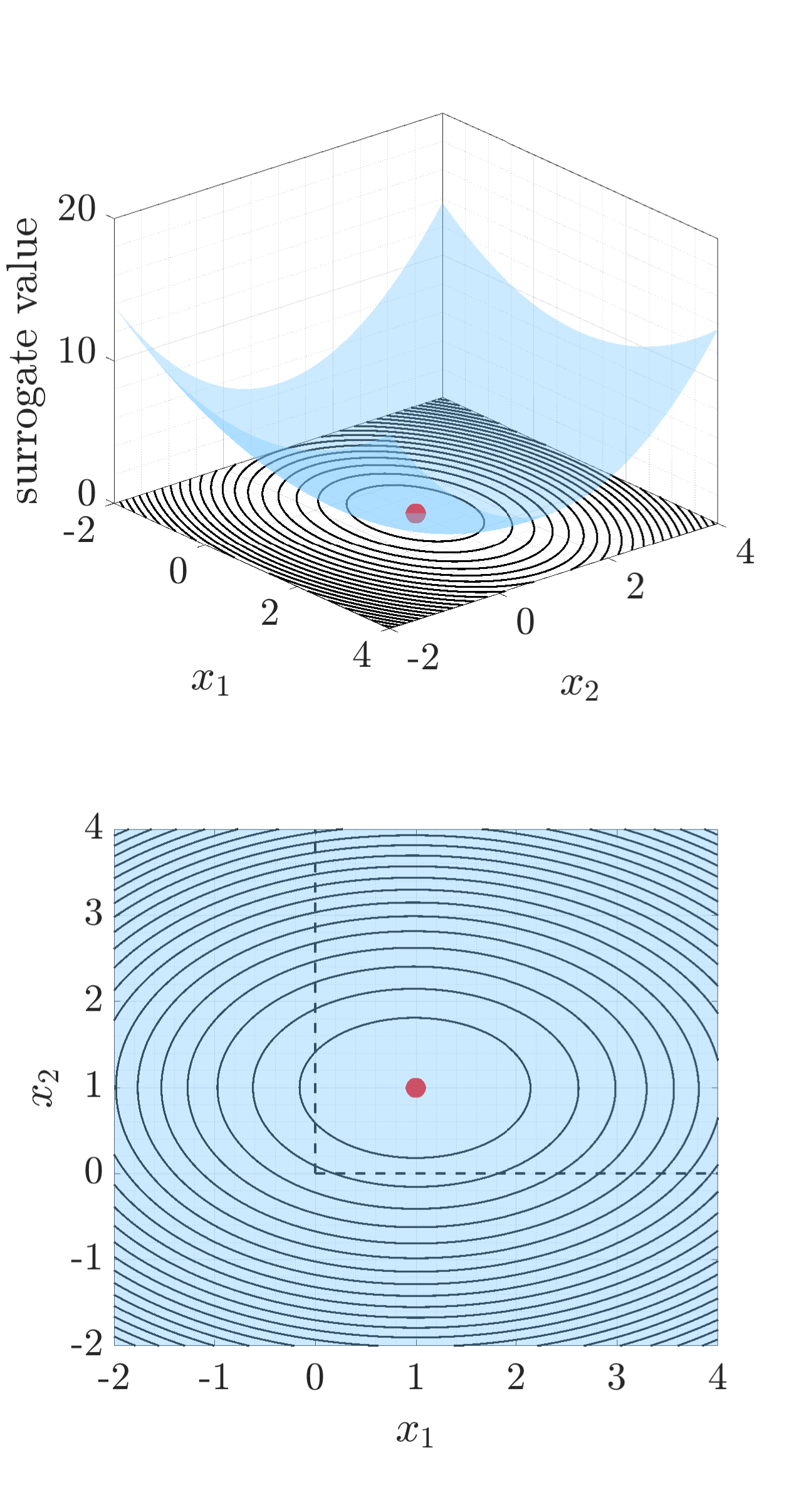}
			\end{subfigure}%
			\begin{subfigure}{0.32\textwidth}
			\includegraphics[width=\linewidth]{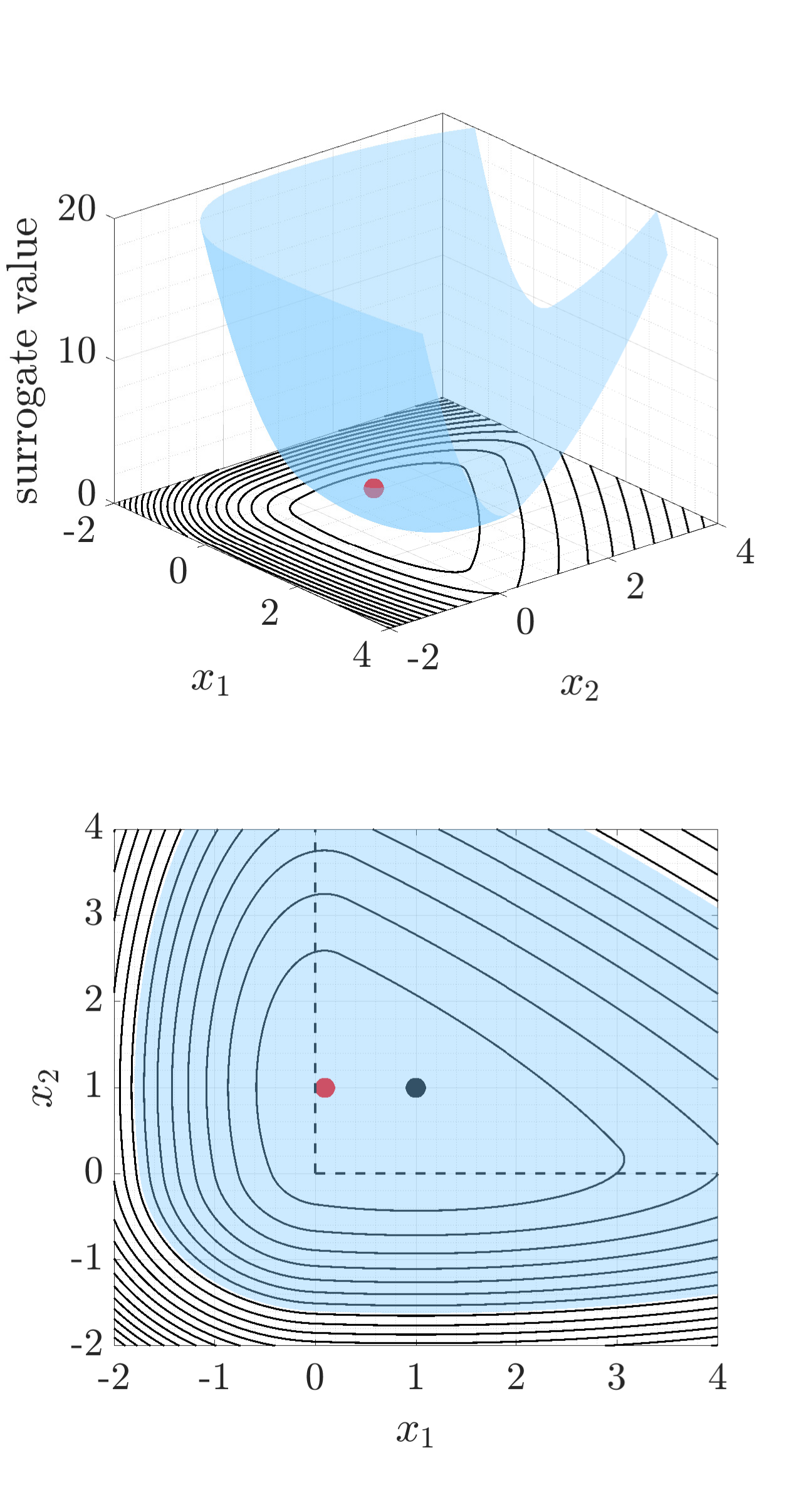}
			\end{subfigure}%
			\begin{subfigure}{0.32\textwidth}
			\includegraphics[width=\linewidth]{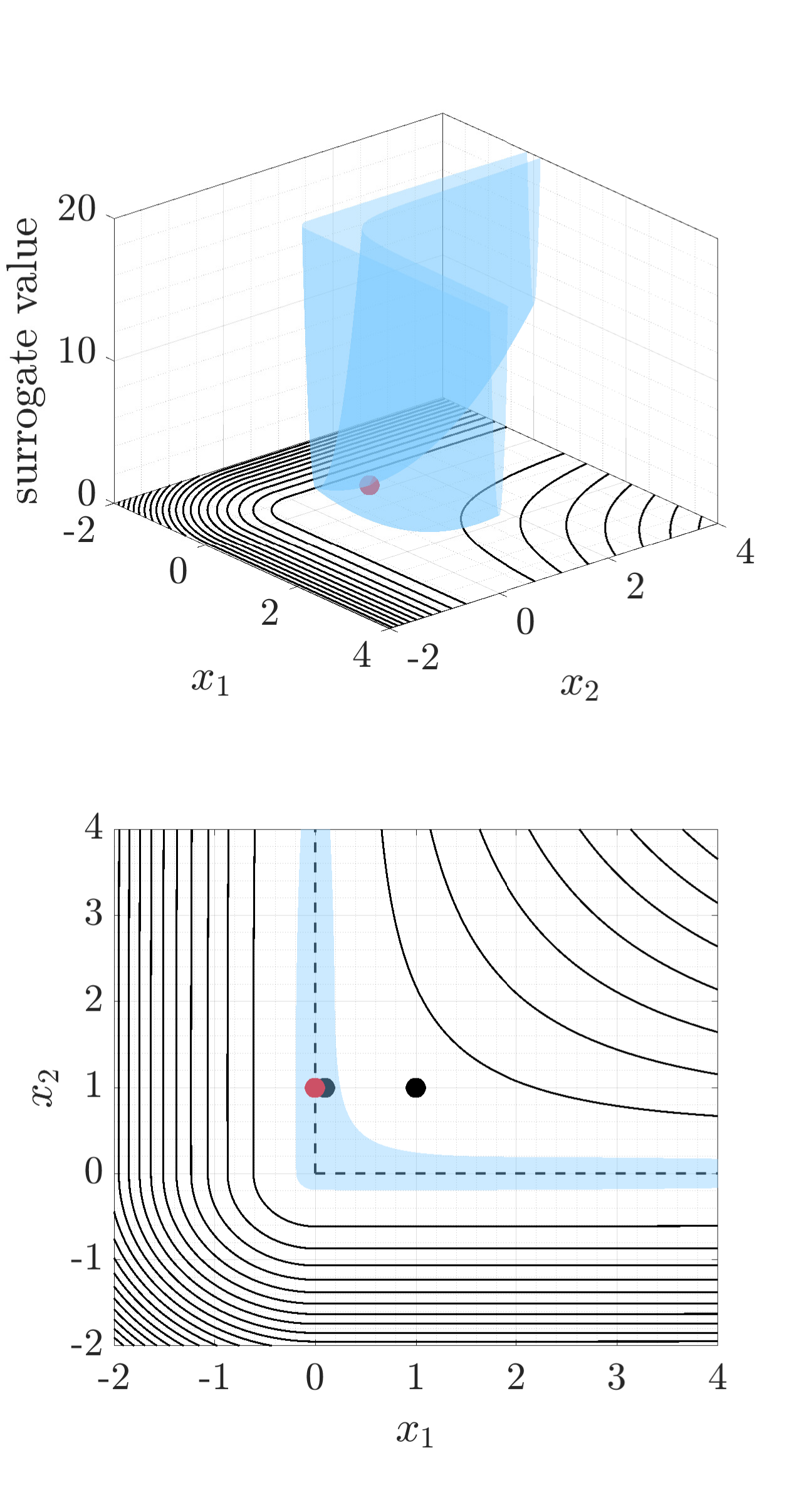}
			\end{subfigure}%
			\caption{
			Solving the \emph{kth3}~\eqref{example_kth3} via the LL-based homotopy approach.
			The red dot indicates the current subproblem's solution, while the black dot indicates the solution from the previous subproblem.
			The black curves are contours of \(s_{\lambda,\beta}\).
			The dashed line represents the feasible set.
			The bottom row of subfigures provides a top-down view of the subfigures on top.}%
			\label{fig_example_kth3}
		\end{figure}
		
		\begin{example}[kth3]
		The basic idea of our Lasry--Lions (LL)-based homotopy approach is illustrated by using the simple example \emph{kth3} from the MacMPEC collection~\cite{Leyffer_2000}.
		Consider the MPCC\@:
		\begin{equation}\label{example_kth3}
			\begin{split}
			\minimize_{x\in\R^2}
			\ \
			& 0.5{(x_1-1)}^2+{(x_2-1)}^2
			\\
			\stt\ \
			& 0\leq x_1\perp x_2\geq0,
			\end{split}
		\end{equation}
		whose global minimizer is \((0,1)\).
		The associated subproblem is
		\[
			\minimize_{x\in\R^2}
			s_{\lambda,\beta}(x)
			=
			0.5{(x_1-1)}^2+{(x_2-1)}^2
			+
			\tfrac{1}{\lambda}
			r_{\beta}(x).
		\]
		Let \(\beta=0.9\) be fixed and start with \(\lambda^1=100\).
		As shown in the left column of \cref{fig_example_kth3}, and in light of \cref{proposition_conv_envelope}, the resulting smoothed subproblem is nearly convex and well-conditioned.
		This allows it to be efficiently solved by first-order methods, yielding the solution \(x^1=(1,1)\).
		Next, set \(\lambda^2=1\).
		The corresponding subproblem, depicted in the middle column of \cref{fig_example_kth3}, yields the solution \(x^2=(0.1,1)\).
		Finally, with \(\lambda^3=0.01\), the subproblem grows ill-conditioned, as shown in the right column of \cref{fig_example_kth3}.
		While this improves the approximation of the feasible set, it would typically hinder convergence due to numerical instability.
		However, thanks to warm-starting from the previous solution \(x^2\), the algorithm avoids the local minimizer \((1,0)\) and successfully recovers the global minimizer \(x^3=(0,1)\).
		\end{example}
		
		Building on the above notation, we introduce
		\begin{equation}\label{function_s}
		s(x)
		\coloneqq
		f(x)
		+
		\sum_{i=1}^{p}
		\delta_D(F_i(x))
		\quad
		\text{and}
		\quad
		s_{\lambda,\beta}(x)
		\coloneqq
		f(x)
		+
		\tfrac{1}{\lambda}
		\sum_{i=1}^{p}
		r_{\beta}(F_{i}(x)).
		\end{equation}
		Then, according to \cref{lemma_epi_converge_double_envelope} and~\cite[Lemma 5.1 and Proposition 5.2]{Burke_2013}, given two sequences $\set{\lambda^{\nu}}\subset\R_{++}$ and $\set{\beta^{\nu}}\subset(0,1)$ with $\lambda^{\nu}\searrow 0$ and $\limsup_{\nu\in\N}\beta^{\nu}<1$, we have
		\begin{equation}\label{epi_converge_s}
		s_{\lambda^\nu,\beta^\nu}
		\overset{\mathrm{e}}{\to}
		s.
		\end{equation}
		As shown in~\eqref{LL_KKT}, \(y_i^\nu=\frac{1}{\lambda^\nu}R_{\beta^\nu}(H_i(x^\nu))\) serves as the multiplier, thus the goal is to show that the sequence \(\set{y_i^\nu}\) is bounded and its limit points belong to either the cone \(N_D^{\rm C}(F_i(x))\) or \(N_D(F_i(x))\) with \(\set{x^\nu}\to x\).
		To this end, we first discuss the properties of \(\set{y_i^\nu}\).
		
		\begin{lemma}\label{lemma_y_unbounded_bounded}
		Given two sequences $\set{\lambda^{\nu}}\subset\R_{++}$ and $\set{\beta^{\nu}}\subset(0,1)$ with $\lambda^{\nu}\searrow 0$ and $\limsup_{\nu\in\N}\beta^{\nu}<1$,
		let \(\set{y^{\nu}}\) be a sequence such that $y^\nu=\tfrac{1}{\lambda^\nu}R_{\beta^\nu}(z^\nu)$ for all $\nu\in\N$ for some $\set{z^{\nu}}\to z\in D$, and let \(y\) be any cluster point of the sequence \(\set{y^{\nu}}\) or (provided that \(y^\nu\neq0\) for all \(\nu\)) of \(\set{\frac{y^{\nu}}{\|y^{\nu}\|}}\).
		Then, it holds that $y\in N_D^{\rm C}(z)$.
		Furthermore, a sufficient condition for the inclusion \(y\in N_D(z)\) to hold  (as opposed to \(y\in N_D^{\rm C}(z)\)) is the existence of \(\xi\in\R_{++}\) and \(k\in\N\) such that \(z^\nu\notin\B(0,\xi)\cap\R^2_{++}\) holds for all \(\nu\geq k\).
		\end{lemma}
		
		\begin{proof}
		Let \(\mathrm{O}_{-}\), \(\mathrm{H}_{\beta}^{\pm}\) and \(\mathrm{T}_{\beta}\) denote the regions of \(\R^2\) as in \cref{fig:areas},
		let \(y\) be a limit point of a subsequence \(\set{\frac{y^{\nu}}{\|y^{\nu}\|}}_{\nu\in K}\) or \(\set{y^{\nu}}_{\nu\in K}\).
		We consider the following cases:
		\begin{itemize}
		\item
			Suppose that \(z\in D\setminus\set{0}\).
			Without loss of generality, we consider \(z=(q,0)\) for some \(q>0\).
			Having assumed \(\sup_{\nu\in\N}\beta^\nu<1\), any sequence \(z^\nu\to z\) eventually belongs to $\mathrm{H}_{\beta^{\nu}}^{-}$. Therefore, for \(\nu\) sufficiently large, one has
			\[
				y^\nu
			=
				\tfrac{1}{\lambda^\nu}R_{\beta^\nu}(z^\nu)
			=
				\frac{1}{\lambda^\nu(1-\beta^\nu)}
				\binom{0}{z^\nu_2}.
			\]
			It follows that \(y\) belongs to \(\set{0}\times\R=\hat{N}_D((q,0))\subseteq N_D^{\rm C}(z)\).
		\item
			Suppose that \(z=0\) and that \(z^\nu\in\mathrm{H}_{\beta^{\nu}}^{-}\) holds for all \(\nu\in K\).
			Then, for all \(\nu\in K\) one again obtains the above relation.
			Therefore, \(y\) belongs to \(\set{0}\times\R=N_D((0,0))\subseteq N_D^{\rm C}(z)\).
			\item
			Suppose that \(z=0\) and that \(z^\nu\in\mathrm{O}_{-}\) holds for all \(\nu\in K\).
			Then, for all \(\nu\in K\) one has
			\[
				y^\nu
			=
				\tfrac{1}{\lambda^\nu}R_{\beta^\nu}(z^k)
			=
				\frac{1}{\lambda^\nu(1-\beta^\nu)}
				\binom{z^\nu_1}{z^\nu_2}
				\leq\binom{0}{0},
			\]
			which implies that \(y\) belongs to \(\R_-^2=\hat{N}_D((0,0))\subseteq N_D^{\rm C}(z)\).
		\item
			Suppose that \(z=0\) and that \(z^\nu\in\mathrm{T}_{\beta^{\nu}}\) holds for all \(\nu\in K\).
			Then, for all \(\nu\in K\) one has
			\[
			y^\nu
			=
				\tfrac{1}{\lambda^\nu}R_{\beta^\nu}(z^\nu)
			=
				\frac{1}{\lambda^\nu\beta^\nu(2-\beta^\nu)}
				\binom{z_2^\nu-(1-\beta^\nu)z_1^\nu}{z_1^\nu-(1-\beta^\nu)z_2^\nu}
			\geq
			\binom{0}{0},
			\]
			which implies that \(y\) belongs to \(\R_+^2\subseteq N_D^{\rm C}(z)\).
			\end{itemize}
			In this way, we conclude that \(y\in N_D^{\rm C}(z)\),
			and that \(y\in N_D(z)\) as long as \(z^\nu\notin\interior \mathrm{T}_{\beta^\nu}\) eventually always holds.
			This latter condition is guaranteed by the sufficient condition provided in the statement.
		\end{proof}
		
		The notation of well-behaved CCs in \cref{definition_well_CC} allows us to exclude the problematic region \(\mathrm{T}_{\beta}\) from the convergence analysis.
		In particular, by using the sufficient conditions in \cref{lemma_y_unbounded_bounded}, we establish the following result.
		
		\begin{corollary}\label{lemma_well_behaved_CC}
		Suppose that \(F_i\) is a well-behaved CC for some \(i\in\set{1,\ldots,p}\), and consider a sequence \(\set{x^\nu}\subset\R^n\) converging to a point \(x\) such that \(F_i(x)\in D\).
		Let \(\set{\lambda^{\nu}}\subset\R_{++}\) and \(\set{\beta^{\nu}}\subset(0,1)\) be sequences such that \(\lambda^{\nu}\searrow 0\) and \(\limsup_{\nu\in\N}\beta^{\nu}<1\), and let \(y_i^\nu=\tfrac{1}{\lambda^\nu}R_{\beta^\nu}(F_i(x^\nu))\).
		Then, any cluster point of \(\set{\frac{y_i^{\nu}}{\|\y^{\nu}\|}}\) or \(\set{y_i^{\nu}}\) belongs to the cone \(N_D(F_i(x))\), for such well-behaved \(F_i\).
		\end{corollary}
		
		\cref{lemma_y_unbounded_bounded} and \cref{lemma_well_behaved_CC} characterize the behavior of \(\set{\y^\nu}\).
		We now use these results to analyze the boundedness of the multipliers in the sequence of gradients
		\(\set{\nabla(\sum_{i=1}^{p}\env_{\lambda^\nu,\beta^\nu\lambda^\nu}\delta_D(F_i(x^\nu)))}\) based on \BCCQ, as stated in the following lemma.

		\begin{lemma}[Bounded multiplier]\label{lemma_bounded_multiplier}%
		Let \(\set{\lambda^{\nu}}\subset\R_{++}\) and \(\set{\beta^{\nu}}\subset(0,1)\) be sequences such that \(\lambda^{\nu}\searrow 0\) and \(\limsup_{\nu\in\N}\beta^{\nu}<1\).
		Consider a sequence \(\set{x^\nu}\subset\R^n\) such that \(\dist (-\nabla s_{\lambda^\nu,\beta^\nu}(x^\nu),N_C(x^\nu))\searrow0\), and suppose that \(\set{x^{\nu}}\to x\), where \(x\) satisfies \(F_i(x)\in D\) for all \(i\in\set{1,\ldots,p}\), and \BCCQ\@ holds at \(x\).
		Let \(\y^\nu=(y^\nu_{1},\ldots,y^\nu_{p})\) with \(y_i^\nu=\tfrac{1}{\lambda^\nu}R_{\beta^\nu}(F_i(x^\nu))\).
		Then, the sequence \(\set{\y^{\nu}}\) is bounded as \(\set{x^{\nu}}\to x\).
		\end{lemma}
		
		\begin{proof}
		To arrive at a contradiction, suppose that \(\set{\y^{\nu}}\) is unbounded.
		Without loss of generality we can assume that \(\set{\frac{\y^{\nu}}{\|\y^{\nu}\|}}\to \y=(y_1,\dots,y_p)\neq0\).
		It follows from the assumptions that there exists \(r^\nu\to0 \) such that
		\[
			r^\nu
		\in
		\nabla f(x^\nu)
		+
		\tfrac{1}{\lambda^\nu}
		\sum_{i=1}^{p}
		JF_{i}{(x^\nu)}^\top
		R_{\beta^\nu}(F_{i}(x^{\nu}))
		+
		N_C(x^\nu).
		\]
		Dividing the inclusion by \(\|\y^{\nu}\|\) and using the fact that \(N_C(x^\nu)\) is a cone, we have
		\[
			\frac{r^\nu}{\|\y^{\nu}\|}
		\in
		\frac{\nabla f(x^\nu)}{\|\y^{\nu}\|}
		+
		\sum_{i=1}^{p}
		JF_{i}{(x^\nu)}^\top
		\frac{y_i^{\nu}}{\|\y^{\nu}\|}
		+
		N_C(x^\nu),
		\]
		which, upon passing to the limit, yields that
		\[
		0
		\in
		\sum_{i=1}^{p}
		JF_{i}{(x)}^\top
		y_i
		+
		N_C(x).
		\]
		We next consider the well-behaved and the remaining CCs separately.
		\begin{itemize}
		\item
			If \(i\notin\mathcal{W}\), it follows from \cref{lemma_y_unbounded_bounded} that \(y_i\in N_D^{\rm C}(F_i(x))\).
		\item
			If \(i\in\mathcal{W}\), then \cref{lemma_well_behaved_CC} implies that \(y_i\in N_D(F_i(x))\).
			Moreover, if \(H_i\geq0\) is a constant mapping for the well-behaved CC \(F_i=(G_i,H_i)\), then the expression~\eqref{R_beta} of \(R_\beta\) implies that \({(R_{\beta^\nu}(F_i(x^\nu)))}_2=0\) for all \(\nu\) large enough, implying that \(y_i^H=0\) in this case.
			The same reasoning shows that \(y_i^G=0\) if \(G_i\geq0\) is a constant mapping.
		\end{itemize}
		Altogether, this shows that \(\y\) satisfies the conditions in \BCCQ, and must thus be zero by the \BCCQ\@ assumption, yielding the contradiction.
		\end{proof}
		
		\begin{remark}\label{remark_BCCQ}
			Since \BCCQ\@ is weaker than \MFCQ\@ (see \eqref{relation_CQs}), the statement of \cref{lemma_bounded_multiplier} also holds under \MFCQ, which is more commonly used in the MPCC literature.
			On the other hand, it can be seen from~\cite{Hoheisel_2013,Kanzow_2015} that \MFCQ\@ is not always required; instead, MPCC-CRCQ~\cite[Definition 3(c)]{Kanzow_2015} or MPCC-CPLD~\cite[Definition 3(d)]{Kanzow_2015} may be employed.
			Establishing the implications between MPCC-CRCQ, MPCC-CPLD and \BCCQ\@ remains an interesting topic for future research.
		\end{remark}
		
		\begin{remark}
		For simplicity of exposition, unless explicitly mentioned otherwise we set \(\mathcal{W} =\emptyset\) in the remaining proofs;
		all subsequent results have straightforward counterparts when \(\mathcal{W} \neq\emptyset\) is considered,
		up to replacing the Clarke limiting cone with the limiting normal cone for the corresponding indices \(i\in\mathcal{W}\).
		\end{remark}
		
		We are now in a position to state whether the approximating problems \eqref{LL_subproblem}, together with their optimality conditions~\eqref{LL_KKT}, are well-justified surrogates of the original problem~\eqref{MPCC_simple} and its optimality condition in \cref{definition_M_C}.
		To this end, we leverage the concept of \emph{consistent approximation}.
		
		\begin{definition}[Consistent approximation {\cite[Definition 2.2]{Royset_2023}}]\label{definition_WSA}
		Given functions \(s,s^\nu:\R^{n}\rightarrow\overline{\R}\) and set-valued mappings \(S,S^\nu:\R^{n}\rightrightarrows\R^{m}\), \(\nu\in\N\),
		the pairs \(\set{(s^\nu,S^\nu)}\) form a \emph{consistent approximation} of \((s, S)\) if
		\[
		s^{\nu}\overset{\mathrm{e}}{\rightarrow}s
		\ \ \text{and}\ \
		S^{\nu}\overset{\mathrm{g}}{\rightarrow}S.
		\]
		If the graphical convergence is relaxed to
		\begin{equation}\label{weakly_approximating}
		\Limsup_{\nu\to\infty}\graph S^\nu\subseteq \graph S,
		\end{equation}
		then the pairs \(\set{(s^\nu,S^\nu)}\) form a \emph{weakly consistent approximation} of \((s, S)\).
		\end{definition}
		
		\cref{fact_preservation} implies that if the subproblem can be progressively solved to global optimality, then the global minimizer of the original problem can be retrieved.
		However, solving the subproblems to global optimality is typically intractable.
		As an alternative, the following theorem shows that using first-order conditions, the limit \(\dist (-\nabla s_{\lambda^\nu,\beta^\nu}(x^\nu),N_C(x^\nu))\searrow0\) with \(\lambda^{\nu}\searrow 0\) and \(\limsup\beta^{\nu}<1\), characterizes a C-stationary point of~\eqref{MPCC_simple}.

		\begin{theorem}[C-stationary point]\label{theorem_C_stationarity}
		Let \(\set{\lambda^{\nu}}\subset\R_{++}\) and \(\set{\beta^{\nu}}\subset(0,1)\) be sequences such that \(\lambda^{\nu}\searrow 0\) and \(\limsup_{\nu\in\N}\beta^{\nu}<1\).
		Consider a sequence \(\set{x^\nu}\subset\R^n\) such that \(\dist (-\nabla s_{\lambda^\nu,\beta^\nu}(x^\nu),N_C(x^\nu))\searrow0\), and suppose that \(\set{x^{\nu}}\to x\), where \(x\) satisfies \(F_i(x)\in D\) for all \(i\in\set{1,\ldots,p}\), and \BCCQ{} holds at \(x\).
		Let \(\y^\nu=(y^\nu_{1},\ldots,y^\nu_{p})\) with \(y_i^\nu=\tfrac{1}{\lambda^\nu}R_{\beta^\nu}(F_i(x^\nu))\).
		Then, we have \(\set{\y^{\nu}}\to\y\), and every limit pair \((x,\y)\) is a C-stationary primal-dual pair of~\eqref{MPCC_simple}, i.e.,
		\[
		0
		\in
		\nabla f(x)
		+
		\sum_{i=1}^{p}JF_{i}{(x)}^{\top}y_i
		+
		N_{C}(x)
		\quad
		\text{with}
		\quad
		y_i\in N^{\rm C}_D(F_{i}(x))\quad
		\forall i\in\set{1,\ldots,p}.
		\]
		\end{theorem}
		\begin{proof}
		\cref{lemma_bounded_multiplier} implies that \(\set{\y^\nu}\) is bounded under \BCCQ\@, thus we have \(\set{\y^{\nu}}\to \y\) as \(\set{x^{\nu}}\to x\).
		On the other hand, to demonstrate that the approximating problems, \(\minimize_{x\in C}s_{\lambda^\nu,\beta^\nu}(x)\), eventually become accurate relative to~\eqref{MPCC_simple} in terms of stationary points, under the light of \cref{definition_WSA} we need to prove that the surrogate function epi-converges to the original function, which is already established in~\eqref{epi_converge_s}.
		In addition, it is necessary to verify that the generalized equation corresponding to~\eqref{LL_KKT} at least satisfies the inclusion~\eqref{weakly_approximating}. To this end, let the limit
		\[
		\begin{aligned}
		v
		&\in
		\Limsup_{x^{\nu}\to x,\lambda^{\nu}\searrow0}
		\nabla
		s_{\lambda^{\nu},\beta^\nu}(x^\nu)
		+
		\partial\delta_C(x^\nu)
		\\
		&=
		\Limsup_{x^{\nu}\to x,\lambda^{\nu}\searrow0}
		\nabla f(x^\nu)
		+
		\tfrac{1}{\lambda^\nu}
		\sum_{i=1}^{p}
		JF_{i}{(x^\nu)}^\top
		R_{\beta^\nu}(F_{i}(x^{\nu}))
		+N_C(x^\nu)
		\end{aligned}
		\]
		be given under a sequence \(\set{\beta^{\nu}}\subset(0,1)\) such that \(\limsup_{\nu\in\N}\beta^{\nu}<1\).
		Then, the sequences \(\set{\lambda^{\nu}}\), \(\set{\beta^{\nu}}\), \(\set{x^\nu}\) and \(\set{\y^\nu}\) satisfy
		\[
		S_{\lambda^{\nu},\beta^\nu}(x^\nu):=
		\nabla f(x^\nu)
		+
		\sum_{i=1}^{p}
		JF_{i}{(x^\nu)}^\top y_i^{\nu}
		+
		N_C(x^\nu)
		\to v.
		\]
		\cref{lemma_y_unbounded_bounded} establishes that every limit point of \(\set{y_i^\nu}\) belongs to \(N_D^{\rm C}(F_i(x))\), i.e., \(y_i\in N^{\rm C}_D(F_{i}(x))\) for all \(i\in\set{1,\ldots,p}\). Thus, we have
		\[
		S(x):=
		\nabla f(x)
		+
		\sum_{i=1}^{p}
		JF_{i}{(x)}^\top N_D^{\rm C}(F_i(x))
		+
		N_C(x)
		\ni v,
		\]
		which implies that \(\Limsup_{\nu\to\infty}\graph S_{\lambda^{\nu},\beta^\nu}\subseteq \graph S\).
		Therefore, the pairs \(\set{(s_{\lambda^{\nu},\beta^\nu},S_{\lambda^{\nu},\beta^\nu})}\) are a weakly consistent approximation of \((s, S)\) based on \cref{definition_WSA}.
		Then, the condition \(\dist (-\nabla s_{\lambda^\nu,\beta^\nu}(x^\nu),N_C(x^\nu))\searrow0\) results in all limit points of \(\set{x^\nu}\) satisfying \(S(x)\ni0\) according to \cite[Proposition 2.3]{Royset_2023}.
		In this way, we conclude that \((x,\y)\) is a C-stationary primal-dual pair of~\eqref{MPCC_simple}.
		\end{proof}
		
		We now turn our attention to the M-stationarity, which hinges on the following important observation.
		
		\begin{lemma}\label{thm:Ncone}%
		Fix \(\beta\in(0,1)\) and let \(R_\beta\) be as in~\eqref{R_beta} and \(\mathrm{T}_\beta\subset\R^2\) as in \cref{fig:areas}. For every \(z\in\R^2\setminus\interior \mathrm{T}_\beta\) it holds that \(R_\beta(z)\in\hat{N}_D(\bar z)\), where \(\bar z=\Pi_D(z)\).
		\end{lemma}
		\begin{proof}
		Note that \(\bar z\) as in the statement is well defined.
		Indeed, \(\Pi_D\) is single-valued everywhere except on the diagonal within the strictly positive orthant, which is entirely contained in \(\interior \mathrm{T}_\beta\) and thus not considered here.
		With \(\mathrm{O}_{-}\) and \(\mathrm{H}_{\beta}^{\pm }\) as in \cref{fig:areas}, the three possible scenarios are:
		\begin{itemize}
		\item
			\(z\in\mathrm{O}_{-}\).
			In this case, \(\bar z=(0,0)\) and \(R_\beta(z)\leq (0,0)\), as needed.
		\item
			\(z\in\mathrm{H}_{\beta}^{+}\setminus \mathrm{O}_{-}\).
			In this case, \(\bar z=(0,z_2)\) and \((R_\beta(z))_2=0\), as needed.
		\item
			\(z\in\mathrm{H}_{\beta}^{-}\setminus \mathrm{O}_{-}\).
			In this case, \(\bar z=(z_1,0)\) and \((R_\beta(z))_1=0\), as needed.
		\qedhere
		\end{itemize}
		\end{proof}
		
		The region \(\interior \mathrm{T}_\beta\) turns out to be the problematic one, since all components of \(R_\beta\) are strictly positive there.
		Indeed, \(\frac{1}{\beta(2-\beta)}(z_1+z_2)-\frac{1}{\beta}z_1\leq0\) if and only if \(z_2\leq(1-\beta)z_1\), and similarly \(\frac{1}{\beta(2-\beta)}(z_1+z_2)-\frac{1}{\beta}z_2\leq0\) if and only if \(z_2\geq\frac{1}{1-\beta}z_1\).
		Therefore, in order to establish the M-stationarity at \(x\), we need to assume that there is no constraint \(F_i(x)\) that falls into the area \(\interior \mathrm{T}_{\beta}\);
		in other words, all constraints are well-behaved.
		In this case, the cone \(N_D^{\rm C}(F_i(x))\) can be replaced by \(N_D(F_i(x))\) in~\cref{definition_M_C},
		and \BCQ+, which is weaker than \BCCQ\@ (see~\eqref{relation_CQs}), suffices for guaranteeing the boundedness of \(\set{\y^\nu}\).

		\begin{theorem}[AM- and M-stationary point]\label{theorem_M}
		Let \(\set{\lambda^{\nu}}\subset\R_{++}\) and \(\set{\beta^{\nu}}\subset(0,1)\) be sequences such that \(\lambda^{\nu}\searrow 0\) and \(\limsup_{\nu\in\N}\beta^{\nu}<1\).
		Consider a sequence \(\set{x^\nu}\subset\R^n\) such that \(\dist (-\nabla s_{\lambda^\nu,\beta^\nu}(x^\nu),N_C(x^\nu))\searrow0\), and suppose that \(\set{x^{\nu}}\to x\), where \(x\) satisfies \(F_i(x)\in D\), and \(F_i\) is well-behaved for all \(i\in\set{1,\ldots,p}\).
		Let \(\y^\nu=(y^\nu_{1},\ldots,y^\nu_{p})\) with \(y_i^\nu=\tfrac{1}{\lambda^\nu}R_{\beta^\nu}(F_i(x^\nu))\).
		Then, the point \(x\) is AM-stationary.
		Furthermore, any such limit point \(x\) at which \BCQ+ is satisfied is actually M-stationary.
		\end{theorem}
		
		\begin{proof}
		By following the method in \cref{lemma_bounded_multiplier}, we know that the sequence \(\set{y_i^{\nu}}\) with \(y_i^\nu=\tfrac{1}{\lambda^\nu}R_{\beta^\nu}(F_i(x^\nu))\) is bounded if \BCQ+ is satisfied at \(x\).
		On the other hand, we have
		\begin{equation}\label{eq:APO}
			\dist
			\begin{pmatrix}
			-\nabla f(x^\nu)
			-
			\sum_{i=1}^{p}JF_{i}{(x^\nu)}^{\top}y_i^\nu
			,
			N_C(x^\nu)
			\end{pmatrix}
			\searrow 0.
		\end{equation}
		For any \(i\in\set{1,\ldots,p}\) we consider two cases:
		\begin{itemize}
			\item
			\(F_i(x)\neq(0,0)\).
			For every \(\nu\) let
			\[
				z_i^\nu
			\in
				\Pi_{D}(F_i(x^\nu)).
			\]
			Since \(\limsup\beta^\nu<1\) one has that \(z_i^\nu\notin \interior \mathrm{T}_{\beta^\nu}\) for all \(\nu\) large enough, and \cref{thm:Ncone} implies that
			\begin{equation}\label{eq:yNcone}
				y_i^\nu\in N_D(z_i^\nu)
			\end{equation}
			for all such \(\nu\).
			Furthermore, since \(F_i(x)\in D\), clearly one has
			\begin{equation}\label{eq:APF}
				F_i(x^\nu)-z_i^\nu\to 0.
			\end{equation}
			All properties~\eqref{eq:APO},~\eqref{eq:yNcone},~\eqref{eq:APF} and the feasibility of \(x\) combined show that all the conditions of \cref{def:AKKT} are satisfied, hence that \(x\) is AM-stationary.
		
			\item
			\(F_i(x)=(0,0)\).
			If \(F_i(x^\nu)\notin\interior \mathrm{T}_{\beta^\nu}\) holds for all \(\nu\) large enough, then the arguments appealing to \cref{thm:Ncone} as in the previous case can be replicated and the same conclusion follows.
		\end{itemize}
		In case \(\set{\y^\nu}\) is bounded, this being ensured by \BCQ+ at \(x\) as shown above, M-stationarity follows.
		\end{proof}

		\subsection{Special Instance Assuming Feasible Initialization}
			Without further assumptions, limit points may be infeasible, even in convex composite optimization, such as the Gauss-Newton method for solving nonlinear systems of equations~\cite{Burke_2013}.
			Therefore, it is reasonable to introduce a further restriction on how to generate \(\set{x^\nu}\).
			Specifically, if it is possible to select a feasible starting point, then one can proceed along the lines of~\cref{algorithm1}.

			\begin{algorithm}[h]
			\caption{LL-based homotopy approach with feasible initialization.}\label{algorithm1}
			\begin{algorithmic}[1]
			\itemsep=3pt%
			\renewcommand{\algorithmicindent}{0.3cm}%

			\item[\textbf{Initialization:}]

			Let \(\nu=1\) and \(x^0\in C\) be such that \(F_i(x^0)\in D\) for all \(i\in\set{1,\ldots,p}\).

			\item[\textbf{Repeat:}]

			\State Choose \(\epsilon^{\nu},\lambda^\nu>0\) and \(\beta^\nu\in(0,1)\), and set \(s_{\lambda^{\nu},\beta^{\nu}}:=f+\sum_{i=1}^{p}\env_{\lambda^\nu,\beta^\nu \lambda^\nu}\delta_{D}\circ F_{i}\).

			\State Find an approximate solution \(x^\nu\) to the problem \(\minimize_{x\in C}s_{\lambda^{\nu},\beta^{\nu}}(x)\) such that

			\begin{itemize}
			\item \(s_{\lambda^{\nu},\beta^{\nu}}(x^\nu)\leq\min\set{s_{\lambda^{\nu},\beta^{\nu}}(x^{\nu-1}),f(x^0)}\) and
			\item \(\dist (-\nabla s_{\lambda^{\nu},\beta^{\nu}}(x^\nu),N_C(x^\nu))\leq\epsilon^{\nu}\).
			\end{itemize}

			\State Let \(\nu\leftarrow\nu+1\).

			\end{algorithmic}
			\end{algorithm}

			\begin{theorem}\label{theorem_feasibility}
			Consider the iterates generated by \cref{algorithm1} with \(\epsilon^{\nu}\searrow 0\), \(\lambda^{\nu}\searrow 0\) and \(\sup_{\nu\in\N}\beta^{\nu}<1\).
			If \(f\) is bounded from below, then the following hold:
			\begin{enumerate}
			\item
				Any cluster point of \(\set{x^\nu}\) is feasible.
				\item
				Any cluster point of \(\set{x^\nu}\) at which \BCCQ\@ is satisfied is C-stationary.
			\item
				Any cluster point of \(\set{x^\nu}\) is AM-stationary as long as \(F_i\) is well-behaved for all \(i\in\set{1,\ldots,p}\).
				Any such cluster point at which \BCQ+\@ is satisfied is actually M-stationary.
			\end{enumerate}
			\end{theorem}

			\begin{proof}
			According to the sandwich property \cref{fact_sandwich}, we have
			\[
				f(x^0)
				\geq
				s_{\lambda^{\nu},\beta^{\nu}}(x^\nu)
			\geq
				f(x^\nu)
				+
				\tfrac{1}{2\lambda^\nu}
				\sum_{i=1}^{p}
				\dist {(F_i(x^\nu),D)}^2.
			\]
			Suppose that \(x\) is a cluster point such that \(F_i(x)\notin D\) for some \(i\in\set{1,\ldots,p}\).
			Let \(J\in\mathcal{N}_{\infty}^{\sharp}\) be such that \(\set{x^\nu}\to_J x\), then \(\tfrac{1}{2\lambda^\nu}\dist {(F_i(x^\nu),D)}^2\to_J+\infty\) for such \(i\), which gives a contradiction \(f(x^\nu)\to_J-\infty\)
			(otherwise, the function \(f\) would be unbounded below).
			Thus, we have \(
			F_i(x)\in D\) for all \(i\in\set{1,\ldots,p}\).
			Claims (ii) and (iii) come from \cref{theorem_C_stationarity,theorem_M}.
			\end{proof}

	\section{Practical Algorithm and Complexity Analysis}\label{sec_finite}
		The previous section analyzes the asymptotic behavior when \(\lambda^\nu\searrow0\) and \(\epsilon^{\nu}\searrow 0\).
		However, a practical algorithm always requires termination in a finite number of iterations to meet a user-defined accuracy tolerance.
		In this section, we present a practical variant of \cref{algorithm1} and analyze its worst-case complexity to reach an approximate stationary point.
		
		\subsection{Assumptions and Auxiliary Lemmas}
			At the beginning of this paper, we assume that the solution set of~\eqref{MPCC_simple} is nonempty and all functions are smooth.
			In this section, to establish the worst-case complexity, we additionally assume the compactness of the level set of the objective function \(f\) over the set \(C\), as stated below.

			\begin{assumption}\label{assumption_level_set}
			For \(\alpha\in\R\), the level set \(\lev_{\leq\alpha}[f;C]:=\set{x\in C}[f(x)\leq\alpha]\) is compact (possibly empty).
			\end{assumption}

			\Cref{assumption_level_set} implies that there exists $\underline{f}\in\R$ such that $f(x)\geq\underline{f}$ for all $x\in C$.
			In addition, we assume that the constraint function \(F_{i}\) is Lipschitz continuous.
			Together with the previously stated smoothness assumption, these conditions are formalized as follows:

			\begin{assumption}\label{assumption_Lipschitz}
			The constraint function $F_{i}:\R^{n}\to\R^{2}$ is $K_{i}$-Lipschitz continuous and its Jacobian \(JF_i:\R^{n}\to\R^{2\times n}\) is $L_{i}$-Lipschitz continuous for all \(i\in\set{1,\ldots,p}\).
			Furthermore, the objective gradient $\nabla f:\R^{n}\to\R^n$ is $L_{f}$-Lipschitz continuous.
			\end{assumption}

			As established in the previous section, the original problem~\eqref{MPCC_simple} admits a smoothed approximation of the form: \(\minimize_{x\in C}s_{\lambda,\beta}(x)\).
			The following lemma establishes the property of $s_{\lambda,\beta}$.

			\begin{lemma}\label{lemma_level_set}
			Let~\cref{assumption_level_set} hold and \(s_{\lambda,\beta}\) be as in~\eqref{function_s}, choose \(\lambda>0\), \(\beta\in(0,1)\) and \(\tilde{x}\in C\),
			then the level set
			\[
			\lev_{\leq s_{\lambda,\beta}(\tilde{x})}[s_{\lambda,\beta};C]
			=
			\set{x\in C}[s_{\lambda,\beta}(x)\leq s_{\lambda,\beta}(\tilde{x})]
			\]
			is compact.
			\end{lemma}

			\begin{proof}
			Let \(x\in\lev_{\leq s_{\lambda,\beta}(\tilde{x})}[s_{\lambda,\beta};C]\).
			Using the fact that \(r_{\beta}\) is nonnegative, as established in~\eqref{r_beta_nonnegative}, we obtain
			\[
			f(x)
			\leq
			f(x)
			+
			\tfrac{1}{\lambda}
			\sum_{i=1}^{p}
			r_{\beta}(F_i(x))
			=
			s_{\lambda,\beta}(x)
			\leq
			s_{\lambda,\beta}(\tilde{x}),
			\]
			implying \(x\in\lev_{\leq s_{\lambda,\beta}(\tilde{x})}[f;C]\), and then \(\lev_{\leq s_{\lambda,\beta}(\tilde{x})}[s_{\lambda,\beta};C]\subseteq\lev_{\leq s_{\lambda,\beta}(\tilde{x})}[f;C]\).
			According to~\cref{assumption_level_set} and the continuity of $f$, $F_{i}$ and $r_{\beta}$, we know that \(\lev_{\leq s_{\lambda,\beta}(\tilde{x})}[s_{\lambda,\beta};C]\) is compact.
			\end{proof}

			The following lemma establishes an important stopping criterion for our algorithm, ensuring that the violation of CCs remains bounded.

			\begin{lemma}\label{lemma_stopping_criterion}
			Let \(r_\beta\) be as in~\eqref{r_beta} and \(F_i=(G_i,H_i)\).
			Suppose \(\beta\in(0,1)\) and \(\epsilon>0\).
			If \(x\) is such that
			\begin{equation}\label{stopping_criterion}
			\sum_{i=1}^{p}
			r_\beta(F_i(x))
			\leq
			\frac{\epsilon^{2}}{2},
			\end{equation}
			then the complementarity constraints are satisfied approximately by
			\begin{equation}\label{stopping_criterion_norm}
			\begin{Vmatrix}
				\min\set{
				G(x),
				H(x)
				}
			\end{Vmatrix}
			\leq
			\epsilon,
			\end{equation}
			where the minimum is taken componentwise.
			\end{lemma}

			\begin{proof}
			For notational conciseness, we denote \(\sum_{\mathrm{H}_{\beta}^{\pm}}\) as the sum over the indices \(i=1,\dots,p\) for which \(F_i(x)\in\mathrm{H}_{\beta}^{\pm}\), and similarly for the other two areas of \cref{fig:areas}.
			Building on this notation, we can reformulate~\eqref{stopping_criterion} equivalently as
			\begin{align*}
				\sum_{\interior\mathrm{O}_{-}}\tfrac{1}{2(1-\beta)}
				\begin{pmatrix}
					G_i{(x)}^{2}+H_i{(x)}^{2}
				\end{pmatrix}
			&
				{}+
				\sum_{\mathrm{H}_{\beta}^{\pm}}\tfrac{1}{2(1-\beta)}
				\min\set{G_i(x),H_i(x)}^{2}
			\\
			&
				{}+
				\sum_{\interior\mathrm{T}_{\beta}}
				r_\beta(F_i(x))
				\leq\frac{\epsilon^{2}}{2}.
			\end{align*}
			When \(F_i(x)\in\interior\mathrm{O}_-\), the value of \(G_{i}{(x)}^{2}+H_{i}{(x)}^{2}\) can be bounded from below by \(\min\set{G_{i}{(x)},H_{i}{(x)}}^2\).
			Furthermore, given that \(\nicefrac{1}{(1-\beta)}>1\) for \(\beta\in(0,1)\), it follows that
			\[
			\sum_{\interior\mathrm{O}_{-}\cup\mathrm{H}_{\beta}^{\pm}}
			\tfrac{1}{2}
			\min\set
			{
				G_{i}{(x)},
				H_{i}{(x)}
			}^2
			+
			\sum_{\interior\mathrm{T}_{\beta}}
			r_\beta(F_i(x))
			\leq\frac{\epsilon^{2}}{2}.
			\]
			As for the case where \(F_i(x)\in\interior\mathrm{T}_{\beta}\), combining \cref{fact_sandwich} and the formulation of Moreau envelope~\eqref{moreau_compl}, we have
			\[
			\sum_{\interior\mathrm{T}_{\beta}}
			\tfrac{1}{2}
			\min
			\set{
				G_{i}{(x)},
				H_{i}{(x)}
			}^2
			=
			\sum_{\interior\mathrm{T}_{\beta}}
			\lambda
			\env_\lambda\delta_{D}(F_i(x))
			\leq
			\sum_{\interior\mathrm{T}_{\beta}}
			r_\beta(F_i(x)),
			\]
			which leads to
			\[
			\sum_{\interior\mathrm{O}_{-}\cup\mathrm{H}_{\beta}^{\pm}\cup\interior\mathrm{T}_{\beta}}
			\tfrac{1}{2}
			\min
			\set
			{
				G_{i}{(x)},
				H_{i}{(x)}
			}^2
			=
			\sum_{i=1}^{p}
			\tfrac{1}{2}
			\min
			\set
			{
				G_{i}{(x)},
				H_{i}{(x)}
			}^2
			\leq\frac{\epsilon^{2}}{2}.
			\]
			In this way, the claimed result is obtained.
			\end{proof}

			The following lemma establishes an upper bound of $\nicefrac{1}{\lambda}$, which is determined by the violation of CCs.

			\begin{lemma}\label{lemma_mu_lowerbound}
			Let \(\lambda>0\), \(\beta\in(0,1)\) and \(\epsilon>0\),
			and let \(r_\beta,s_{\lambda,\beta}\) be as in~\eqref{r_beta} and~\eqref{function_s}.
			Choose \(x,\tilde{x}\in C\) such that
			\[
			s_{\lambda,\beta}(x)
			\leq
			s_{\lambda,\beta}(\tilde{x})
			\quad
			\text{and}
			\quad
			\sum_{i=1}^{p}
			r_{\beta}(F_{i}(\tilde{x}))
			\leq
			\frac{\epsilon^{2}}{4}.
			\]
			If the condition (\ref{stopping_criterion}) is not satisfied by \(x\), i.e., \(
			\sum_{i=1}^{p}
			r_{\beta}(F_{i}(x))
			>
			\nicefrac{\epsilon^{2}}{2},
			\)
			then the parameter $\nicefrac{1}{\lambda}$ is  upper bounded by
			\[
			\frac{1}{\lambda}
			<
			\frac{4(f(\tilde{x})-f(x))}{\epsilon^{2}}.
			\]
			\end{lemma}

			\begin{proof}
			From the definition of \(s_{\lambda,\beta}\) in~\eqref{function_s} and the condition \(s_{\lambda,\beta}(x)\leq s_{\lambda,\beta}(\tilde{x})\), it follows that
			\begin{align*}
				f(x)
				+
				\tfrac{1}{\lambda}
				\sum_{i=1}^{p}
				r_{\beta}(F_{i}(x))
			=
				s_{\lambda,\beta}(x)
			\leq{} &
				s_{\lambda,\beta}(\tilde{x})
			\\
			={} &
				f(\tilde{x})
				+
				\tfrac{1}{\lambda}
				\sum_{i=1}^{p}
				r_{\beta}(F_{i}(\tilde{x}))
				\leq
				f(\tilde{x})
				+
				\tfrac{\epsilon^{2}}{4\lambda}.
			\end{align*}
			Then, the violation of~\eqref{stopping_criterion} implies
			\[
			\begin{aligned}
			f(x)+\frac{\epsilon^{2}}{2\lambda}
			<
			f(x)
			+
			\tfrac{1}{\lambda}
			\sum_{i=1}^{p}
			r_{\beta}(F_{i}(x))
			\leq
			f(\tilde{x})
			+
			\tfrac{\epsilon^{2}}{4\lambda},
			\end{aligned}
			\]
			which leads to the result.
			\end{proof}

		\subsection{Algorithm and Outer Complexity}
			Let us now consider \cref{algorithm2}, which is a practical variant of \cref{algorithm1}.
			In this version, the tolerance \(\epsilon\) is held constant rather than decreasing to 0, and similarly, the parameter \(\beta\) is fixed.
			The starting point \(x^0\) is near-feasible rather than strictly feasible.
			Furthermore, a specific stopping criterion and warm-start strategy are introduced.
			This subsection also analyzes the worst-case complexity of the outer iterations, i.e., the number of subproblems that need to be solved.
			The details of the inner solver will be discussed in the next subsection.

			\begin{algorithm}[h]
			\caption{LL-based homotopy approach with near-feasible initialization.}\label{algorithm2}
			\begin{algorithmic}[1]
			\itemsep=3pt%
			\renewcommand{\algorithmicindent}{0.3cm}%
			\item[\textbf{Initialization:}]
			Let \(\nu=1\). Given \(\epsilon>0\), \(\lambda^0>0\), \(\beta\in(0,1)\) and \(x^0\in C\) such that
			\(
			\sum_{i=1}^{p}
			r_{\beta}(F_{i}(x^0))
			\leq
			\nicefrac{\epsilon^{2}}{4}.
			\)
			\item[\textbf{Repeat:}]
			\State  Choose \(\lambda^{\nu}\) and set \(s_{\lambda^{\nu},\beta}:=f+\sum_{i=1}^{p}\frac{1}{\lambda^\nu}
			r_{\beta}\circ F_{i}\).
			\State Find an approximate solution \(x^\nu\) to the problem \(\minimize_{x\in C}s_{\lambda^{\nu},\beta}(x)\) such that
			\begin{itemize}
				\item \(s_{\lambda^{\nu},\beta}(x^\nu)\leq\min\set{s_{\lambda^{\nu},\beta}(x^{\nu-1}),s_{\lambda^{\nu},\beta}(x^0)}\) and
				\item \(\dist (-\nabla s_{\lambda^{\nu},\beta}(x^\nu),N_C(x^\nu))\leq\epsilon\)
			\end{itemize}
			by using an inner solver from the starting point
			\(x^{\nu,0}\coloneqq\arg\min\set{s_{\lambda^{\nu},\beta}(\hat x)}[\hat{x}\in\set{x^{\nu-1},x^{0}}]\).
			\State Let \(\nu\leftarrow\nu+1\).
			\item[\textbf{Until:} The stopping criterion
			\(
				\sum_{i=1}^{p}
			r_{\beta}(F_{i}(x^\nu))
				\leq
				\nicefrac{\epsilon^{2}}{2}
			\)
			is satisfied.]
			\end{algorithmic}
			\end{algorithm}

			\begin{remark}
				Step 2 of \cref{algorithm2} is realistic for first-order methods.
				The descent condition (the first requirement) follows directly from standard convergence guarantees for descent methods; see, for instance,~\cite[Theorem 10.15]{Beck_2017}.
				The approximate stationarity condition (the second requirement) is widely employed as a practical stopping criterion.
				This criterion can be satisfied within at most \(\mathcal{O}(\epsilon^{-2})\) iterations by standard first-order methods, such as a projected gradient method~\cite{Lan_2024}, or an accelerated gradient method~\cite{Ghadimi_2019}.
				In practice, quasi-Newton methods are well known to often outperform the aforementioned methods.
				Therefore, we adopt the off-the-shelf L-BFGS-B solver~\cite{Zhu_1997} in our numerical experiments presented in \cref{sec_experiments}.
			\end{remark}

			Before presenting the main result of this subsection, we introduce the following operator:
			\[
				\Pi_{D}^{\rm C,\beta}(z)
				\coloneqq
				\begin{cases}
					\Pi_{D}(z) & \text{if } z\in\mathrm{O}_{-}\cup\mathrm{H}_{\beta}^{\pm}
				\\
					(0,0) & \text{if } z\in\interior\mathrm{T}_{\beta},
				\end{cases}
			\]
			which projects the point \(z\) to the origin when \(z\in\interior\mathrm{T}_{\beta}\), rather than to one of the two branches of \(D\).
			The following theorem provides an upper bound on the number of outer iterations in~\cref{algorithm2},
			and shows that the returned point is an approximate stationary point of~\eqref{MPCC_simple}.

			\begin{theorem}[Outer complexity]\label{theorem_approximate_C}
			Let~\cref{assumption_level_set} hold and consider the iterates generated by \cref{algorithm2} with \(\lambda^{\nu}\) being updated as \(\lambda^{\nu}=\rho\lambda^{\nu-1}\) with \(\rho\in(0,1)\).
			Then, \cref{algorithm2} terminates in at most
			\begin{equation}\label{outer_complexity_bound}
			\bar{\nu}<
			1+
			\log_{\frac{1}{\rho}}
			\begin{pmatrix}
				\frac{4\lambda^{0}(f(x^{0})-\underline{f})}{\epsilon^{2}}
			\end{pmatrix}
			\end{equation}
			iterations.
			Furthermore, the last iterate \(x^{\bar{\nu}}\) satisfies the approximate C-sta\-tion\-ary condition
			\[
			\begin{cases}
			\dist
			\begin{pmatrix}
				-
				\nabla f(x^{\bar{\nu}})
				-
				\sum_{i=1}^{p}
				JF_{i}{(x^{\bar{\nu}})}^{\top}y_{i}^{\bar{\nu}},\
				N_{C}(x^{\bar{\nu}})
			\end{pmatrix}
			\leq
			\epsilon
			\\[5pt]
			\begin{Vmatrix}
				\min\set{
				G(x^{\bar{\nu}}),
				H(x^{\bar{\nu}})
				}
			\end{Vmatrix}
			\leq
			\epsilon
			\\[5pt]
			y_{i}^{\bar{\nu}}\in N_D^{\rm C}(z_i^{\bar{\nu}})
			\end{cases}
			\]
			for all \(i\in\set{1,\ldots,p}\),
			where \(z_i^{\bar{\nu}}=\Pi_{D}^{\rm C,\beta}(F_{i}(x^{\bar{\nu}}))\) and \(y_i^{\bar{\nu}}=\tfrac{1}{\lambda^{\bar{\nu}}}R_{\beta}(F_i(x^{\bar{\nu}}))\).

			\end{theorem}

			\begin{proof}
			If~\cref{algorithm2} is terminated by meeting \(
				\sum_{i=1}^{p}
			r_{\beta}(F_{i}(x^{\bar{\nu}}))
				\leq
				\nicefrac{\epsilon^{2}}{2}
			\),
			the following inequality holds at the iteration \(\bar{\nu}-1\):
			\[
			\sum_{i=1}^{p}
			r_{\beta}(F_{i}(x^{\bar{\nu}-1}))
				>
				\frac{\epsilon^{2}}{2}
			\]
			(otherwise, the algorithm would have terminated at iteration \(\bar{\nu}-1\)).
			Recalling that \cref{assumption_level_set} implies the existence of \(\underline{f}\), and applying~\cref{lemma_mu_lowerbound} along with the condition \(s_{\lambda^{\bar\nu-1},\beta}(x^{\bar\nu-1})\leq s_{\lambda^{\bar\nu-1},\beta}(x^0)\), we obtain
			\[
			\frac{1}{\lambda^{\bar\nu-1}}
			<
			\frac{4(f(x^{0})-f(x^{\bar\nu-1}))}{\epsilon^{2}}
			\leq
			\frac{4(f(x^{0})-\underline{f})}{\epsilon^{2}}.
			\]
			On the other hand, the tuning strategy indicates \(\lambda^{\bar\nu-1}=\rho^{\bar\nu-1}\lambda^{0}\), which leads to the following:
			\[
			\frac{1}{\rho^{\bar\nu-1}}
			<
			\frac{4\lambda^{0}(f(x^{0})-\underline{f})}{\epsilon^{2}}.
			\]
			Since $\rho\in(0,1)$, applying $\log_{1/\rho}$ to both sides yields~\eqref{outer_complexity_bound}.
			Moreover, the above discussion implies that if
			\[
			\bar{\nu}\geq
			1+
			\log_{\frac{1}{\rho}}
			\begin{pmatrix}
				\frac{4\lambda^{0}(f(x^{0})-\underline{f})}{\epsilon^{2}}
			\end{pmatrix},
			\]
			then the stopping criterion~\eqref{stopping_criterion} must be satisfied, i.e.,
			\[
			\sum_{i=1}^{p}
			r_{\beta}(F_{i}(x^{\bar{\nu}-1}))
				\leq
				\frac{\epsilon^{2}}{2}.
			\]
			Therefore, \cref{algorithm2} terminates finitely.
			In turn, this also shows that the complexity bound~\eqref{outer_complexity_bound} must be true, for otherwise the algorithm would have terminated earlier.

			Let us now consider the approximate C-stationarity.
			The first-order condition follows directly from~\cref{algorithm2},
			while the approximate satisfaction of the CCs is ensured by~\cref{lemma_stopping_criterion}.
			Regarding the inclusion \(y_{i}^{\bar{\nu}}\in N_D^{\rm C}(z_i^{\bar{\nu}})\), \cref{thm:Ncone} implies that \(R_\beta(F_{i}(x^{\bar{\nu}}))\in\hat{N}_D(z_i^{\bar{\nu}})\) with \(z_i^{\bar{\nu}}=\Pi_D(F_{i}(x^{\bar{\nu}}))\) when \(F_{i}(x^{\bar{\nu}})\notin\interior\mathrm{T}_{\beta}\).
			In case
			\(F_{i}(x^{\bar{\nu}})\in\interior\mathrm{T}_{\beta}\), we have \(\Pi_{D}^{\rm C,\beta}(F_{i}(x^{\bar{\nu}}))=(0,0)\).
			Hence, it follows that \(R_\beta(F_{i}(x^{\bar{\nu}}))\in N_D^{\rm C}(z_i^{\bar{\nu}})\) with \(z_i^{\bar{\nu}}=\Pi_D^{\rm C,\beta}(F_{i}(x^{\bar{\nu}}))\).
			Since \(\lambda^{\bar{\nu}}>0\) (as \cref{algorithm2} terminates finitely), and due to the property of the cone, we conclude that \(y_i^{\bar{\nu}}=\tfrac{1}{\lambda^{\bar{\nu}}}R_{\beta}(F_i(x^{\bar{\nu}}))\in N_D^{\rm C}(z_i^{\bar{\nu}})\) with \(z_i^{\bar{\nu}}=\Pi_D^{\rm C,\beta}(F_{i}(x^{\bar{\nu}}))\) for all \(F_i(x^{\bar{\nu}})\in\R^2\).
			\end{proof}

			\begin{remark}
			If \(F_i\) is well-behaved for all \(i\in\set{1,\ldots,p}\), then analogously to \cref{theorem_M}, we refer to $x^{\bar{\nu}}$ as an approximate M-stationary point.
			Furthermore, it is worth noting that the definitions of approximate C/M-stationarity do not require any CQs.
			Under \BCCQ/\BCQ+, in the limiting setting as \(\epsilon\searrow0\) and \(\lambda\searrow0\), they turn to be standard C/M-stationarity as in~\cref{definition_M_C}.
			\end{remark}

		\subsection{Total Complexity}
			To determine the total complexity of \cref{algorithm2}, it is necessary to estimate the computational cost of the inner solver for each subproblem.
			To this end, we first analyze a smoothness property of \(s_{\lambda,\beta}\).

			\begin{lemma}\label{lemma_Lipschitz}
			Let~\cref{assumption_Lipschitz,assumption_level_set} hold, and let \(s_{\lambda,\beta}\) be as in~\eqref{function_s}
			with \(\lambda>0\), \(\beta\in(0,1)\) and \(x^0\in C\).
			Then, the gradient \(\nabla s_{\lambda,\beta}\) is Lipschitz continuous on \(\conv (\lev_{\leq s_{\lambda,\beta}(x^{0})}[s_{\lambda,\beta};C])\) with modulus
			\begin{equation}\label{Lipschitz_sub_modulus_1}
			L_{\lambda}(x^{0})
			=
			L_{f}
			+
			\tfrac{1}{\lambda}
			\tilde{L}_{\lambda}(x^{0}),
			\end{equation}
			where
			\begin{equation}\label{Lipschitz_sub_modulus}
				\tilde{L}_{\lambda}(x^{0})
			=
				\sum_{i=1}^{p}
				\biggl[
					L_i
					\begin{pmatrix}
					L_\beta
					K_i
					D^0_\lambda
					+
					\begin{Vmatrix}
						R_{\beta}(F_i(x^0))
					\end{Vmatrix}
					\end{pmatrix}
					+
					L_\beta
					K_i
					\begin{pmatrix}
						L_i
						D^0_\lambda
						+
						\begin{Vmatrix}
						R_{\beta}(F_i(x^0))
						\end{Vmatrix}
					\end{pmatrix}
				\biggr]
			\end{equation}
			and the constant \(D^0_\lambda\) is defined in~\eqref{finite_value_D}.
			\end{lemma}

			\begin{proof}
			For all \(x,\tilde{x}\in\conv (\lev_{\leq s_{\lambda,\beta}(x^{0})}[s_{\lambda,\beta};C])\), we have
			\begin{subequations}
			\allowdisplaybreaks
			\begin{align*}
			&
			\begin{Vmatrix}
				\nabla s_{\lambda,\beta}(x)
				-
				\nabla s_{\lambda,\beta}(\tilde{x})
			\end{Vmatrix}
			\\
			={} &
			\begin{Vmatrix}
				\nabla f(x) - \nabla f(\tilde{x})
				+
				\tfrac{1}{\lambda}
				\sum_{i=1}^{p}
				\biggl[
					JF_i{(x)}^\top R_{\beta}(F_i(x))
					-
					JF_i{(\tilde{x})}^\top R_{\beta}(F_i(\tilde{x}))
				\biggr]
			\end{Vmatrix}
			\\
			\leq{} &
			L_{f}
			\begin{Vmatrix}
				x-\tilde{x}
			\end{Vmatrix}
			+
			\tfrac{1}{\lambda}
			\sum_{i=1}^{p}
			\begin{Vmatrix}
				JF_i{(x)}^\top R_{\beta}(F_i(x))
				-
				JF_i{(\tilde{x})}^\top R_{\beta}(F_i(\tilde{x}))
			\end{Vmatrix}
			\\
			\leq{} &
			L_{f}
			\begin{Vmatrix}
				x-\tilde{x}
			\end{Vmatrix}
			+
			\tfrac{1}{\lambda}
			\sum_{i=1}^{p}
			\begin{Vmatrix}
				JF_i{(x)}^\top R_{\beta}(F_i(x))
				-
				JF_i{(\tilde{x})}^\top R_{\beta}(F_i(x))
			\end{Vmatrix}
			\\
			&
				\hphantom{  L_{f}
			\begin{Vmatrix}
				x-\tilde{x}
			\end{Vmatrix}}
				+
				\tfrac{1}{\lambda}
				\sum_{i=1}^{p}
				\begin{Vmatrix}
				JF_i{(\tilde{x})}^\top R_{\beta}(F_i(x))
				-
				JF_i{(\tilde{x})}^\top R_{\beta}(F_i(\tilde{x}))
			\end{Vmatrix}
			\\
			\leq{} &
			L_{f}
			\begin{Vmatrix}
				x-\tilde{x}
			\end{Vmatrix}
			+
			\tfrac{1}{\lambda}
			\sum_{i=1}^{p}
			\begin{Vmatrix}
				JF_i{(x)}
				-
				JF_i{(\tilde{x})}
			\end{Vmatrix}
			\begin{Vmatrix}
				R_{\beta}(F_i(x))
			\end{Vmatrix}
			\\
			&
				\hphantom{  L_{f}
				\begin{Vmatrix}
					x-\tilde{x}
				\end{Vmatrix}}
				+
				\tfrac{1}{\lambda}
				\sum_{i=1}^{p}
				\begin{Vmatrix}
				JF_i{(\tilde{x})}
			\end{Vmatrix}
			\begin{Vmatrix}
				R_{\beta}(F_i(x))
				-
				R_{\beta}(F_i(\tilde{x}))
			\end{Vmatrix}
			\\
			\leq{} &
			L_{f}
			\begin{Vmatrix}
				x-\tilde{x}
			\end{Vmatrix}
			+
			\tfrac{1}{\lambda}
			\sum_{i=1}^{p}
			L_i
			\begin{Vmatrix}
				R_{\beta}(F_i(x))
			\end{Vmatrix}
			\begin{Vmatrix}
				x-\tilde{x}
			\end{Vmatrix}
			\\
			&
				\hphantom{  L_{f}
				\begin{Vmatrix}
					x-\tilde{x}
				\end{Vmatrix}}
				+
				\tfrac{1}{\lambda}
				\sum_{i=1}^{p}
				L_\beta
				\begin{Vmatrix}
				JF_i{(\tilde{x})}
			\end{Vmatrix}
			\begin{Vmatrix}
				F_i(x)
				-
				F_i(\tilde{x})
			\end{Vmatrix}
			\\
			\leq{} &
			L_{f}
			\begin{Vmatrix}
				x-\tilde{x}
			\end{Vmatrix}
			+
			\tfrac{1}{\lambda}
			\sum_{i=1}^{p}
			L_i
			\begin{Vmatrix}
				R_{\beta}(F_i(x))
			\end{Vmatrix}
			\begin{Vmatrix}
				x-\tilde{x}
			\end{Vmatrix}
			\\
			&
				\hphantom{  L_{f}
				\begin{Vmatrix}
					x-\tilde{x}
				\end{Vmatrix}}
				+
				\tfrac{1}{\lambda}
				\sum_{i=1}^{p}
				L_\beta
				K_i
				\begin{Vmatrix}
				JF_i{(\tilde{x})}
			\end{Vmatrix}
			\begin{Vmatrix}
				x-\tilde{x}
			\end{Vmatrix},
			\end{align*}
			\end{subequations}
			where the Lipschitz modulus \(L_\beta\) is given in \cref{proposition_globally_Lipschitz}.
			We now aim to establish upper bounds for \(\|R_{\beta}(F_i(x))\|\) and \(\|JF_i{(\tilde{x})}\|\), which will be derived based on the point \(x^0\).
			Specifically, we have
			\[
			\begin{aligned}
			\begin{Vmatrix}
				R_{\beta}(F_i(x))
			\end{Vmatrix}
			&\leq
			\begin{Vmatrix}
				R_{\beta}(F_i(x))
				-
				R_{\beta}(F_i(x^0))
			\end{Vmatrix}
			+
			\begin{Vmatrix}
				R_{\beta}(F_i(x^0))
			\end{Vmatrix}
			\\
			&
			\leq
			L_\beta
			K_i
			\begin{Vmatrix}
				x
				-
				x^0
			\end{Vmatrix}
			+
			\begin{Vmatrix}
				R_{\beta}(F_i(x^0))
			\end{Vmatrix}
			\end{aligned}
			\]
			and
			\[
			\begin{aligned}
				\begin{Vmatrix}
				JF_i{(\tilde{x})}
				\end{Vmatrix}
				&\leq
				\begin{Vmatrix}
				JF_i{(\tilde{x})}
				-
				JF_i{(x^0)}
				\end{Vmatrix}
				+
				\begin{Vmatrix}
				JF_i{(x^0)}
				\end{Vmatrix}
				\\
				&
				\leq
				L_i
				\begin{Vmatrix}
				\tilde{x}
				-
				x^0
				\end{Vmatrix}
				+
				\begin{Vmatrix}
				R_{\beta}(F_i(x^0))
				\end{Vmatrix}.
			\end{aligned}
			\]
			\Cref{lemma_level_set} establishes that the level set \(\lev_{\leq s_{\lambda,\beta}(x^0)}[s_{\lambda,\beta};C]\) is compact, then so is \(\conv (\lev_{\leq s_{\lambda,\beta}(x^{0})}[s_{\lambda,\beta};C])\), which implies that the constant
			\begin{equation}\label{finite_value_D}
			D^0_\lambda
			\coloneqq\sup_{x\in\conv (\lev_{\leq s_{\lambda,\beta}(x^{0})}[s_{\lambda,\beta};C])}
			\begin{Vmatrix}
				x-x^{0}
			\end{Vmatrix}
			\end{equation}
			is finite. By summarizing the above results, we conclude that
			\begin{subequations}
			\begin{align*}
				&
				\begin{Vmatrix}
				\nabla s_{\lambda,\beta}(x)
				-
				\nabla s_{\lambda,\beta}(\tilde{x})
				\end{Vmatrix}
			\\
			\leq{} &
				L_{f}
				\begin{Vmatrix}
				x-\tilde{x}
				\end{Vmatrix}
				+
				\tfrac{1}{\lambda}
				\sum_{i=1}^{p}
				L_i
				\begin{pmatrix}
				L_\beta
				K_i
				D^0_\lambda
				+
				\begin{Vmatrix}
					R_{\beta}(F_i(x^0))
				\end{Vmatrix}
				\end{pmatrix}
				\begin{Vmatrix}
				x-\tilde{x}
				\end{Vmatrix}
			\\
			&
				\hphantom{	L_{f}
				\begin{Vmatrix}
				x-\tilde{x}
				\end{Vmatrix}}
				+
				\tfrac{1}{\lambda}
				\sum_{i=1}^{p}
				L_\beta
				K_i
				\begin{pmatrix}
					L_i
					D^0_\lambda
					+
					\begin{Vmatrix}
					R_{\beta}(F_i(x^0))
					\end{Vmatrix}
				\end{pmatrix}
				\begin{Vmatrix}
				x-\tilde{x}
				\end{Vmatrix}.
			\end{align*}
			\end{subequations}
			Therefore, the gradient \(\nabla s_{\lambda,\beta}\) is Lipschitz continuous on \(\conv (\lev_{\leq s_{\lambda,\beta}(x^{0})}[s_{\lambda,\beta};C])\) with the claimed modulus.
			\end{proof}

			In view of \cref{lemma_Lipschitz}, a first-order inner solver can be used to solve the subproblem in \cref{algorithm2}.
			In particular, we consider first-order solvers \(\mathcal{S} _1\) that achieve a complexity bound of \(\mathcal{O}(\epsilon^{-2})\) for finding an approximate stationary point of the subproblem \(\minimize_{x\in C}s_{\lambda^\nu,\beta}(x)\).
			By combining \cref{theorem_approximate_C} and \cref{lemma_Lipschitz}, we can now establish the total complexity of \cref{algorithm2} for generating an \(\epsilon\)-approximate C-stationary point of~\eqref{MPCC_simple}.

			\begin{theorem}[Total complexity]\label{theorem_total_complexity}
			Let~\cref{assumption_Lipschitz,assumption_level_set} hold, and suppose that a first-order solver \(\mathcal{S} _1\) is used to solve the subproblem approximately.
			Consider the iterates generated by \cref{algorithm2} with \(\lambda^{\nu}\) being updated as \(\lambda^{\nu}=\rho\lambda^{\nu-1}\) with \(\rho\in(0,1)\).
			Then, the last iterate \(x^{\bar{\nu}}\) is an \(\epsilon\)-approximate C-stationary point of~\eqref{MPCC_simple} and at most
			\[
			\mathcal{O}
			\begin{pmatrix}
			\epsilon^{-4}\log_{\frac{1}{\rho}}
			\begin{bmatrix}
				\frac{4\lambda^{0}(f(x^{0})-\underline{f})}{\epsilon^{2}}
			\end{bmatrix}
			\end{pmatrix}
			\]
			gradient evaluations are needed.

			\end{theorem}

			\begin{proof}
			By following~\cite{Grapiglia_2023}, to obtain an \(\epsilon\)-stationary point \(x^{\nu}\) of each subproblem, i.e.,
			\[
			\dist
			\begin{pmatrix}
				-
				\nabla f(x^{\nu})
				-
				\frac{1}{\lambda^\nu}
				\sum_{i=1}^{p}
				JF_{i}{(x^{\nu})}^{\top}
				R_\beta(F_i(x^{\nu})),\
				N_{C}(x^{\nu})
			\end{pmatrix}
			\leq\epsilon,
			\]
			the worst-case complexity bound of the inner solver \(\mathcal{S} _1\) is given by
			\begin{equation}\label{first_order_complexity}
			\frac{c_{1}L_{\lambda^\nu}(x^{0})
			\begin{pmatrix}
				s_{\lambda^{\nu},\beta}(x^{\nu,0})-\underline{s}_{\lambda^{\nu},\beta}
			\end{pmatrix}}{\epsilon^{2}},
			\end{equation}
			where \(L_{\lambda^\nu}(x^{0})\) is defined in \cref{lemma_Lipschitz}, $x^{\nu,0}$ is the starting point of the subproblem and $c_{1}>0$ is a constant that depends on the first-order method.
			Note that \(x^{0}\) is used in \(L_{\lambda^\nu}(x^{0})\) as specified in the step 2 of \cref{algorithm2}, i.e., \(s_{\lambda^{\nu},\beta}(x^\nu)\leq s_{\lambda^{\nu},\beta}(x^0)\).
			Furthermore, \(\underline{s}_{\lambda^{\nu},\beta}\) is a lower bound to \(s_{\lambda^{\nu},\beta}\).
			It follows from~\eqref{r_beta_nonnegative},~\eqref{function_s} and \cref{assumption_level_set} that
			\[
			s_{\lambda,\beta}(x)
			=
			f(x)
			+
			\tfrac{1}{\lambda}
			\sum_{i=1}^{p}
			r_{\beta}(F_{i}(x))
			\geq
			f(x)
			\geq
			\underline{f}
			\]
			for all \(x\in\R^n\).
			Therefore, we can let \(\underline{s}_{\lambda^{\nu},\beta}=\underline{f}\).
			We now aim to establish an upper bound for  \(s_{\lambda^{\nu},\beta}(x^{\nu,0})\).
			To this end, let us start by establishing an upper bound of \(\nicefrac{1}{\lambda^{\nu}}\).
			Suppose that \(x^{\bar{\nu}}\) is the point returned by meeting \(
				\sum_{i=1}^{p}
			r_{\beta}(F_{i}(x^{\bar{\nu}}))
				\leq
				\nicefrac{\epsilon^{2}}{2}
			\), and let \(\nu\in\set{1,\ldots,\bar{\nu}}\).
			Then, according to \cref{theorem_approximate_C}, we have
			\begin{equation}\label{lambda_upper_bound}
				\frac{1}{\lambda^{\nu}}
				=
				\frac{1}{\rho^{\nu}}
				\cdot
				\frac{1}{\lambda^{0}}
				\leq
				\frac{1}{\rho^{\bar{\nu}}}
				\cdot
				\frac{1}{\lambda^{0}}
				=
				\frac{1}{\rho}
				\cdot
				\frac{1}{\rho^{\bar{\nu}-1}}
				\cdot
				\frac{1}{\lambda^{0}}
				<
				\frac{1}{\rho}
				\cdot
				\frac{4(f(x^{0})-\underline{f})}{\epsilon^{2}}.
			\end{equation}
			Furthermore, using the sandwich property \cref{fact_sandwich}, the initial condition \(
			\sum_{i=1}^{p}
			r_{\beta}(F_{i}(x^0))
			\leq
			\nicefrac{\epsilon^{2}}{4}
			\)
			implies that
			\[
			\sum_{i=1}^{p}
			\lambda^{\nu}
			\env_{\lambda^{\nu}}
			\delta_D(F_{i}(x^0))
			\leq
			\sum_{i=1}^{p}
			\lambda^{\nu}
			\env_{\lambda^{\nu},\beta\lambda^\nu}
			\delta_D(F_{i}(x^0))
			\leq
			\frac{\epsilon^{2}}{4}.
			\]
			Then, by applying the formulation of the Moreau envelope in~\eqref{moreau_compl}, we obtain
			\begin{equation}\label{moreau_upper_bound}
			\sum_{i=1}^{p}
			\dist
			\begin{pmatrix}
				F_{i}(x^0),D
			\end{pmatrix}^2
			\leq
			\frac{\epsilon^{2}}{2}.
			\end{equation}
			Using again the sandwich property \cref{fact_sandwich}, we obtain an upper bound for \(s_{\lambda^{\nu},\beta}(x^{\nu,0})\):
			\begin{align}
			s_{\lambda^{\nu},\beta}(x^{\nu,0})
			\leq
			s_{\lambda^{\nu},\beta}(x^{0})
			={} &
			f(x^{0})
			+
			\sum_{i=1}^{p}
			\tfrac{1}{\lambda^\nu}
			r_\beta(F_i(x^0))
			\nonumber\\
			\leq{} &
			f(x^{0})
			+
			\sum_{i=1}^{p}
			\env_{\lambda^{\nu}(1-\beta)}\delta_{D}(F_i(x^0))
			\nonumber\\
			\text{\small by \eqref{moreau_compl}}
			={} &
			f(x^{0})
			+
			\frac{1}{2\lambda^{\nu}(1-\beta)}
			\sum_{i=1}^{p}
			\dist
			\begin{pmatrix}
				F_{i}(x^0),D
			\end{pmatrix}^2
			\nonumber\\
			\text{\small by \eqref{moreau_upper_bound}}
			\leq{} &
			f(x^{0})
			+
			\frac{\epsilon^{2}}{4\lambda^{\nu}(1-\beta)}
			\nonumber\\
			\text{\small by \eqref{lambda_upper_bound}}
			<{} &
			\underbrace{
			f(x^{0})
			+
			\frac{f(x^{0})-\underline{f}}
			{\rho(1-\beta)}
			}_{c_{2}:=}.\label{c2}
			\end{align}
			Recalling that \(D^0_\lambda\) is finite for all \(\lambda>0\), as established in \cref{lemma_Lipschitz}, it follows that the constant
			\begin{equation}\label{max_L_lambda}
				L(x^{0})\coloneqq
				\max_{\nu\in\{1,2,\ldots,\bar\nu\} }
				\tilde{L}_{\lambda^\nu}(x^{0})
			\end{equation}
			is also finite, where \(\tilde{L}_{\lambda}(x^{0})\) is defined in~\eqref{Lipschitz_sub_modulus}.
			In this way, by~\eqref{first_order_complexity}, the total number of gradient evaluations in \cref{algorithm2} is bounded by
			\begin{align*}
				\sum_{\nu=1}^{\bar{\nu}}
				c_{1}
				L_{\lambda^\nu}(x^{0})
				\begin{pmatrix}
				s_{\lambda^{\nu},\beta}(x^{\nu,0})-\underline{f}
				\end{pmatrix}
				\epsilon^{-2}
			\leq{} &
				\sum_{\nu=1}^{\bar{\nu}}
				c_{1}
				L_{\lambda^\nu}(x^{0})
				\begin{pmatrix}
				s_{\lambda^{\nu},\beta}(x^0)-\underline{f}
				\end{pmatrix}
				\epsilon^{-2}
			\\
				\text{\small by \eqref{Lipschitz_sub_modulus_1} and \eqref{c2}}
			<{} &
				\sum_{\nu=1}^{\bar{\nu}}
				c_{1}
				(c_{2}-\underline{f})
				\begin{pmatrix}
				L_{f}
				+
				\frac{1}{\lambda^{\nu}}
				\tilde{L}_{\lambda^{\nu}}(x^{0})
				\end{pmatrix}
				\epsilon^{-2}
			\\
				\text{\small by \eqref{max_L_lambda}}
			\leq{} &
				\sum_{\nu=1}^{\bar{\nu}}
				c_{1}
				(c_{2}-\underline{f})
				\begin{pmatrix}
				L_{f}
				+
				\frac{1}{\lambda^{\nu}}
				L(x^{0})
				\end{pmatrix}
				\epsilon^{-2}.
			\end{align*}
			Finally, with the aid of \cref{theorem_approximate_C}, the above result leads to
			\begin{align*}
			&
				\sum_{\nu=1}^{\bar{\nu}}
				c_{1}
				(c_{2}-\underline{f})
				\begin{pmatrix}
					L_{f}
					+
					\frac{1}{\lambda^{\nu}}
					L(x^{0})
				\end{pmatrix}
				\epsilon^{-2}
			\\
				\text{\small by \eqref{lambda_upper_bound}}
			\leq{} &
				c_{1}
				(c_{2}-\underline{f})
				\bar{\nu}
				\begin{pmatrix}
					L_{f}
					+
					\frac{4(f(x^{0})-\underline{f})}{\rho\epsilon^{2}}
					L(x^{0})
				\end{pmatrix}
				\epsilon^{-2}
			\\
				\text{\small by \eqref{outer_complexity_bound}}
			<{} &
				c_{1}
				(c_{2}-\underline{f})
				\begin{pmatrix}
				1+
				\log_{\frac{1}{\rho}}
				\begin{pmatrix}
					\frac{4\lambda^{0}(f(x^{0})-\underline{f})}{\epsilon^{2}}
				\end{pmatrix}
				\end{pmatrix}
				\begin{pmatrix}
					L_{f}
					+
					\frac{4(f(x^{0})-\underline{f})}{\rho\epsilon^{2}}
					L(x^{0})
				\end{pmatrix}
				\epsilon^{-2}.
			\end{align*}
			Therefore, we conclude that the total complexity of \cref{algorithm2} to obtain an \(\epsilon\)-approximate C-stationary point is \(\mathcal{O}(\epsilon^{-4}\log_{\frac{1}{\rho}}[4\lambda^{0}(f(x^{0})-\underline{f})\epsilon^{-2}])\).
			\end{proof}

		\subsection{Two-Stage Approach}
			The complexity bound in the previous subsection is derived under the assumption that the starting point of the homotopy method is near-feasible, namely such that \(
			\sum_{i=1}^{p}
			r_{\beta}(F_{i}(x^0))
			\leq
			\nicefrac{\epsilon^{2}}{4}.
			\)
			To obtain such a point, we solve the following problem starting from an (infeasible) initial point \(\tilde{x}^{0}\):
			\begin{equation}\label{feasibility_prob}
			\minimize_{x\in\R^n}\ \mathcal{G}(x)
			=
			\sum_{i=1}^{p}
			r_\beta(F_i(x))
			+
			\delta_{C}(x),
			\end{equation}
			where \(\delta_{C}\) is the indicator function of the set \(C\).
			The solution \(x^{0}\) of~\eqref{feasibility_prob} is then used to initialize \cref{algorithm2}.
			This yields a two-stage approach for solving~\eqref{MPCC_simple}.
			\cref{fact_global_optimum} implies that the global minimum of \(\sum_{i=1}^{p}
			r_{\beta}\circ F_i\) is 0, which occurs when \(x\) is feasible for all \(F_i\).
			Moreover, since we assume \(\set{x\in C}[F_{i}(x)\in D,i=1,\ldots,p]\) is nonempty, the global minimum of~\eqref{feasibility_prob} is 0.
			Therefore, the problem~\eqref{feasibility_prob} can be interpreted as a feasibility problem for~\eqref{MPCC_simple}.

			The above observation motivates the introduction of the \emph{proximal-PL} property for~\eqref{feasibility_prob},
			as it ensures that all first-order stationary points attain the global minimum; see~\cite[Section 4]{Karimi_2016} for details.
			Although \cref{proposition_PL} establishes that \(r_\beta\) has the PL property, this does not necessarily imply that \(\mathcal{G}\) has the proximal-PL property.
			However, as shown in~\cite[Appendix G]{Karimi_2016}, the proximal-PL property is equivalent to the \emph{Kurdyka-\L{}ojasiewicz (KL)} property over \(\dom\mathcal{G}\).
			Therefore, in the special case of \(p=1\), we provide the following result to characterize conditions under which \(r_\beta\circ F_i+\delta_C\) has the KL property over \(\dom(r_\beta\circ F_i+\delta_C)\).
			To maintain notational consistency, we retain the subscript \(i\) of \(F_i\) even if \(p=1\).

			\begin{proposition}[KL property of \(r_\beta\circ F_i+\delta_C\)]\label{proposition_PL_linear}
			Fix \(\beta\in(0,1)\), let \(r_\beta,R_\beta\) be as in~\eqref{r_beta} and~\eqref{R_beta}, and \(F_i(x)=(G_i(x),H_i(x))\).
			Suppose that
			\begin{enumerate}
				\item[either] (i) \(\nabla G_i(x)\) and \(\nabla H_i(x)\) are linearly independent for all \(x\in\R^n\),
				\item[or] (ii) \(\nabla G_i(x)\neq0\) and \(H_i(x)=0\) for all \(x\in\R^n\).
			\end{enumerate}
			Then, any point \(x\) such that \(\nabla (r_\beta\circ F_i)(x) = 0\) satisfies \(F_i(x)\in D\).
			If the conditions can be strengthened to
			\begin{enumerate}
				\item[either] (i') \(\inf_{x\in\R^n}\lambda_{\rm min}(JF_i(x){JF_i(x)}^\top)>0\), where \(\lambda_{\rm min}\) denotes the minimal eigenvalue,
				\item[or] (ii') \(\inf_{x\in\R^n}\|\nabla G_i(x)\|>0\) and \(H_i(x)=0\),
			\end{enumerate}\noindent
			and the following inclusion
			\begin{equation}\label{set_intersection}
			\B
			\begin{pmatrix}
				-
				\nabla (r_\beta\circ F_i)(x)
				,
				\begin{Vmatrix}
				\nabla (r_\beta\circ F_i)(x)
			\end{Vmatrix}
			\end{pmatrix}
			\cap
			N_C(x)
			=
			\set{0}
			\end{equation}
			holds,
			then the composite function \(r_\beta\circ F_i+\delta_C\) has the KL property over \(\dom(r_\beta\circ F_i+\delta_C)\).
			\end{proposition}
			\begin{proof}
			Let us first consider condition (i).
			We have
			\[
			\begin{aligned}
				\begin{Vmatrix}
				\nabla (r_\beta\circ F_i)(x)
				\end{Vmatrix}^2
				&=
				\begin{Vmatrix}
				{JF_i(x)}^\top
				R_\beta(F_i(x))
				\end{Vmatrix}^2
				\\
				&=
				{R_\beta(F_i(x))}^\top
				\begin{pmatrix}
				JF_i(x)
				{JF_i(x)}^\top
				\end{pmatrix}
				R_\beta(F_i(x)).
			\end{aligned}
			\]
			Linear independence of \(\nabla G_i(x)\) and \(\nabla H_i(x)\) implies that the minimal eigenvalue of \(JF_i(x){JF_i(x)}^\top\), denoted by \(q_i(x)\), satisfies \(q_i(x)>0\).
			Then, by \cref{proposition_PL}, we obtain
			\begin{align}
				\nonumber
				\tfrac{1}{2}
				\begin{Vmatrix}
				\nabla (r_\beta\circ F_i)(x)
				\end{Vmatrix}^2
				&\geq
				\tfrac{q_i(x)}{2}
				\begin{Vmatrix}
				R_\beta(F_i(x))
				\end{Vmatrix}^2\\
				&\geq
				q_i(x)
				l_\beta
				(r_\beta\circ F_i)(x).
				\label{eq:KLi}
			\end{align}
				Since \(q_i(x)>0\), if \(\nabla (r_\beta\circ F_i)(x)=0\) one must have \(r_\beta(F_i(x))=0\), which by \cref{fact_global_optimum} implies that \(F_i(x)\in\argmin\delta_D=D\).

				For condition (ii), the transpose of the Jacobian becomes \(JF_i{(x)}^\top=(\nabla G_i(x),0)\). Let \(z=(z_1,z_2)=(z_1,0)\).
				Then, according to~\eqref{r_beta} and~\eqref{R_beta}, we have
			\[
			r_\beta(z)=
			\begin{cases}
				\tfrac{1}{2(1-\beta)}z_1^2
				\\
				0
			\end{cases}
			\ \ \text{and}\ \
			R_\beta(z)=
			\begin{cases}
				\tfrac{1}{1-\beta}\binom{z_1}{0} \ \ \ \ &\mbox{if } z_1\leq0
				\\
				\binom{0}{0} \ &\mbox{if } z_1>0.
			\end{cases}
			\]
				With \(z_1\coloneqq G_i(x)\leq0\), it holds that
			\begin{align}
				\nonumber
				\tfrac{1}{2}
				\begin{Vmatrix}
				\nabla (r_\beta\circ F_i)(x)
				\end{Vmatrix}^2
				&=
				\tfrac{1}{2}
				{(R_\beta(F_i(x)))}_1^2
				\|\nabla G_i(x)\|^2
			\\
				\nonumber
				&=
				\tfrac{1}{2}
				\begin{pmatrix}
				\tfrac{1}{1-\beta}
				z_1
				\end{pmatrix}^2
				\|\nabla G_i(x)\|^2\\
				&=
				\tfrac{1}{1-\beta}
				\|\nabla G_i(x)\|^2
				(r_\beta\circ F_i)(x).
			\label{eq:KLii}
			\end{align}
			According to the assumption, we have \(\|\nabla G_i(x)\|>0\), and again we conclude that \(r_\beta(F_i(x))=0\).
			Finally, we have \(r_\beta(F_i(x))=0\) trivially when \(z_1\coloneqq G_i(x)>0\).

			Let us now consider the KL property under condition~\eqref{set_intersection}, along with either condition (i') or (ii').
			Since \(\inf_{x\in\R^n}q_i(x)>0\) under condition (i') and \(\inf_{x\in\R^n}\|\nabla G_i(x)\|>0\) under condition (ii'), in both cases it follows from \eqref{eq:KLi} and \eqref{eq:KLii} that \(r_\beta\circ F_i\) satisfies the PL inequality with some constant \(l_{\beta,i}>0\).
			According to~\cite[Appendix G]{Karimi_2016}, the KL property of \(r_\beta\circ F_i+\delta_C\) is defined by the existence of a constant \(\tilde l_\beta>0\) such that
			\[
			\min_{s\in\partial(r_\beta\circ F_i+\delta_C)(x)}\ \|s\|^2
			\
			\geq
			\
			2
			\tilde l_\beta
			\begin{pmatrix}
				r_\beta\circ F_i+\delta_C
			\end{pmatrix}(x)
			\]
			for all \(x\in\dom(r_\beta\circ F_i+\delta_C)\).
			In our case where \(r_\beta\circ F_i\) is differentiable and \(\delta_C\) is convex, we have
			\[
			\partial(r_\beta\circ F_i+\delta_C)(x)
			=
			\bigl\{
			\nabla (r_\beta\circ F_i)(x)+\xi
			\bigm|
			\xi\in N_C(x)
			\bigr\}.
			\]
			Based on this, let us consider the inequality:
			\[
			\begin{Vmatrix}
				\xi
			\end{Vmatrix}^2
			+
			2
			\big\langle
				\nabla (r_\beta\circ F_i)(x),
				\xi
			\big\rangle
			\leq0
			\quad
			\Leftrightarrow
			\quad
			\begin{Vmatrix}
				\xi
				+
				\nabla (r_\beta\circ F_i)(x)
			\end{Vmatrix}
			\leq
			\begin{Vmatrix}
				\nabla (r_\beta\circ F_i)(x)
			\end{Vmatrix}.
			\]
			Therefore, the condition \eqref{set_intersection} implies that
			\begin{equation}\label{KL_inf}
			\inf_{\xi\in N_C(x)}
			\begin{pmatrix}
				\begin{Vmatrix}
				\xi
			\end{Vmatrix}^2
			+
			2
			\big\langle
				\nabla (r_\beta\circ F_i)(x),
				\xi
			\big\rangle
			\end{pmatrix}
			\geq0.
			\end{equation}
			Let \(\xi\in N_C(x)\), then we have
			\begin{align*}
			&
				\begin{Vmatrix}
					\nabla (r_\beta\circ F_i)(x)
					+
					\xi
				\end{Vmatrix}^2
			\\
			={} &
				\begin{Vmatrix}
					\nabla (r_\beta\circ F_i)(x)
				\end{Vmatrix}^2
				+
				\begin{Vmatrix}
					\xi
				\end{Vmatrix}^2
				+
				2
				\big\langle
					\nabla (r_\beta\circ F_i)(x)
					,
					\xi
				\big\rangle
			\\
			\geq{} &
				\begin{Vmatrix}
					\nabla (r_\beta\circ F_i)(x)
				\end{Vmatrix}^2
				+
				\inf_{\tilde\xi\in N_C(x)}
				\begin{pmatrix}
					\begin{Vmatrix}
					\tilde\xi
				\end{Vmatrix}^2
				+
				2
				\big\langle
					\nabla (r_\beta\circ F_i)(x)
					,
					\tilde\xi
				\big\rangle
				\end{pmatrix}
			\\
				\text{\small by \eqref{KL_inf}}
			\geq{} &
				\begin{Vmatrix}
					\nabla (r_\beta\circ F_i)(x)
				\end{Vmatrix}^2
			\\
				\text{\small by (i') or (ii')}
			\geq{} &
				2l_{\beta,i}
				(r_\beta\circ F_i)(x).
			\end{align*}
			Since the condition~\eqref{set_intersection} implies \(N_C(x)\neq\emptyset\), thus \(x\in C\) and we have
			\[
			\begin{Vmatrix}
				\nabla (r_\beta\circ F_i)(x)
				+\xi
			\end{Vmatrix}^2
			\geq
			2l_{\beta,i}
			\begin{pmatrix}
				r_\beta\circ F_i+\delta_C
			\end{pmatrix}(x).
			\]
			In this way, we conclude that \(r_\beta\circ F_i+\delta_C\) has the KL property over \(\dom(r_\beta\circ F_i+\delta_C)\).
			\end{proof}

			\begin{remark}
			Note that the null gradient in conditions (ii) and (ii') of \cref{proposition_PL_linear} correspond to the setting in \cref{example_equality_inequality_CC}.
			Furthermore, although conditions (i), (ii), (i') and (ii') may be difficult to verify for all \(x\in\R^n\), an important and commonly encountered special case arises when the constraints are linear.
			Specifically, consider \(F_i(x)=(a_i^\top x+b_i,c_i^\top x+d_i)\), where \(a_i,c_i\in\R^n\) and \(b_i,d_i\in\R\).
			Then, all conditions in \cref{proposition_PL_linear} simply reduce to:
			either (i) \(a_i\) and \(b_i\) being linearly independent,
			or (ii) \(a_i\neq0\) and \(c_i=0,d_i=0\).
			\end{remark}

			Assuming that all conditions in \cref{proposition_PL_linear} are satisfied, the result in~\cite[Appendix G]{Karimi_2016} implies that \(r_\beta\circ F_i+\delta_C\) has the proximal-PL property.
			Then, according to~\cite[Theorem 5]{Karimi_2016}, the projected gradient method with the step-size \(\nicefrac{1}{L}\),
			\[
			\tilde{x}^{\nu+1}=
			\Pi_C
			\begin{pmatrix}
				\tilde{x}^{\nu}
				-
				\frac{1}{L}
				\nabla (r_\beta\circ F_i)(\tilde{x}^{\nu})
			\end{pmatrix}
			\]
			converges linearly to the global minimum 0 from the starting point \(\tilde{x}^0\), where \(L\) is the Lipschitz modulus of \(r_\beta\circ F_i\).
			Therefore, the complexity to meet \(
			(r_{\beta}\circ F_{i})(x^0)
			\leq
			\nicefrac{\epsilon^{2}}{4}
			\) is \(\mathcal{O} (\log\frac{1}{\epsilon})\).
			When the Lipschitz modulus \(L\) is unknown, backtracking is commonly used to determine the step-size.
			In practice, however, it is well known that backtracking often yields  faster practical convergence than the fixed step-size, even when \(L\) is known.
			Excluding function evaluations needed in the linesearch, the complexity remains \(\mathcal{O} (\log\frac{1}{\epsilon})\) calls to the gradient in this case.

			The condition~\eqref{set_intersection} in \cref{proposition_PL_linear} imposes an additional geometric structure between \(F_i\) and \(C\) that ensures the KL property.
			However, this condition is nontrivial to verify.
			When \(p>1\), the situation becomes even more complex.
			The condition~\eqref{set_intersection} now becomes
			\[
			\B
			\begin{pmatrix}
				-
				\sum_{i=1}^{p}
				\nabla (r_\beta\circ F_i)(x)
				,
				\begin{Vmatrix}
				\sum_{i=1}^{p}
				\nabla (r_\beta\circ F_i)(x)
			\end{Vmatrix}
			\end{pmatrix}
			\cap
			N_C(x)
			=
			\set{0}
			\]
			and, in addition, the following inequality is required:
			\begin{equation}\label{grad_align}
			\begin{Vmatrix}
				\sum_{i=1}^{p}
				\nabla (r_\beta\circ F_i)(x)
			\end{Vmatrix}^2
			\geq
			\kappa\sum_{i=1}^{p}
			\begin{Vmatrix}
				\nabla (r_\beta\circ F_i)(x)
			\end{Vmatrix}^2
			\ \
			\text{for some}\ \kappa>0.
			\end{equation}
			Eq.~\eqref{grad_align} can be interpreted as a \emph{gradient alignment} condition, which requires that the gradients do not cancel each other out excessively, or, equivalently, exhibit some degree of positive correlation.
			The problem structures under which~\eqref{set_intersection} and~\eqref{grad_align}  naturally arise will be explored in future work.
			Nevertheless, even when \(\mathcal{G}\) lacks the proximal-PL property, we can still solve~\eqref{feasibility_prob} for a starting point of \cref{algorithm2}, albeit heuristically, as the method then only guarantees convergence to a stationary point of~\eqref{feasibility_prob}.
			The complete algorithm is presented in \cref{algorithm3}, and its effectiveness is demonstrated through a range of numerical experiments in the next section.

			\begin{algorithm}[h]
			\caption{LL-based homotopy (two-stage) approach}\label{algorithm3}
			\begin{algorithmic}[1]
			\itemsep=3pt%
			\renewcommand{\algorithmicindent}{0.3cm}%

			\item[\textbf{Initialization:}]

			Let \(\nu=1\). Given \(\epsilon>0\), \(\lambda^0>0\), \(\beta\in(0,1)\) and \(\tilde{x}^0\in\R^n\).

			\item[\textbf{\textbf{Stage 1}:}]

			\State
				\begin{tabular}[t]{@{}l@{}}
					Find an approximate solution \(x^0\) to the problem \(\displaystyle\minimize_{x\in C}\textstyle\sum_{i=1}^{p}r_\beta(F_i(x))\) such that%
				\\[3pt]
					either \(\sum_{i=1}^{p}r_{\beta}(F_{i}(x^0))\leq\nicefrac{\epsilon^{2}}{4}\)
					\hfill \textcolor{gray}{\small \(\rhd\) ideal: near-feasible point}%
				\\[3pt]
						\hphantom{either}\llap{or}
						\(\dist (-\nabla (\sum_{i=1}^{p}r_\beta\circ F_i)(x^0),N_C(x^0))\leq\epsilon\)
						\hfill \textcolor{gray}{\small \(\rhd\) heuristic: \(\epsilon\)-stationary point}%
				\\[3pt]
					by using an inner solver from the starting point \(\tilde{x}^0\).
				\end{tabular}

			\item[\textbf{\textbf{Stage 2}:}]

			\item[\textbf{Repeat:}]

			\State  Choose \(\lambda^{\nu}\) and set \(s_{\lambda^{\nu},\beta}:=f+\sum_{i=1}^{p}\frac{1}{\lambda^\nu}
			r_{\beta}\circ F_{i}\).

			\State Find an approximate solution \(x^\nu\) to the problem \(\minimize_{x\in C}s_{\lambda^{\nu},\beta}(x)\) such that

			\begin{itemize}
				\item \(s_{\lambda^{\nu},\beta}(x^\nu)\leq\min\set{s_{\lambda^{\nu},\beta}(x^{\nu-1}),s_{\lambda^{\nu},\beta}(x^0)}\) and
				\item \(\dist (-\nabla s_{\lambda^{\nu},\beta}(x^\nu),N_C(x^\nu))\leq\epsilon\)
			\end{itemize}
			by using an inner solver with starting point
			\(x^{\nu,0}\coloneqq\arg\min\set{s_{\lambda^{\nu},\beta}(\hat x)}[\hat{x}\in\set{x^{\nu-1},x^{0}}]\).

			\State Let \(\nu\leftarrow\nu+1\).

			\item[\textbf{Until:} The stopping criterion
			\(
				\sum_{i=1}^{p}
			r_{\beta}(F_{i}(x^\nu))
				\leq
				\nicefrac{\epsilon^{2}}{2}
			\)
			is satisfied.]
			\end{algorithmic}
			\end{algorithm}

	\section{Numerical Results}\label{sec_experiments}
		\subsection{Implementation}
			In this section, we test \cref{algorithm3} on several benchmark problems.
			Following \cref{example_equality_inequality_CC}, all equality and inequality constraints are reformulated as CCs, which leads to the form of~\eqref{MPCC_simple}.
			\cref{algorithm3} with a first-order inner solver is referred to as \emph{LL2(1st)} (two-stage LL-based homotopy approach with a first-order inner solver).
			The method without trying to find a feasible point at the beginning is called \emph{LL1} (one-stage LL-based homotopy approach).
			Furthermore, we observe that the second-order solver~\cite{Coleman_1996} may perform well in stage 1, and result in a better final solution.
			Therefore, it is also included in the comparison and referred to as \emph{LL2(2nd)}.
			It is well known that the Scholtes method~\cite{Scholtes_2001} exhibits strong numerical performance compared to other regularization methods in practice~\cite{Hoheisel_2013}.
			It is implemented as follows:
			\begin{align*}
			\minimize_{x\in C}
			\ \
			& f(x)
			\\
			\stt\ \
			&
			G(x)\geq0,\
			H(x)\geq0,\
			G_i(x)H_i(x)\leq\tau,\
			\ i\in\set{1,\ldots,p}
			\\
			& c_{i}(x)=0,\ i\in\mathcal{E}
			\\
			& c_{i}(x)\leq0,\ i\in\mathcal{I}
			\end{align*}
			where \(\tau\in\R_+\) is the relaxation parameter.
			We also implemented the LL method without relaxing the general equality and inequality constraints:
			\begin{align*}
			\minimize_{x\in C}
			\ \
			& f(x)
			+
			\sum_{i=1}^{p}
			\env_{\lambda,\beta\lambda}\delta_D
			(F_i(x))
			\\
			\stt\ \
			& c_{i}(x)=0,\ i\in\mathcal{E}
			\\
			& c_{i}(x)\leq0,\ i\in\mathcal{I}
			\end{align*}
			which is referred to as \emph{LL penalty}.
			For the comparison of open source and commercial solvers, we used ALGENCAN~\cite{Andreani_2008}, IPOPT~\cite{Wachter_2006} and MINOS~\cite{Murtagh_1983}.
			These NLP solvers have been shown to be efficient for MPCCs to some extent, as demonstrated in~\cite{Izmailov_2012,Kirches_2022,Leyffer_2006,Raghunathan_2005}.
			Following the methods in~\cite{Izmailov_2012,Kirches_2022}, we report the performance of NLP solvers based on the following formulation:
			\[
			G(x)\geq0,\
			H(x)\geq0
			\quad
			\text{and}
			\quad
			\langle G(x),H(x)\rangle
			\leq0.
			\]
			We set the tolerance as \(\epsilon=10^{-8}\) for LL-based, Scholtes and NLP methods.
			The parameters for LL-based methods are: \(\beta=0.999\), \(\rho=0.8\) and \(\lambda^0=1\). The parameters in~\cite{Scholtes_2001} are adopted for Scholtes method, i.e., \(\tau^\nu=0.1\tau^{\nu-1}\) and \(\tau^0=1\).
			The maximum number of outer iterations for the LL-based and Scholtes methods is 200.
			Eq.~\eqref{stopping_criterion} and~\eqref{stopping_criterion_norm} serve as the stopping criteria for the LL-based and Scholtes methods, respectively.
			The above methods are implemented in MATLAB (R2024b).
			L-BFGS-B\footnote{\url{https://github.com/stephenbeckr/L-BFGS-B-C}} is used as the inner solver for LL1, LL2(1st) and stage 2 of LL2(2nd), with the memory parameter set to 5.
			The problems are modeled in AMPL (A Mathematical Programming Language)~\cite{Fourer_2003}.
			The NLP solvers are invoked via AMPL MATLAB API\@.

		\subsection{MacMPEC Benchmark}
			In order to obtain quantitative results, we first benchmark \cref{algorithm3} on the MacMPEC collection.\footnote{\url{https://wiki.mcs.anl.gov/leyffer/index.php/MacMPEC}}
			As a preliminary result, we use 77 problems from the collection with fewer than 100 variables.
			For each test problem, we perform 20 runs from randomly generated starting points, where each element is a uniform random value in the interval \([-50,50]\).
			\cref{table1} reports the number of cases in which the algorithm converges to an optimal point, a suboptimal but feasible point, an infeasible point, or encounters a solver failure.
			The returned objective value is regarded as optimal if its relative suboptimality with respect to the best known value, i.e., \((\nicefrac{|\text{obj}-\text{obj}^*|}{|\text{obj}^*|})\times 100\%\), is within $5\%$.
			We use IPOPT as the inner solver for both the Scholtes and LL penalty methods, as it outperforms other NLP solvers on MacMPEC benchmark.
			\cref{table1} shows that the LL penalty performs best in finding the optimal solution, and it is followed by the Scholtes method.
			However, the Scholtes method returns more infeasible points compared to LL-based methods.
			LL2(1st) and LL2(2nd) are particularly effective at avoiding infeasible points, demonstrating the effectiveness of the two-stage approach.
			Since LL2(2nd) initializes \cref{algorithm2} using a second-order stationary point of problem~\eqref{feasibility_prob}, it has a higher likelihood of converging to the global minimizer.
			In contrast, LL2(1st), which uses only a first-order stationary point, may be more prone to converging to local minimizers.
			While LL1 returns the highest number of infeasible points among the LL-based methods, it still performs better in this regard than the Scholtes method and other NLP solvers.
			Overall, we conclude that \cref{algorithm3} exhibits superior robustness compared to the other methods.
			It consistently translates infeasible points into at least suboptimal points, indicating strong potential for generating high-quality solutions in practice.

			\begin{table}[h]
			\centering
			\caption{Problem number counting of optimal, suboptimal and infeasible points, and solver failures.}
			\resizebox{\columnwidth}{!}{%
			\begin{tabular}{|l|l|l|l|l|l|l|l|l|}
			\hline
			& LL1 & LL penalty & LL2 (1st) & LL2 (2nd) & Scholtes & ALGENCAN & IPOPT & MINOS
			\\ \hline
			optimal  & \makecell[c]{1262} & \makecell[c]{1372} & \makecell[c]{1273} & \makecell[c]{1286}  & \makecell[c]{1339} & \makecell[c]{1240} & \makecell[c]{1293} & \makecell[c]{1076}
			\\ \hline
			suboptimal & \makecell[c]{243} & \makecell[c]{148} & \makecell[c]{251} & \makecell[c]{238} & \makecell[c]{151} & \makecell[c]{148} & \makecell[c]{171} & \makecell[c]{340}
			\\ \hline
			infeasible & \makecell[c]{35} & \makecell[c]{20} & \makecell[c]{16} & \makecell[c]{16} & \makecell[c]{50} & \makecell[c]{49} & \makecell[c]{76} & \makecell[c]{105}
			\\ \hline
			solver failure & \makecell[c]{0} & \makecell[c]{0} & \makecell[c]{0} & \makecell[c]{0} & \makecell[c]{0} & \makecell[c]{103} & \makecell[c]{0} & \makecell[c]{19}
			\\ \hline
			\end{tabular}
			}\label{table1}
			\end{table}

			\begin{figure}[htb]\centering
			\includegraphics[width=\textwidth]{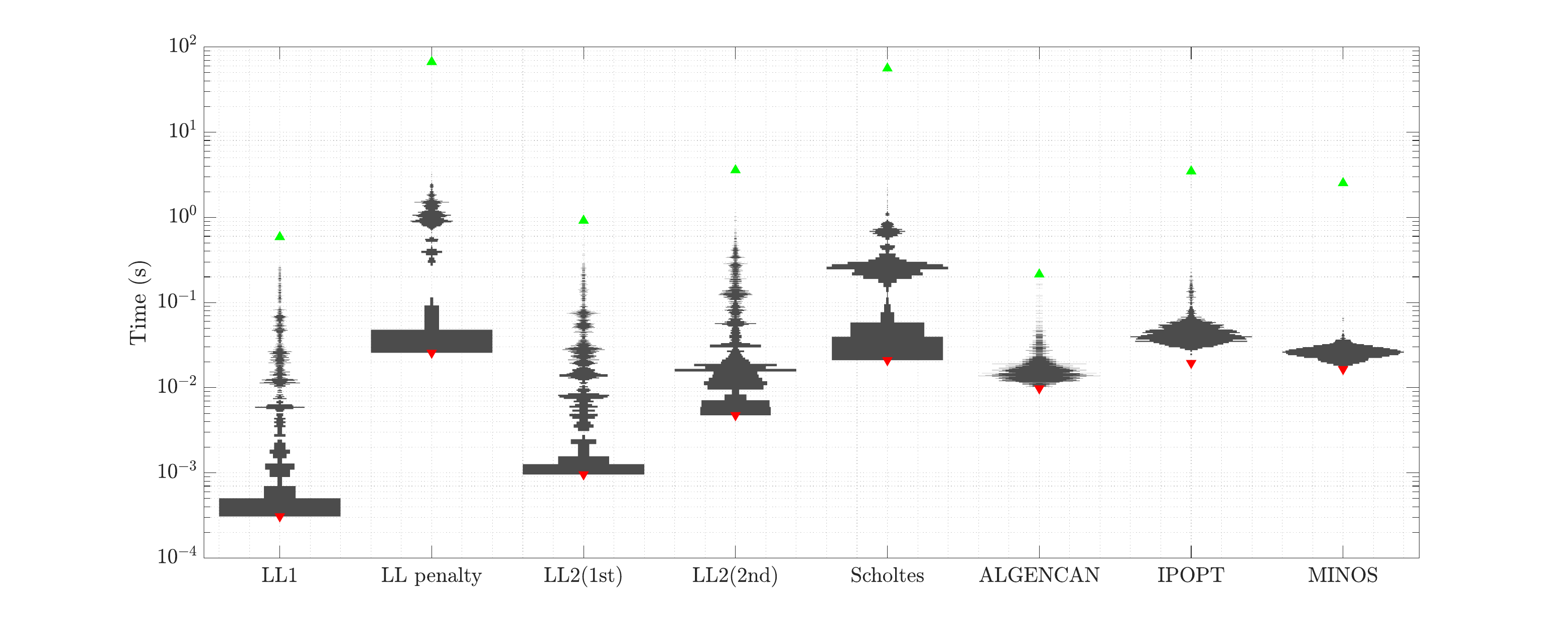}
			\caption{Computation time (time value distribution and minimal/maximal values).}\label{fig_time}
			\end{figure}

			The computation time of each method is presented in \cref{fig_time}, excluding cases of solver failure.
			LL1 and LL2(1st) present higher computation efficiency compared to the LL penalty, LL2(2nd) and Scholtes method, which is not surprising since it benefits from the fast first-order solver for smooth subproblems.
			Although the LL penalty and Scholtes methods achieve better performance in terms of solution quality, see \cref{table1}, they incur significantly higher computational costs due to their reliance on IPOPT as the inner solver.
			On the other hand, LL2(2nd) is slower than LL2(1st) since the second-order solver is used in stage 1.
			ALGENCAN, IPOPT and MINOS exhibit more consistent and less volatile computation times, which come from their efficient and mature implementation.
			Even though the current implementation of \cref{algorithm3} is a prototype in MATLAB and not optimized, it is still competitive in terms of the computation efficiency to the NLP solvers on the MacMPEC benchmark.

		\subsection{Bound-Constrained Quadratic MPCCs}
			In this part, we consider the larger size MPCCs with quadratic objective and bound-constrained variables:
			\begin{equation}\label{QPCC}
			\begin{split}
			\minimize_{x\in \R^n}
			\ \
			& \tfrac{1}{2} x^\top Qx+g^\top x
			\\
			\stt\ \
			& l_0\leq x_0\leq u_0
			\\
			& 0\leq x_1\ \bot\ x_2\geq0,
			\end{split}
			\end{equation}
			where \(x:=(x_0,x_1,x_2)\in\R^{n}\), \(x_0\in\R^{n_0}\), \(x_1,x_2\in\R^p\), \(Q\in\R^{n\times n}\) is a symmetric matrix, \(g\in\R^n\) is a vector, and the lower and upper bounds satisfy \({(l_0)}_i<{(u_0)}_i\) for \(i\in\set{1,\ldots,n_0}\).
			The problem of the form~\eqref{QPCC} is studied in~\cite{Kirches_2022}, as it arises as a subproblem within an augmented Lagrangian framework for solving more general MPCCs.
			Since it also constitutes a special case of~\eqref{MPCC_simple}, it provides a meaningful benchmark for evaluating our algorithm.
			A batch of test problems are randomly generated by using the code\footnote{\url{https://wiki.mcs.anl.gov/leyffer/index.php/BndMPCC}} in~\cite{Kirches_2022} and the problem information is:
			\(Q\) is sparse with the density \(\nicefrac{n^2}{4}\), whose elements are normally distributed, the elements in \(g\) are uniform random values in the interval \([-10,10]\), the bounds \(l_0,u_0\) are uniform random values in the interval \([-10,10]\) and \([0,20]\).
			In order to increase the difficulty of solving, the matrix \(Q\) is generated to be indefinite.
			We consider two problem sizes: \((n_0=100,p=200)\) and \((n_0=200,p=400)\), where the total number of variables is 500 and 1000.
			We randomly generate \(10\) problems for each size.
			The comparison with MINOS is excluded due to instability issues observed in the previous subsection.
			ALGENCAN is used as the inner solver of the Scholtes method due to its good performance on~\eqref{QPCC}, as shown in \cref{table_QPCC}.
			Since larger problem sizes are considered, we focus primarily on the first-order methods among the LL-based methods from this point onward.
			Furthermore, it is trivial to find a feasible point of~\eqref{QPCC}, so we initialize all methods with \(x^0=(l_0,0,0)\) without loss of generality.
			As a result, stage 1 of the LL-based method is not required.
			Noting that only bound constraints are involved in~\eqref{QPCC}, LL penalty is equivalent to LL1.

			In this subsection, we include a comparison with the projected gradient method (PGM), taking advantage of the fact that the constraint set in~\eqref{QPCC} admits an efficient projection.
			Specifically, we define the projection operator and the objective function as
			\[
				\Pi_{[l_0,u_0]\times D}(x)\coloneqq(\Pi_{[l_0,u_0]}(x_0),\Pi_D((x_1,x_2)))
				\quad
				\text{and}
				\quad
				f(x)\coloneqq\tfrac{1}{2} x^\top Qx+g^\top x.
			\]
			The PGM with Armijo linesearch for solving problem~\eqref{QPCC} begins by initializing \(x^0\) and setting parameters \(\epsilon>0\), \(\eta_{\max}>0\), \(\gamma\in(0,1)\) and \(c\in(0,1)\).
			At each iteration, the algorithm sets \(\eta\leftarrow\eta_{\max}\),
			computes the projected trial point \(x_{\mathrm{trial}}\leftarrow\Pi_{[l_0,u_0]\times D}(x^\nu-\eta\nabla f(x^\nu))\),
			and backtracks via \(\eta\leftarrow\gamma\eta\) until the sufficient decrease condition
			\[
				f(x_{\mathrm{trial}})
				\leq
				f(x^\nu)
				-
				c
				\langle
				\nabla f(x^\nu),x^\nu-x_{\mathrm{trial}}
				\rangle
			\]
			is satisfied.
			The next iterate is then set as \(x^{\nu+1}\leftarrow x_{\mathrm{trial}}\),
			and the process terminates when \(\|\nicefrac{(x^{\nu+1} - x^\nu)}{\eta}\|\leq\epsilon\).
			This method follows the standard PGM framework~\cite[Section 2.3]{Bertsekas_1999}, with commonly used parameter values: \(\eta_{\max}=1\), \(\gamma=0.5\) and \(c=10^{-4}\).
			The tolerance in PGM is consistent with the other methods, i.e., \(\epsilon=10^{-8}\).

			\begin{table}[htb]
				\centering
				\caption{Objective, constraint violation and computation time on bounded quadratic MPCCs}
				\resizebox{\columnwidth}{!}{%
				\begin{tabular}{|l|ccc|ccc|ccc|ccc|cc|}
				\hline
				\(n_0\)--\(p\)--num. & LL1 & & & Scholtes & & & ALGENCAN & & & IPOPT & & & PGM & \\
				& obj. & vio. & time\@(s) & obj. & vio. & time\@(s) & obj. & vio. & time\@(s) & obj. & vio. & time\@(s) & obj. & time\@(s) \\
				\hline
			100--200--0 & 0.90\% & 1.1e-11 & 1.82 & \textbf{-58355.23} & 2.9e-09 & 6.37 & 2.78\% & 0 & \textbf{\underline{1.41}} & 4.29\% & 1.4e-05 & 58 & 1.24\% & 4.64 \\
			100--200--1 & 0.46\% & 1.2e-11 & 2.10 & \textbf{-78425.76} & 2.2e-09 & 11.12 & 3.79\% & 0 & 0.79 & 4.41\% & 7.3e-06 & 70 & 1.11\% & \textbf{\underline{0.11}} \\
			100--200--2 & 3.61\% & 1.9e-10 & 1.82 & \textbf{-46340.69} & 2.8e-09 & 10.07 & 13.89\% & 0 & 0.68 & 14.45\% & 2.0e-05 & 66 & 7.29\% & \textbf{\underline{0.19}} \\
			100--200--3 & \textbf{-64054.32} & 4.4e-11 & 1.53 & 0.21\% & 2.5e-09 & 7.46 & 2.46\% & 0 & 0.63 & 3.90\% & 1.5e-05 & 82 & 0.21\% & \textbf{\underline{0.16}} \\
			100--200--4 & \textbf{-55115.72} & 1.1e-10 & 1.69 & 0.53\% & 2.6e-09 & 11.52 & 11.88\% & 0 & 1.15 & 15.87\% & 1.4e-05 & 94 & 0.01\% & \textbf{\underline{0.09}} \\
			100--200--5 & \textbf{-67241.94} & 1.0e-10 & 1.86 & 0.83\% & 2.8e-09 & 6.99 & 5.57\% & 0 & 0.89 & 16.00\% & 2.5e-05 & 75 & 0.47\% & \textbf{\underline{0.19}} \\
			100--200--6 & 0.70\% & 1.6e-11 & 1.88 & \textbf{-71063.23} & 3.6e-09 & 10.59 & 0.79\% & 0 & 0.72 & 2.94\% & 2.8e-05 & 62 & 0.64\% & \textbf{\underline{0.09}} \\
			100--200--7 & \textbf{-61504.67} & 4.9e-11 & 2.04 & 4.21\% & 0 & 9.08 & 8.17\% & 0 & 0.93 & 20.39\% & 5.3e-06 & 82 & 9.58\% & \textbf{\underline{0.06}} \\
			100--200--8 & \textbf{-54981.98} & 1.6e-11 & 2.30 & 0.34\% & 2.4e-09 & 19.86 & 22.47\% & 0 & 0.89 & 14.95\% & 2.4e-05 & 138 & 8.01\% & \textbf{\underline{0.10}} \\
			100--200--9 & \textbf{-54423.15} & 1.6e-10 & 1.66 & 0.15\% & 3.0e-09 & 7.52 & 1.89\% & 0 & 0.53 & 4.35\% & 1.9e-05 & 77 & 0.19\% & \textbf{\underline{0.28}} \\
			\hline
			200--400--0 & 1.09\% & 3.9e-12 & 11.73 & \textbf{-210017.27} & 9.4e-09 & 93 & 7.02\% & 0 & 8.65 & 4.57\% & 6.0e-05 & 883 & 1.52\% & \textbf{\underline{0.11}} \\
			200--400--1 & 4.50\% & 2.0e-11 & 10.89 & \textbf{-190225.77} & 0 & 126 & 6.41\% & 0 & 8.33 & 8.76\% & 1.6e-05 & 1296 & 0.38\% & \textbf{\underline{0.12}} \\
			200--400--2 & 4.25\% & 3.0e-11 & 8.63 & \textbf{-195462.75} & 0 & 106 & 8.53\% & 0 & 7.01 & 6.45\% & 1.1e-04 & 1551 & 5.26\% & \textbf{\underline{0.15}} \\
			200--400--3 & 1.51\% & 5.8e-12 & 10.80 & 1.00\% & 3.5e-09 & 104 & 4.52\% & 0 & 6.49 & 7.99\% & 5.8e-05 & 811 & \textbf{-176247.85} & \textbf{\underline{0.16}} \\
			200--400--4 & 1.54\% & 6.5e-11 & 8.59 & 2.08\% & 7.3e-09 & 196 & 8.19\% & 0 & 5.16 & 6.37\% & 1.3e-05 & 743 & \textbf{-197038.66} & \textbf{\underline{3.07}} \\
			200--400--5 & 6.29\% & 8.1e-12 & 8.51 & 1.93\% & 4.6e-09 & 93 & 5.37\% & 0 & 8.15 & 9.71\% & 4.4e-05 & 907 & \textbf{-175486.10} & \textbf{\underline{0.23}} \\
			200--400--6 & \textbf{-172420.88} & 6.9e-12 & 9.80 & 5.81\% & 3.6e-09 & 89 & 6.36\% & 0 & \textbf{\underline{6.54}} & 3.66\% & 3.5e-05 & 1619 & 4.83\% & 9.26 \\
			200--400--7 & \textbf{-192654.03} & 1.0e-10 & 7.76 & 4.41\% & 7.9e-10 & 116 & 7.41\% & 0 & 6.89 & 7.97\% & 1.8e-05 & 943 & 4.27\% & \textbf{\underline{0.14}} \\
			200--400--8 & \textbf{-147134.93} & 3.4e-12 & 9.31 & 2.62\% & 0 & 151 & 7.27\% & 0 & \textbf{\underline{5.31}} & 3.59\% & 7.1e-05 & 1049 & 3.80\% & 9.65 \\
			200--400--9 & \textbf{-166871.82} & 8.5e-11 & 7.98 & 0.67\% & 0 & 78 & 2.87\% & 0 & 8.99 & 10.34\% & 2.3e-05 & 754 & 6.24\% & \textbf{\underline{0.16}} \\
			\hline
				\end{tabular}
				}\label{table_QPCC}
			\end{table}

			The objective value, the maximal violation of CCs, i.e., \(\|\min\set{x_1,x_2}\|_{\infty}\), and the computation time of each method are presented in \cref{table_QPCC}.
			The minimal objective value in each problem is highlighted in bold and the minimal computation time is in bold with an underline.
			The suboptimal objective value is reported in terms of relative suboptimality.
			Computation times exceeding 50 seconds are rounded to the nearest integer.
			The column for constraint violation of the PGM is omitted,
			since the projection operator \(\Pi_{[l_0,u_0]\times D}\) naturally ensures feasibility for problem~\eqref{QPCC}.
			Notably, ALGENCAN also ensures feasibility, as reported by its zero violation in \cref{table_QPCC}.
			However, if a small violation of CCs (on the level of 1e-10) is acceptable, then LL1 is preferable, as it yields improved objective values in many instances.
			The Scholtes method delivers comparable solution quality to LL1 but is significantly slower.
			On the other hand, it is expected that the PGM is the fastest method overall, as it employs a single-loop structure (requiring no inner solver) and operates over a simple constraint set.
			However, it typically does not outperform LL1 in terms of solution quality.
			Finally, IPOPT reports a struggling performance on~\eqref{QPCC} with indefinite Hessian \(Q\).
			Overall, we find that the LL-based method is efficient at finding high-quality solutions for large-scale bound-constrained quadratic MPCCs.

		\subsection{Linearly Constrained Quadratic MPCCs}
			In this subsection, we consider some larger size MPCCs with quadratic objective and linearly constrained variables:
			\begin{equation}\label{QPCC_general}
			\begin{split}
			\minimize_{x\in \R^n}
			\ \
			& \tfrac{1}{2} x^\top Qx+g^\top x
			\\
			\stt\ \
			& Ax_0+a\leq 0
			\\
			& 0\leq x_1\ \bot\ Nx_0+Mx_1+q\geq0,
			\end{split}
			\end{equation}
			where \(x:=(x_0,x_1)\in\R^{n}\), \(x_0\in\R^{n_0}\) and \(x_1\in\R^p\).
			Test problems are randomly generated by using the code\footnote{\url{https://www3.eng.cam.ac.uk/~dr241/Papers/qpecgen.m}} in~\cite{Jiang_1999}, which also provides the global minima.
			We use the same number of variables as the previous subsection, i.e., \(n=500,1000\).
			Specifically, we choose \((n_0=100,p=400)\) and \((n_0=200,p=800)\), noting that \(n=n_0+p\) for~\eqref{QPCC_general}.
			For the problem data, we set the following: the matrix \(Q\) is positive definite with a condition number \(10^4\), the number of rows of \(A\) is \(\nicefrac{1}{2}\times n_0\), and the matrix \(M\) is nonsymmetric and nonmonotone (\(M+M^\top\) is not positive semidefinite) with a condition number \(2\times10^4\).
			The optimality conditions of~\eqref{QPCC_general} follow from \cref{definition_M_C}, in which the primal-dual feasibility and complementary slackness are
			\begin{align*}
			&
			{(Ax_0+a)}_i\leq 0,\
			\alpha_i\geq0\ \text{and}\
			\alpha_i{(Ax_0+a)}_i=0,\
			i\in\mathcal{I}
			\\
			&
			x_1\geq0,\ Nx_0+Mx_1+q\geq0,\
			y_i^G x_1=0\ \text{and}\
			y_i^H (Nx_0+Mx_1+q)=0,\
			i\in\mathcal{P}.
			\end{align*}
			By following~\cite[Definition 1]{Jiang_1999}, an index \(i\in\mathcal{I}\) is called \emph{first-level degeneracy} if \(\alpha_i={(Ax_0+a)}_i=0\),
			an index \(i\in\mathcal{P}\) is called \emph{second-level degeneracy} if \(i\in I_{00}\) and an index \(i\in I_{00}\) is called \emph{mixed degeneracy} if either \(y_i^G=0\) or \(y_i^H=0\),
			in which the set \(I_{00}\) is defined in~\eqref{index_set_I00}, and \(\mathcal{I},\mathcal{P}\) are index sets of inequality and complementarity constraints, respectively.
			We first test problems with \emph{zero degeneracy}, i.e., the cardinalities of the first-level, second-level, and mixed degeneracy index sets are set to 0.
			Subsequently, to increase the problem-solving difficulty, we generate problems with \emph{high degeneracy}.
			Specifically, we choose the index set cardinality of the first-level degeneracy as \(\nicefrac{3}{5}\) of the number of rows of \(A\), the second-level degeneracy as \(\nicefrac{3}{4}\times p\) and the mixed degeneracy as \(\nicefrac{2}{3}\) of the second-level degeneracy.
			Furthermore, for each problem size, 10 random problems are generated and tested, with the starting point of each instance randomly selected from the interval \([-50, 50]\).

			\begin{table}[htb]
				\centering
				\caption{Objective, constraint violation and computation time on linearly constrained quadratic MPCCs}
				\resizebox{\columnwidth}{!}{%
				\begin{tabular}{|l|ccc|ccc|ccc|ccc|}
				\hline
				\(n_0\)--\(p\)--degen.--num. & LL1 & & & LL2(1st) & & & ALGENCAN & & & IPOPT & & \\
				& obj. & vio. & time\@(s) & obj. & vio. & time\@(s) & obj. & vio. & time\@(s) & obj. & vio. & time\@(s) \\
				\hline
				100--400--zero--0 & \textbf{-221887.01} & 2.5e-10 & 97 & \textbf{0\%} & 3.0e-10 & \textbf{\underline{93}} & \(<\)0.1\% & 2.4e-04 & 712 & \(>\)100\% & 3.1e-03 & 525  \\
				100--400--zero--1 & \textbf{-100443.01} & 5.5e-11 & \textbf{\underline{9.21}} & \textbf{0\%} & 6.7e-11 & 10.65 & \(<\)0.1\% & 1.1e-04 & 631 & 1.39\% & 2.3e-05 & 457 \\
				100--400--zero--2 & \textbf{-127464.67} & 7.7e-11 & \textbf{\underline{30.42}} & \textbf{0\%} & 6.3e-11 & 30.49 & \(<\)0.1\% & 9.3e-05 & 776 & \(>\)100\% & 1.1e-05 & 232 \\
				100--400--zero--3 & \textbf{-78290.04} & 4.5e-11 & \textbf{\underline{2.25}} & \textbf{0\%} & 5.2e-11 & 3.61 & 57.99\% & 8.3e-09 & 672 & 2.12\% & 2.2e-04 & 295 \\
				100--400--zero--4 & \textbf{-111968.19} & 1.1e-10 & 35.18 & \textbf{0\%} & 9.0e-11 & \textbf{\underline{34.50}} & \(<\)0.1\% & 2.3e-03 & 340 & \(>\)100\% & 1.7e-04 & 170 \\
				100--400--zero--5 & \textbf{-75648.18} & 1.8e-10 & \textbf{\underline{28.96}} & \textbf{0\%} & 3.4e-11 & 29.17 & \(<\)0.1\% & 3.6e-04 & 322 & \(>\)100\% & 1.4e-03 & 130 \\
				100--400--zero--6 & \textbf{-152564.60} & 5.6e-11 & \textbf{\underline{13.58}} & \textbf{0\%} & 1.8e-10 & 16.67 & 1.05\% & 8.5e-10 & 1077 & 32.86\% & 1.1e-05 & 69 \\
				100--400--zero--7 & \textbf{-38928.21} & 5.8e-11 & \textbf{\underline{35.00}} & \textbf{0\%} & 5.8e-11 & 39.22 & \(<\)0.1\% & 9.5e-05 & 667 & 0.15\% & 2.7e-03 & 324 \\
				100--400--zero--8 & \textbf{-24295.04} & 1.5e-09 & \textbf{\underline{5.10}} & \textbf{0\%} & 1.1e-08 & 6.49 & \(<\)0.1\% & 1.8e-04 & 420 & \(>\)100\% & 2.6e-05 & 584 \\
				100--400--zero--9 & \textbf{-153419.83} & 7.0e-10 & 175 & \(<\)0.1\% & 7.4e-10 & 198 & \(<\)0.1\% & 4.6e-04 & 930 & \(>\)100\% & 8.3e-03 & \textbf{\underline{163}} \\
				\hline
				100--400--high--0 & \textbf{-68683.11} & 3.3e-10 & 196 & \textbf{0\%} & 2.9e-10 & \textbf{\underline{192}} & \(<\)0.1\% & 2.1e-03 & 4073 & \(>\)100\% & 9.8e-04 & 265 \\
				100--400--high--1 & \textbf{-55672.63} & 1.1e-09 & \textbf{\underline{17.19}} & \textbf{0\%} & 4.9e-11 & 19.14 & \textbf{0\%} & 6.8e-06 & 2545 & 0.38\% & 9.9e-04 & 533 \\
				100--400--high--2 & \textbf{-53030.90} & 6.1e-11 & \textbf{\underline{78}} & \textbf{0\%} & 7.5e-11 & 93 & \(<\)0.1\% & 5.0e-05 & 4092 & 0.24\% & 3.2e-05 & 681 \\
				100--400--high--3 & \textbf{-32820.07} & 1.5e-09 & \textbf{\underline{4.01}} & \textbf{0\%} & 1.2e-09 & 5.53 & \textbf{0\%} & 9.2e-05 & 2454 & 42.55\% & 8.2e-06 & 1047 \\
				100--400--high--4 & \textbf{-42400.77} & 1.1e-10 & \textbf{\underline{52}} & \textbf{0\%} & 6.2e-11 & 57 & \(<\)0.1\% & 1.0e-02 & 3376 & \(>\)100\% & 1.9e-05 & 293 \\
				100--400--high--5 & \textbf{-25807.63} & 1.1e-10 & 118 & \textbf{0\%} & 1.0e-10 & \textbf{\underline{98}} & \(<\)0.1\% & 2.8e-03 & 3982 & \(>\)100\% & 5.3e-04 & 102 \\
				100--400--high--6 & 3.08\% & 2.5e-03 & 292 & 0.41\% & 1.0e-03 & \textbf{\underline{285}} & \textbf{-2729.96} & 5.9e-03 & 3938 & \(>\)100\% & 5.4e-05 & 443 \\
				100--400--high--7 & \textbf{-59925.08} & 1.1e-10 & \textbf{\underline{99}} & \textbf{0\%} & 2.8e-10 & 136 & \(<\)0.1\% & 3.2e-03 & 3739 & 0.39\% & 1.1e-03 & 447 \\
				100--400--high--8 & \textbf{-45550.65} & 3.8e-07 & 172 & \textbf{0\%} & 5.2e-08 & \textbf{\underline{156}} & \(<\)0.1\% & 1.2e-03 & 3674 & 0.15\% & 1.0e-03 & 418 \\
				100--400--high--9 & \textbf{-52139.38} & 6.3e-07 & 184 & \textbf{0\%} & 6.0e-07 & \textbf{\underline{175}} & \(<\)0.1\% & 7.1e-03 & 2945 & \(>\)100\% & 6.2e-03 & 505 \\
				\hline
				200--800--zero--0 & \textbf{-25562.79} & 5.3e-11 & \textbf{\underline{655}} & \textbf{0\%} & 6.0e-11 & 691 & \(<\)0.1\% & 1.7e-03 & 12803 & \(>\)100\% & 1.2e-03 & 3328 \\
				200--800--zero--1 & \textbf{-232452.69} & 4.8e-11 & \textbf{\underline{151}} & \textbf{0\%} & 3.6e-11 & 162 & \(<\)0.1\% & 3.5e-03 & 2764 & \(>\)100\% & 2.5e-04 & 2007 \\
				200--800--zero--2 & \textbf{-323667.78} & 3.3e-11 & \textbf{\underline{11.82}} & \textbf{0\%} & 7.4e-11 & 18.80 & \textbf{0\%} & 3.6e-04 & 1984 & 20.01\% & 3.7e-04 & 1890 \\
				200--800--zero--3 & \textbf{-244977.07} & 1.2e-09 & \textbf{\underline{682}} & \(<\)0.1\% & 6.3e-10 & 753 & \(<\)0.1\% & 3.4e-03 & 2606 & \(>\)100\% & 1.1e-04 & 2580 \\
				200--800--zero--4 & \textbf{-128268.60} & 3.4e-11 & \textbf{\underline{13.72}} & \textbf{0\%} & 5.1e-11 & 21.06 & \textbf{0\%} & 7.1e-05 & 2486 & 11.97\% & 1.0e-04 & 2184 \\
				200--800--zero--5 & \textbf{-117778.88} & 3.5e-11 & \textbf{\underline{38.50}} & \textbf{0\%} & 9.6e-12 & 45.72 & \textbf{0\%} & 2.4e-03 & 4054 & 57.22\% & 5.1e-05 & 4305 \\
				200--800--zero--6 & \textbf{-234265.58} & 1.6e-10 & \textbf{\underline{224}} & \(<\)0.1\% & 2.2e-10 & 230 & \(<\)0.1\% & 3.1e-03 & 2247 & \(>\)100\% & 7.3e-05 & 1935 \\
				200--800--zero--7 & \textbf{-28086.92} & 7.4e-12 & \textbf{\underline{23.52}} & \textbf{0\%} & 3.3e-11 & 30.75 & \(<\)0.1\% & 3.0e-04 & 1708 & 14.89\% & 1.7e-04 & 7822 \\
				200--800--zero--8 & \textbf{-292242.58} & 1.0e-11 & \textbf{\underline{72}} & \textbf{0\%} & 1.8e-11 & 79 & \(<\)0.1\% & 2.1e-03 & 6099 & \(>\)100\% & 1.6e-04 & 6265 \\
				200--800--zero--9 & \textbf{-668174.30} & 1.0e-10 & \textbf{\underline{62}} & \textbf{0\%} & 5.2e-11 & 70 & \(<\)0.1\% & 1.3e-03 & 2149 & \(>\)100\% & 5.3e-05 & 5044 \\
				\hline
				200--800--high--0 & \textbf{-11090.41} & 6.7e-11 & 382 & \textbf{0\%} & 8.4e-11 & \textbf{\underline{358}} & \(<\)0.1\% & 1.2e-02 & 18047 & 0.36\% & 6.6e-05 & 8932 \\
				200--800--high--1 & \textbf{-65477.02} & 8.2e-07 & \textbf{\underline{858}} & \textbf{0\%} & 2.1e-07 & 907 & \(<\)0.1\% & 6.0e-03 & 16212 & 0.16\% & 2.0e-03 & 10552 \\
				200--800--high--2 & \textbf{-122853.43} & 2.2e-10 & \textbf{\underline{25.50}} & \textbf{0\%} & 3.5e-09 & 29.68 & \textbf{0\%} & 1.0e-04 & 10190 & 38.39\% & 7.8e-05 & 8718 \\
				200--800--high--3 & \textbf{-81872.19} & 1.6e-07 & \textbf{\underline{952}} & \textbf{0\%} & 1.2e-06 & 961 & \(<\)0.1\% & 1.8e-02 & 18445 & \(>\)100\% & 1.4e-02 & 9795 \\
				200--800--high--4 & \textbf{-59376.49} & 2.3e-09 & \textbf{\underline{32.30}} & \textbf{0\%} & 2.7e-09 & 37.18 & \textbf{0\%} & 1.5e-05 & 12224 & 79.89\% & 3.8e-05 & 12542 \\
				200--800--high--5 & \textbf{-46166.97} & 2.7e-11 & 39.90 & \textbf{0\%} & 9.4e-10 & \textbf{\underline{39.42}} & \(<\)0.1\% & 1.4e-03 & 22028 & \(>\)100\% & 1.0e-04 & 7088 \\
				200--800--high--6 & \textbf{-76357.26} & 8.8e-11 & 686 & \textbf{0\%} & 7.0e-11 & \textbf{\underline{592}} & \(<\)0.1\% & 8.3e-03 & 21230 & \(>\)100\% & 8.4e-05 & 2966 \\
				200--800--high--7 & \(<\)0.1\% & 8.2e-07 & \textbf{\underline{1253}} & \(<\)0.1\% & 2.9e-06 & 1262 & \textbf{-14294.96} & 7.1e-03 & 16925 & 0.30\% & 8.0e-04 & 6985 \\
				200--800--high--8 & \textbf{-102912.14} & 7.3e-10 & \textbf{\underline{75}} & \textbf{0\%} & 4.1e-11 & 88 & \(<\)0.1\% & 3.1e-03 & 18732 & \(>\)100\% & 6.3e-05 & 4821 \\
				200--800--high--9 & \textbf{-241410.59} & 5.7e-11 & \textbf{\underline{123}} & \textbf{0\%} & 3.5e-11 & 124 & \(<\)0.1\% & 1.9e-03 & 15742 & \(>\)100\% & 4.5e-05 & 5331 \\
				\hline
				\end{tabular}
				}\label{table_QPCC_general}
			\end{table}

			The results are presented in \cref{table_QPCC_general}.
			We omit the Scholtes method from comparison, as it requires even more computational time than ALGENCAN{}.
			In most instances, LL-based methods find the best objective values, which are consistent with the global minima provided by the problem generator, and they are also faster than NLP solvers.
			Therefore, it is expected that LL penalty offers limited improvement in objective values, while incurring higher computational costs due to the use of an NLP solver as the inner solver.
			For these reasons, LL penalty is also omitted in \cref{table_QPCC_general}.
			Although ALGENCAN also returns high-quality solutions, it performs significantly slower, especially compared to the fast computation observed in the previous subsection.
			This indicates inefficiency in handling complex CCs and infeasible starting points.
			While IPOPT is faster than ALGENCAN, it returns lower-quality solutions, often yielding highly suboptimal solutions (\(>\)100\% deviation).
			On the other hand, the computation times of ALGENCAN and IPOPT increase significantly when the degeneracy level is raised from \emph{zero} to \emph{high}.
			In contrast, LL-based methods demonstrate robust performance regardless of degeneracy.
			It is also worth noting that LL-based methods violate the CCs minimally compared to ALGENCAN and IPOPT\@.
			Therefore, we conclude that LL-based methods are efficient for solving large-scale, linearly constrained quadratic MPCCs, particularly in the presence of infeasible starting points and high degeneracy, and are well suited for problems that demand high satisfaction of the CCs.

	\section{Conclusion}
		We have introduced a new homotopy-based approach for solving MPCCs.
		The indicator function associated with the CC is nonsmooth and nonconvex, and is approximated using the Lasry--Lions double envelope to obtain a smooth reformulation.
		The original problem is then approximated by solving a sequence of smoothed subproblems with increasing penalization.
		When the proposed approach, \cref{algorithm1}, is initialized from a feasible point, we show that it converges to a C/M-stationary point in the limiting setting.
		To obtain an approximate stationary point within a finite number of steps, we further propose a practical and easily implementable variant, \cref{algorithm3}.
		It is shown that an \(\epsilon\)-approximate C/M-stationary point is generated within \(\mathcal{O}(\log(\epsilon^{-2}))\) outer iterations, and \(\mathcal{O}(\epsilon^{-4}\log(\epsilon^{-2}))\) total iterations when a first-order method is used as the inner solver.
		In \cref{algorithm3}, a smooth subproblem with simple constraints is used to approximate the general constrained MPCCs, and thus fast optimization algorithms can be readily deployed.
		The effectiveness of the proposed approach is validated through a comprehensive set of benchmark problems against the classical regularization method and state-of-the-art NLP solvers.
		As for future work, the following three directions will be pursued:
		(i) To make \BCCQ\@ usable in practice, verifiable conditions for \BCCQ\@ will be studied.
		(ii) As mentioned in \cref{remark_BCCQ}, the implications among MPCC-CRCQ, MPCC-CPLD, and \BCCQ\@ will be established.
		(iii) As observed in \cref{table_QPCC,table_QPCC_general}, the computational time increases with the problem dimension, the number of CCs and the level of degeneracy.
		Therefore, more efficient algorithms, such as those incorporating preconditioning in the inner solver, are worth developing.

	\appendix
	\section{\texorpdfstring{{Proof of \cref{theorem_LL_env}}}{Proof of Theorem \ref{theorem_LL_env}}}\label{appendix_theorem_LL_env}
		The complementarity set \(D\), defined in~\eqref{CC_set}, can be equivalently written as
		\begin{equation}
			D=D_1\cup D_2
			\quad\text{with}\quad
			D_1 \coloneqq \set{0}\times\R_+
			\quad\text{and}\quad
			D_2 \coloneqq \R_+\times\set{0},\label{disjunctive_set}
		\end{equation}
		which corresponds to the so-called \emph{disjunctive constraint}; see, for
		instance,~\cite[Section 5.1]{Jia_2023}.
		Then, the Moreau envelope in~\eqref{moreau_compl} admits the following representation:
		\begin{equation}\label{moreau_compl_disjunctive}
		\env_\lambda\delta_D(z)
		=
			\tfrac{1}{2\lambda}\dist {(z,D)}^2
		=
		\tfrac{1}{2\lambda}
		\min
		\set{
			\dist{(z,D_1)}^2,\dist{(z,D_2)}^2
		}.
		\end{equation}
		Based on the above notation, we derive the Lasry--Lions double envelope following \cref{definition_LL}.
		Specifically, we have
		\begin{subequations}
		\allowdisplaybreaks\@
		\begin{align}
		&
			\env_{\lambda,\mu}\delta_{D}(z)
		\nonumber
		\\
		\nonumber
		\stackrel{\text{\clap{\eqref{sup_LL_definition}}}}{=}{} &
			\sup_{z'\in\R^2}\{\env_\lambda\delta_{D}(z')-\tfrac{1}{2\mu}\|z'-z\|^2 \}
		\\
		\nonumber
		\stackrel{\text{\clap{\eqref{moreau_compl_disjunctive}}}}{=}{} &
			\tfrac{1}{\lambda}
			\sup_{z'\in\R^2}
			\left\{
			\tfrac{1}{2}
			\min
			\set{
				\dist{(z',D_1)}^2,\dist{(z',D_2)}^2
			}
			-
			\tfrac{1}{2\beta}\|z'-z\|^2
			\right\}
		\\
		\nonumber
		={} &
			-\tfrac{1}{\lambda}
			\inf_{z'\in\R^2}
			\left\{
			\tfrac{1}{2\beta}\|z'-z\|^2
			-
			\tfrac{1}{2}
			\min
			\set{
				\dist{(z',D_1)}^2,\dist{(z',D_2)}^2
			}
			\right\}
		\\
		\nonumber
		={} &
			-\tfrac{1}{\lambda}
			\inf_{z'\in\R^2}
			\left\{
			\tfrac{1}{2\beta}\|z'-z\|^2
			+
			\max
			\set{
			-\tfrac{1}{2}\dist{(z',D_1)}^2,
			-\tfrac{1}{2}\dist{(z',D_2)}^2
			}
			\right\}
		\\
		={} &
			-\tfrac{1}{\lambda}
			\inf_{z'\in\R^2}
			\left\{
			\max
			\set{
			\tfrac{1}{2\beta}\|z'-z\|^2-\tfrac{1}{2}\dist{(z',D_1)}^2,
			\tfrac{1}{2\beta}\|z'-z\|^2-\tfrac{1}{2}\dist{(z',D_2)}^2
			}
			\right\}\label{either_or_max}
		\\
		\nonumber
		={} &
			-\tfrac{1}{\lambda}
			\inf_{z'\in\R^2}
			\biggl\{
				\max\limits_{y\in\Delta_2}
				\Bigl[
					y_1
					\begin{pmatrix}
					\tfrac{1}{2\beta}\|z'-z\|^2 -\tfrac{1}{2}\dist{(z',D_1)}^2
					\end{pmatrix}
		\\
		&
			\hphantom{
				-\tfrac{1}{\lambda}
				\inf_{z'\in\R^2}
				\biggl\{
				\max\limits_{y\in\Delta_2}
				\Bigl[{}
			}
					+
					y_2
					\begin{pmatrix}
					\tfrac{1}{2\beta}\|z'-z\|^2 -\tfrac{1}{2}\dist{(z',D_2)}^2
					\end{pmatrix}
				\Bigr]
			\biggr\},
			\label{p_1_simplex}
		\end{align}
		\end{subequations}
		where \(y\coloneqq(y_1,y_2)\in\R^2\) and \(\Delta_2\coloneqq\set{y\in{\{0,1\}}^2}[y_1+y_2=1]\).
		Hence, the either-or structure of the maximization in~\eqref{either_or_max} is equivalently reformulated as a maximization over the discrete set \(\Delta_2\) in~\eqref{p_1_simplex}.
		Note that the inner maximization problem in~\eqref{p_1_simplex} corresponds to the \emph{support function} of the set \(\Delta_2\); see~\cite[8(10)]{Rockafellar_1998}.
		Thus, in light of~\cite[Example 11.2]{Bauschke_2011}, it follows that~\eqref{p_1_simplex} admits an equivalent formulation:
		\begin{equation}\label{cl_conv_p_1_simplex}
			\env_{\lambda,\mu}\delta_{D}(z)
			=
			-\tfrac{1}{\lambda}
			\inf_{z'\in\R^2}
			\left\{
			\max\limits_{y\in\cl\conv\Delta_2}
			\left[
				\begin{array}{l}
				y_1
				\begin{pmatrix}
				\tfrac{1}{2\beta}\|z'-z\|^2 -\tfrac{1}{2}\dist{(z',D_1)}^2
				\end{pmatrix}
				\\
				+
				y_2
				\begin{pmatrix}
				\tfrac{1}{2\beta}\|z'-z\|^2 -\tfrac{1}{2}\dist{(z',D_2)}^2
				\end{pmatrix}
			\end{array}
			\right]
			\right\},
		\end{equation}
		where \(\cl\conv\Delta_2=\set{y\in\R_+^2}[y_1+y_2=1]\subset\R^2\).
		Note that the distance functions appearing in~\eqref{cl_conv_p_1_simplex} can be expressed explicitly.
		In particular, we recall that the sets \(D_1\) and \(D_2\) are defined in~\eqref{disjunctive_set}, the associated projections \(\Pi_{D_1}(z)\) and \(\Pi_{D_2}(z)\) admit
		\[
			\Pi_{D_1}(z)
			=
			\begin{pmatrix}
			0
			\\
			{[z_2]}_+
			\end{pmatrix}
			\quad\text{and}\quad
			\Pi_{D_2}(z)
			=
			\begin{pmatrix}
			{[z_1]}_+
			\\
			0
			\end{pmatrix}.
		\]
		Therefore, the distance function associated with \(D_1\) can be written as
		\[
		\begin{aligned}
			\tfrac{1}{2}\dist{(z',D_1)}^2
			=
			\tfrac{1}{2}\|z'-\Pi_{D_1}(z')\|^2
			&=
			\tfrac{1}{2}\|z'\|^2-\langle z',\Pi_{D_1}(z')\rangle+\tfrac{1}{2}\|\Pi_{D_1}(z')\|^2
			\\
			&=
			\tfrac{1}{2}\|z'\|^2-z'_2 {[z'_2]}_{+}+\tfrac{1}{2}{[z'_2]}_{+}^2
			\\
			&=
			\tfrac{1}{2}\|z'\|^2-{[z'_2]}_{+}^2+\tfrac{1}{2}{[z'_2]}_{+}^2
			\\
			&=
			\tfrac{1}{2}\|z'\|^2-\tfrac{1}{2}
			{[z'_2]}_{+}^2
		\end{aligned}
		\]
		and, analogously, the distance function associated with \(D_2\) is given by
		\(
			\tfrac{1}{2}\|z'\|^2-\tfrac{1}{2}
			{[z'_1]}_{+}^2
		\).
		Substituting these expressions into~\eqref{cl_conv_p_1_simplex} yields
		\begin{align}
			&\env_{\lambda,\mu}\delta_{D}(z)
			\nonumber\\
			&=
			-\tfrac{1}{\lambda}
			\inf_{z'\in\R^2}
			\left\{
			\max\limits_{y\in\cl\conv\Delta_2}
			\left[
				\begin{array}{l}
				y_1
				\begin{pmatrix}
					\tfrac{1}{2\beta}\|z'-z\|^2
					-
					\tfrac{1}{2}\|z'\|^2
					+
					\tfrac{1}{2}{[z'_2]}_{+}^2
				\end{pmatrix}
				\\
				+y_2
				\begin{pmatrix}
					\tfrac{1}{2\beta}\|z'-z\|^2
					-
					\tfrac{1}{2}\|z'\|^2
					+
					\tfrac{1}{2}{[z'_1]}_{+}^2
				\end{pmatrix}
				\end{array}
			\right]
			\right\}
		\nonumber\\
		&=
			-\tfrac{1}{\lambda}
			\max\limits_{y\in\cl\conv\Delta_2}
			\left\{
			\inf_{z'\in\R^2}
			\left[
				\begin{array}{l}
				y_1
				\begin{pmatrix}
					\tfrac{1}{2\beta}\|z'-z\|^2
					-
					\tfrac{1}{2}\|z'\|^2
					+
					\tfrac{1}{2}{[z'_2]}_{+}^2
				\end{pmatrix}
				\\
				+y_2
				\begin{pmatrix}
					\tfrac{1}{2\beta}\|z'-z\|^2
					-
					\tfrac{1}{2}\|z'\|^2
					+
					\tfrac{1}{2}{[z'_1]}_{+}^2
				\end{pmatrix}
				\end{array}
			\right]
			\right\}
		\nonumber\\
		&=
			-\tfrac{1}{\lambda}
			\max\limits_{y\in\cl\conv\Delta_2}
			\left\{
			\inf_{z'\in\R^2}
			\begin{pmatrix}
				\tfrac{1}{2\beta}\|z'-z\|^2-
				\tfrac{1}{2}\|z'\|^2
			\\
				+
				\tfrac{1}{2}{[z'_1]}_+^2+\tfrac{1}{2}{[z'_2]}_+^2-
				\tfrac{1}{2}y_1{[z'_1]}_+^2-\tfrac{1}{2}y_2{[z'_2]}_+^2
			\end{pmatrix}
			\right\}
			\nonumber\\
			&=
			-\tfrac{1}{\lambda}
			\max\limits_{y\in\cl\conv\Delta_2}
			\left\{
			\inf_{z'\in\R^2}
			\begin{pmatrix}
				\tfrac{1}{2\beta}\sum_{i=1}^{2}{(z'_i-z_i)}^2-
				\tfrac{1}{2}\sum_{i=1}^{2}{(z'_i)}^2
			\\
				+
				\tfrac{1}{2}\sum_{i=1}^{2}{[z'_i]}_+^2-
				\tfrac{1}{2}\sum_{i=1}^{2}y_i{[z'_i]}_+^2
			\end{pmatrix}
			\right\}
			\nonumber\\
			&=
			-\tfrac{1}{\lambda}
			\max\limits_{y\in\cl\conv\Delta_2}
			\left\{
			\sum_{i=1}^{2}
			\inf_{z'_i\in\R}
			\begin{pmatrix}
				\tfrac{1}{2\beta}{(z'_i-z_i)}^2-
				\tfrac{1}{2}{(z'_i)}^2
				+
				\tfrac{1}{2}{[z'_i]}_+^2-
				\tfrac{1}{2}y_i{[z'_i]}_+^2
			\end{pmatrix}
			\right\},\label{pointwise_min}
		\end{align}
		in which the second equality follows from the \emph{minimax theorem}, since  the inner function is linear in \(y\) and convex in \(z'\), owing to the fact that \(\beta\in(0,1)\).
		The third equality follows from the constraint \(y_1+y_2=1\).
		Consequently, we obtain two one-dimensional inner convex minimization problems in~\eqref{pointwise_min}, each of which admits the minimizer, together with the corresponding function value, given by
		\begin{subequations}\label{inner_minimizer_z}
			\begin{equation}
				{z'_i}^\star
			=
				\begin{cases}
				\tfrac{1}{1-\beta}z_i,\ \ &\mbox{if } z_i\leq0
				\\[5pt]
				\tfrac{1}{1-\beta y_i}z_i,\ \ &\mbox{if } z_i\geq0
				\end{cases}
			\end{equation}
			and
			\begin{equation}
					\tfrac{1}{2\beta}{({z'_i}^\star-z_i)}^2
					-
					\tfrac{1}{2}{({z'_i}^\star)}^2
					+
					\tfrac{1}{2}{[{z'_i}^\star]}_+^2
					-
					\tfrac{1}{2}y_i{[{z'_i}^\star]}_+^2
				=
				\begin{cases}
				-\tfrac{1}{2(1-\beta)}z_i^2,\ \ &\mbox{if } z_i\leq0
				\\[5pt]
				-\tfrac{y_i}{2(1-\beta y_i)}z_i^2,\ \ &\mbox{if } z_i\geq0.
				\end{cases}
			\end{equation}
		\end{subequations}
		Depending on the positions of \(z_1\) and \(z_2\), we consider the following four cases.
		\begin{enumerate}
			\item \(z_1\leq0\) and \(z_2\leq0\).
			This case corresponds to the region \(\mathrm{O}_{-}\) shown in \cref{fig:areas}, and~\eqref{pointwise_min} becomes
			\begin{align*}
				\env_{\lambda,\mu}\delta_{D}(z)
			\stackrel{\text{\clap{\eqref{inner_minimizer_z}}}}{=}{} &
				-\tfrac{1}{\lambda}
				\max\limits_{y\in\cl\conv\Delta_2}
				\left\{
					-\tfrac{1}{2(1-\beta)}z_1^2
					-\tfrac{1}{2(1-\beta)}z_2^2
				\right\}
			\\
			={} &
				\tfrac{1}{\lambda}
				\tfrac{1}{2(1-\beta)}\|z\|^2
				=
				\tfrac{1}{\lambda}
				r_{\beta}(z),
			\end{align*}
			as needed in~\eqref{LL_env} and~\eqref{r_beta}.
			\item\label{case:ii} \(z_1\leq0\) and \(z_2\geq0\).
			This case corresponds to the left rectangular region of \(\mathrm{H}_{\beta}^{+}\) as labeled in \cref{fig:areas}, and~\eqref{pointwise_min} can be formulated as
			\begin{align}
			\nonumber
				\env_{\lambda,\mu}\delta_{D}(z)
			\stackrel{\text{\clap{\eqref{inner_minimizer_z}}}}{=}{} &
				-\tfrac{1}{\lambda}
				\max\limits_{y\in\cl\conv\Delta_2}
				\left\{
					-\tfrac{1}{2(1-\beta)}z_1^2
					-\tfrac{y_2}{2(1-\beta y_2)}z_2^2
				\right\}
			\\
			\nonumber
			={} &
				\tfrac{1}{\lambda}
				\min\limits_{y\in\cl\conv\Delta_2}
				\left\{
					\tfrac{1}{2(1-\beta)}z_1^2
					+\tfrac{y_2}{2(1-\beta y_2)}z_2^2
				\right\}
			\\
			={} &
				\tfrac{1}{\lambda}
				\tfrac{1}{2(1-\beta)}z_1^2
				+
				\tfrac{1}{\lambda}
				\min\limits_{y\in\cl\conv\Delta_2}
				\left\{
					\tfrac{y_2}{2(1-\beta y_2)}z_2^2
				\right\}.\label{H_beta_positive_1}
			\end{align}
			Recalling that \(\cl\conv\Delta_2=\set{y\in\R_+^2}[y_1+y_2=1]\), the minimization problem in~\eqref{H_beta_positive_1} is equivalent to
			\[
			\minimize_{y\in\R_+^2}
			\
			\tfrac{y_2}{2(1-\beta y_2)}z_2^2
			\quad
			\stt\
			y_1+y_2=1.
			\]
			Let \(\tau\in\R\) be the Lagrange multiplier associated with the constraint \(y_1+y_2=1\).
			Then, the associated Lagrangian is given by
			\[
			\mathcal{L}(y,\tau)
			=
			\tfrac{y_2}{2(1-\beta y_2)}z_2^2
			+\tau
			\begin{pmatrix}
				1-y_1-y_2
			\end{pmatrix},
			\]
			and the dual function is obtained as
			\begin{equation}\label{H_beta_positive_1_dual_fun}
			\psi(\tau)\coloneqq \inf_{y\in\R_+^2}\mathcal{L}(y,\tau)=
			\tau+
			\inf_{y_1\in\R_+}
			\left\{
				-\tau y_1
			\right\}
			+
			\inf_{y_2\in\R_+}
			\left\{
				\tfrac{y_2}{2(1-\beta y_2)}z_2^2-\tau y_2
			\right\}.
			\end{equation}
			The first minimization problem in~\eqref{H_beta_positive_1_dual_fun} implies that \(\tau\leq0\); otherwise \(\psi(\tau)=-\infty\).
			Thus, its minimizer is \(y_1(\tau)=0\).
			Considering the second minimization problem in~\eqref{H_beta_positive_1_dual_fun}, the gradient with respect to \(y_2\) is
			\(
			\tfrac{z_2^2}{2{(1-\beta y_2)}^2}-\tau
			\),
			which is nonnegative, owing to the facts that \(z_2\geq0\) and \(\tau\leq0\).
			Hence, the minimum is attained at \(y_2(\tau)=0\).
			Consequently, the minimizer in the dual function~\eqref{H_beta_positive_1_dual_fun} is \(y(\tau)=0\), and the dual problem admits
			\[
			\max_{\tau\in\R_-}\psi(\tau)=\max_{\tau\in\R_-}\tau=0.
			\]
			It is straightforward to verify that the minimization problem in~\eqref{H_beta_positive_1} is convex, and thus \emph{strong duality} holds.
			The Lasry--Lions double envelope is then given by
			\[
			\env_{\lambda,\mu}\delta_{D}(z)
			\overset{\eqref{H_beta_positive_1}}{=}
			\tfrac{1}{\lambda}
			\tfrac{1}{2(1-\beta)}z_1^2
			=
			\tfrac{1}{\lambda}
			r_{\beta}(z),
			\]
			which yields the expressions in~\eqref{LL_env} and~\eqref{r_beta}.
			\item \(z_1\geq0\) and \(z_2\leq0\).
			This case corresponds to the lower rectangular region of \(\mathrm{H}_{\beta}^{-}\) shown in \cref{fig:areas}.
			Similar to the previous case, we have
			\[
			\env_{\lambda,\mu}\delta_{D}(z)
			=
			\tfrac{1}{\lambda}
			\tfrac{1}{2(1-\beta)}z_2^2
			=
			\tfrac{1}{\lambda}
			r_{\beta}(z).
			\]
		
			\item \(z_1>0\) and \(z_2>0\).
			This case corresponds to the region \(\mathrm{T}_{\beta}\), the right triangular region of \(\mathrm{H}_{\beta}^{+}\) and the upper triangular region of \(\mathrm{H}_{\beta}^{-}\), as labeled in \cref{fig:areas}.
			Proceeding as in the case~\ref{case:ii}, we have
			\begin{equation}\label{inner_minimizer_z_iv}
			\env_{\lambda,\mu}\delta_{D}(z)
			\overset{\eqref{pointwise_min},\eqref{inner_minimizer_z}}{=}
			\tfrac{1}{\lambda}
			\min\limits_{y\in\cl\conv\Delta_2}
			\left\{
				\tfrac{y_1}{2(1-\beta y_1)}z_1^2
				+\tfrac{y_2}{2(1-\beta y_2)}z_2^2
			\right\}.
			\end{equation}
			Let \(\tau\in\R\) be the Lagrange multiplier, the associated Lagrangian and the dual function are
			\begin{equation}\label{dual_function_iv}
			\mathcal{L}(y,\tau)
			=
			\tau+
			\sum_{i=1}^{2}
			\begin{pmatrix}
				\tfrac{y_iz_i^2}{2(1-\beta y_i)}-\tau y_i
			\end{pmatrix}
			\quad\text{and}\quad
			\psi(\tau)=
			\tau+
			\sum_{i=1}^{2}
			\inf_{y_i\in\R_+}
			\left\{
				\tfrac{y_iz_i^2}{2(1-\beta y_i)}-\tau y_i
			\right\}.
			\end{equation}
			The minimizers of the inner minimization problems in \(\psi(\tau)\) are given by
			\[
			y_i(\tau)=
			\begin{cases}
				\tfrac{1}{\beta}-\tfrac{z_i}{\beta\sqrt{2\tau}},\ \ &\mbox{if } \tau\geq\tfrac{z_i^2}{2}\\[5pt]
				0,\ \ &\mbox{if } \tau\leq\tfrac{z_i^2}{2}.
			\end{cases}
			\]
			In this way, we obtain the function \(\psi_i(\tau)\) and its gradient with respect to \(\tau\) as follows:
			\begin{subequations}\label{dual_function_grad_iv}
				\begin{equation}
					\psi_i(\tau)\coloneqq
					\inf_{y_i\in\R_+}
					\left\{
						\tfrac{y_i}{2(1-\beta y_i)}z_i^2-\tau y_i
					\right\}
				=
					\begin{cases}
						-\tfrac{{(\sqrt{2\tau}-z_i)}^2}{2\beta},\ \ &\mbox{if } \tau\geq\tfrac{z_i^2}{2}
						\\[5pt]
						0,\ \ &\mbox{if } \tau\leq\tfrac{z_i^2}{2}
					\end{cases}
				\end{equation}
				and
				\begin{equation}
					\nabla\psi_i(\tau)
				=
					\begin{cases}
						-\tfrac{\sqrt{2\tau}-z_i}{\beta\sqrt{2\tau}},\ \ &\mbox{if } \tau\geq\tfrac{z_i^2}{2}\\[5pt]
						0,\ \ &\mbox{if } \tau\leq\tfrac{z_i^2}{2}.
					\end{cases}
				\end{equation}
			\end{subequations}
			To solve the dual problem \(\maximize_{\tau}\psi(\tau)\), we consider the following three cases, depending on the positions of \(z_1\) and \(z_2\).
			\begin{enumerate}[label = (\roman{enumi}.\arabic*)]
				\item \(\tau\geq\tfrac{z_1^2}{2}\) and \(\tau\leq\tfrac{z_2^2}{2}\).
				In this case, the gradient of the dual function admits
				\[
				\nabla\psi(\tau)
				=
				1+\nabla\psi_1(\tau)+\nabla\psi_2(\tau)
				=
				1
				-\tfrac{\sqrt{2\tau}-z_1}{\beta\sqrt{2\tau}},
				\]
				and setting \(\nabla\psi(\tau)=0\) yields
				\(
				\tau^\star =\tfrac{1}{2}
				\begin{pmatrix}
					\frac{z_1}{1-\beta}
				\end{pmatrix}^2
				\).
				This implies that \(z_1>0\) and \(\tfrac{z_1}{1-\beta}\leq z_2\),
				which characterizes the right triangular region of \(\mathrm{H}_{\beta}^{+}\).
				By invoking strong duality once again, the Lasry--Lions double envelope is given by
				\[
				\begin{aligned}
				\env_{\lambda,\mu}\delta_{D}(z)
				\overset{\eqref{inner_minimizer_z_iv}}{=}
				\tfrac{1}{\lambda}\max_{\tau}\psi(\tau)
				&\overset{\eqref{dual_function_iv},\eqref{dual_function_grad_iv}}{=}
				\tfrac{1}{\lambda}
				\begin{pmatrix}
					\tau^\star-\tfrac{{(\sqrt{2\tau^\star}-z_1)}^2}{2\beta}
				\end{pmatrix}
				\nonumber\\
				&=
				\tfrac{1}{\lambda}
				\tfrac{1}{2(1-\beta)}z_1^2
				=
				\tfrac{1}{\lambda}
				r_{\beta}(z),
				\end{aligned}
				\]
				as required in~\eqref{LL_env} and~\eqref{r_beta}.
		
				\item \(\tau\leq\tfrac{z_1^2}{2}\) and \(\tau\geq\tfrac{z_2^2}{2}\).
				Similar to the previous case, we obtain \(z_2>0\) and \(\tfrac{z_2}{1-\beta}\leq z_1\), which characterizes the upper triangular region of \(\mathrm{H}_{\beta}^{-}\).
				Moreover, the Lasry--Lions double envelope is expressed as
				\[
				\env_{\lambda,\mu}\delta_{D}(z)
				=
				\tfrac{1}{\lambda}
				\tfrac{1}{2(1-\beta)}z_2^2
				=
				\tfrac{1}{\lambda}
				r_{\beta}(z).
				\]
				\item \(\tau\geq\tfrac{z_1^2}{2}\) and \(\tau\geq\tfrac{z_2^2}{2}\).
				Finally, the gradient of the dual function in this case reads
				\[
				\nabla\psi(\tau)
				=
				1+\nabla\psi_1(\tau)+\nabla\psi_2(\tau)
				=
				1
				-
				\tfrac{\sqrt{2\tau}-z_1}{\beta\sqrt{2\tau}}
				-
				\tfrac{\sqrt{2\tau}-z_2}{\beta\sqrt{2\tau}}.
				\]
				Setting \(\nabla\psi(\tau)=0\) yields
				\(
				\tau^\star =\tfrac{1}{2}
				\begin{pmatrix}
					\frac{z_1+z_2}{2-\beta}
				\end{pmatrix}^2
				\),
				which implies that \((1-\beta)z_1\leq z_2\leq\tfrac{z_1}{1-\beta}\),
				thus characterizing the region \(\mathrm{T}_{\beta}\).
				The Lasry--Lions double envelope then becomes
				\begin{align*}
					\env_{\lambda,\mu}\delta_{D}(z)
				\stackrel{\text{\clap{\eqref{inner_minimizer_z_iv}}}}{=}{} &
					\tfrac{1}{\lambda}\max_{\tau}\psi(\tau)
				\\
					\text{\scriptsize by \eqref{dual_function_iv} and \eqref{dual_function_grad_iv}}
				={} &
					\tfrac{1}{\lambda}
					\begin{pmatrix}
						\tau^\star
						-\tfrac{{(\sqrt{2\tau^\star}-z_1)}^2}{2\beta}
						-\tfrac{{(\sqrt{2\tau^\star}-z_2)}^2}{2\beta}
					\end{pmatrix}
				\\
				={} &
					\tfrac{1}{\lambda}
					\begin{pmatrix}
						\tfrac{1}{2\beta(2-\beta)}
						{(z_1+z_2)}^2
						-
						\tfrac{1}{2\beta}
						\|z\|^2
					\end{pmatrix}
					=
					\tfrac{1}{\lambda}
					r_{\beta}(z),
				\end{align*}
				as in~\eqref{LL_env} and~\eqref{r_beta}.
			\end{enumerate}
		
		\end{enumerate}
		In this way, the formulation of the Lasry--Lions double envelope is verified.

	\phantomsection
	\addcontentsline{toc}{section}{References}
	\bibliographystyle{plain}
	\bibliography{MPCC_LL.bib}

\end{document}